\documentclass[11pt,oneside,reqno]{amsart}

\usepackage{graphicx}
\usepackage{pict2e}
\usepackage{amssymb}
\usepackage{amsthm}                % better theorem environments
\usepackage[margin=1in]{geometry}  % set the margins to 1in on all sides
\usepackage{graphics}
\usepackage{epstopdf}
\usepackage{amsmath}
\usepackage{graphicx}
\usepackage{subfigure}
\usepackage{latexsym}
\usepackage{amsmath}
\usepackage{amsfonts}
\usepackage{amssymb}
\usepackage{tikz}
\usepackage{mathdots}
\usepackage{mathtools}
\usepackage{hyperref}
\usepackage{comment}
\usepackage{mathtools}
\usepackage{float}
\usepackage[mathscr]{euscript}

\usepackage{epsfig}

\theoremstyle{definition}
\newtheorem{theorem}{Theorem}[section]
\newtheorem{lemma}[theorem]{Lemma}
\newtheorem{corollary}[theorem]{Corollary}
\newtheorem{proposition}[theorem]{Proposition}

\theoremstyle{definition}
\newtheorem{definition}[theorem]{Definition}

\theoremstyle{remark}
\newtheorem{remark}[theorem]{Remark}
\numberwithin{equation}{section}

\numberwithin{equation}{section}
\newcommand{\R}{\mathbf{R}}  % The real numbers.

\begin{document}
%\raggedbottom

\title{The Bubble skein element and applications}

\author{Mustafa Hajij}
\address{Department of Mathematics, Louisiana State University, 
Baton Rouge, LA 70803 USA}
\email{mhajij1@math.lsu.edu}
%%\urladdr{www.math.sc.edu/$\sim$howard} % Delete if not wanted.

%%
%% If there is another author uncomment and edit the following.
%%

%\author{Second Author}
%\address{Department of Mathematics, University of South Carolina,
%Columbia, SC 29208}
%\email{second@math.sc.edu}
%\urladdr{www.math.sc.edu/$\sim$second}

%%
%% If there are three of more authors they are added in the obvious
%% way. 
%%

%%%
%%% The following is for the abstract.  The abstract is optional and
%%% if not used just delete, or comment out, the following.
%%%

\begin{abstract}
We study a certain skein element in the relative Kauffman bracket skein module of the disk with some marked points, and expand this element in terms linearly independent elements of this module. This expansion is used to compute and study the head and the tail of the colored Jones polynomial and in particular we give a simple $q-$series for the tail of the knot $8_5$. Furthermore, we use this expansion to obtain an easy determination of the theta coefficients.  
\end{abstract}


 \maketitle

\section{Introduction}

In \cite{Cody2} Armond and Dasbach introduced the head and the tail of the colored Jones polynomial of an alternating knot $L$, two link invariants which take the form of infinite power series with integer coefficients. Skein theoretic techniques have been used in \cite{Cody} and \cite{Cody2} to understand the head and tail of an alternating link. Write $\mathcal{S}(S^{3})$ to denote the Kauffman Bracket Skein Module of $S^{3}$. Let $L$ be an alternating link and let $D$ be a reduced link diagram of $L$. Write $S_B(D)$ to denote the all-$B$ smoothing state of $D$, the state obtained by replacing each crossing by a $B$ smoothing. Let $S_B^{(n)}(D)$ be the skein element obtained from $S_B(D)$ by decorating each circle in this state with the $n^{th}$ Jones-Wenzl idempotent and replacing each place where we had a crossing in $D$ with the $(2n)^{th}$ projector. It was proven in \cite{Cody} that for an adequate link $L$ the first $4(n+1)$ coefficients of $n^{th}$ unreduced colored Jones polynomial coincide with the first $4(n+1)$ coefficients of the skein element $S_B^{(n)}(D)$. Our work here initially aimed to understand $S_B^{(n)}(D)$ for an alternating link diagram $D$. Initial examinations of various examples of $S_B^{(n)}(D)$ showed that a certain skein element in the relative Kauffman bracket skein module of the disk with some marked points mostly shows up as sub-skein element of $S_B^{(n)}(D)$. We will call this element the bubble skein element. We eventually found various applications for the bubble element. In fact, in \cite{Hajij} we use the bubble element to compute $q$-series associated with the colored Jones polynomial of certain knots that were unknown, and also to prove two sets of $q$-series identities: Andrews-Gordon identities involving the Ramanujan theta function and a similar set of identities involving the Ramanujan false theta function.

\subsection{Plan of the paper} In section \ref{section2}, we review the skein theory basics needed in this paper. In section \ref{mainresults}, we state our main results. In section \ref{section3}, we give a recursive formula for the bubble skein element. In section \ref{section4}, we use the recursive identity for the bubble skein element to expand this element in terms of a linearly independent set in the relative module of the disk. In section \ref{section5}, we use our results to give an easy way to determine a theta graph spin network evaluation in $\mathcal{S}(S^{3})$. Furthermore, we study the element $S_B^{(n)}$ for the knot $8_5$ and we use our techniques to compute for the tail of this knot.
\subsection{Acknowledgements}
I would like to thank Oliver Dasbach whose guidance made this possible. I would also like to thank Pat Gilmer for teaching me skein theory and for his helpful discussions, Cody Armond for many valuable conversations, Khaled Bataineh and Kyle Istvan for reading early drafts carefully and pointing out many corrections.
\section{Skein theory}
\label{section2}
In this section we review the fundamentals of the Kauffman Bracket Skein Modules and introduce the skein modules that will be used for our purpose. Furthermore, we discuss the recursive definition of Jones-Wenzl idempotent and recall some of its basic properties. For more details about linear skein theory associated with the Kauffman Bracket, see \cite{kaufflinks}, \cite{Lickorish1}, and \cite{Przytycki}.
\begin{definition}(J. Przytycki \cite{Przytycki} and V. Turaev \cite{tur})
Let $M$ be an oriented $3$-manifold. Let $\R$ be a commutative ring with identity and a fixed invertable element $A$. Let $\mathcal{L}_M$ be the set of equivalence classes of isotopy classes of framed links in $M$ including the empty link. Let $\R\mathcal{L}_M$ be the free $\R$-module generated by the set $\mathcal{L}_M$. Let $K(M)$ be the smallest submodule of $\R\mathcal{L}_M$ that is generated by all expressions of the form
\begin{eqnarray*}(1)\hspace{3 mm}
  \begin{minipage}[h]{0.06\linewidth}
        \vspace{0pt}
        \scalebox{0.04}{\includegraphics{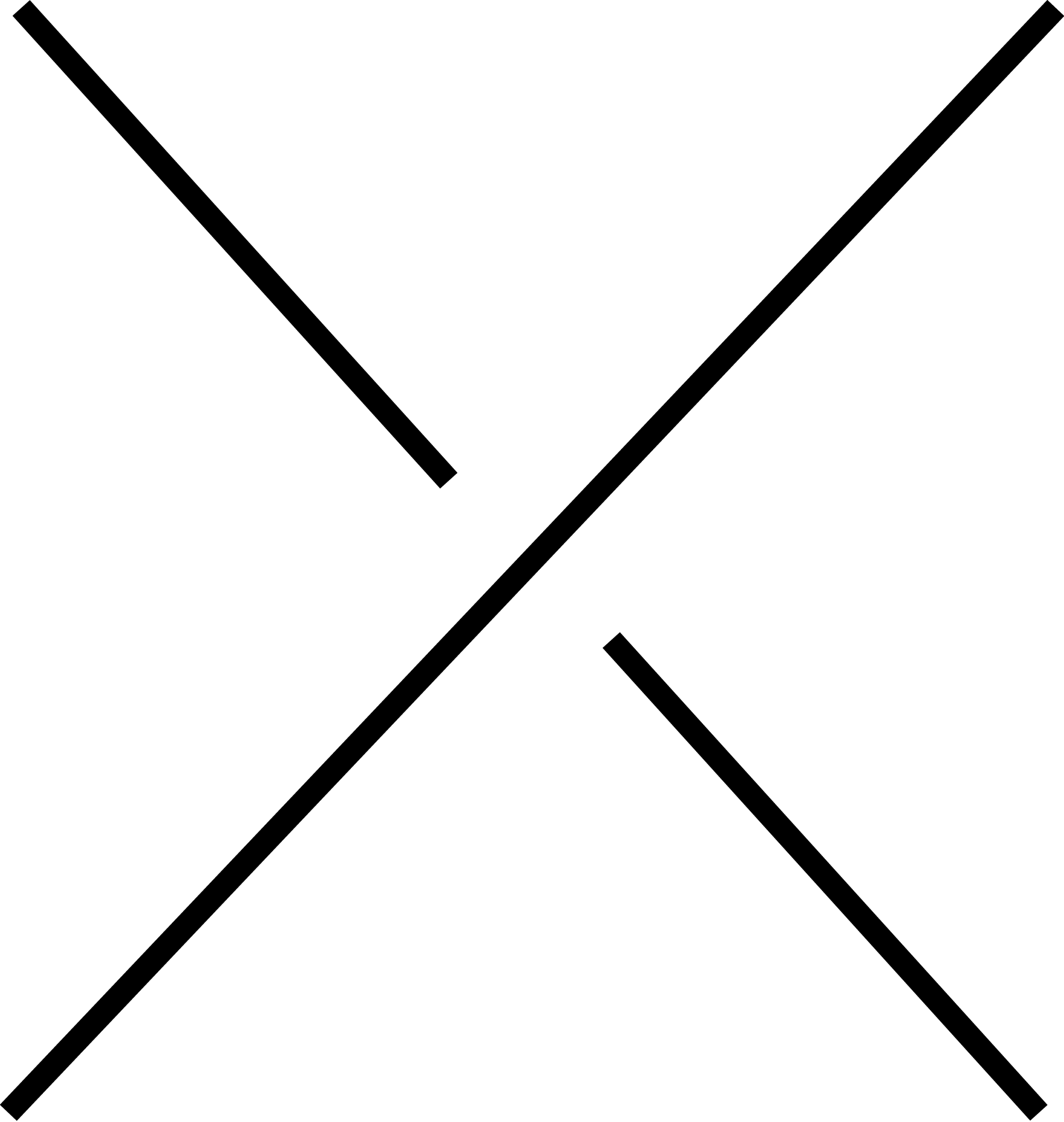}}
   \end{minipage}
   -
     A 
  \begin{minipage}[h]{0.06\linewidth}
        \vspace{0pt}
        \scalebox{0.04}{\includegraphics{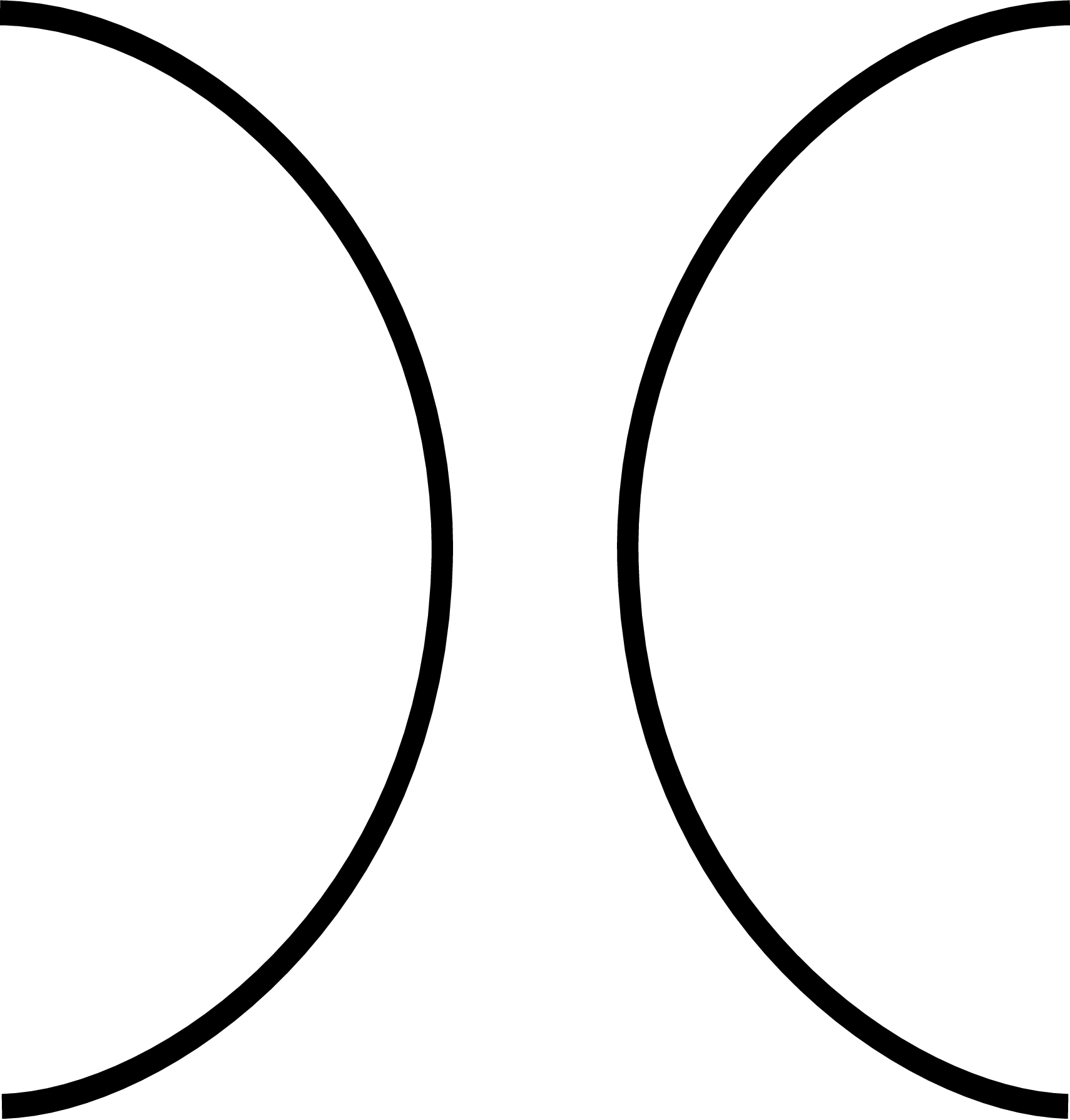}}
   \end{minipage}
   -
  A^{-1} 
  \begin{minipage}[h]{0.06\linewidth}
        \vspace{0pt}
        \scalebox{0.04}{\includegraphics{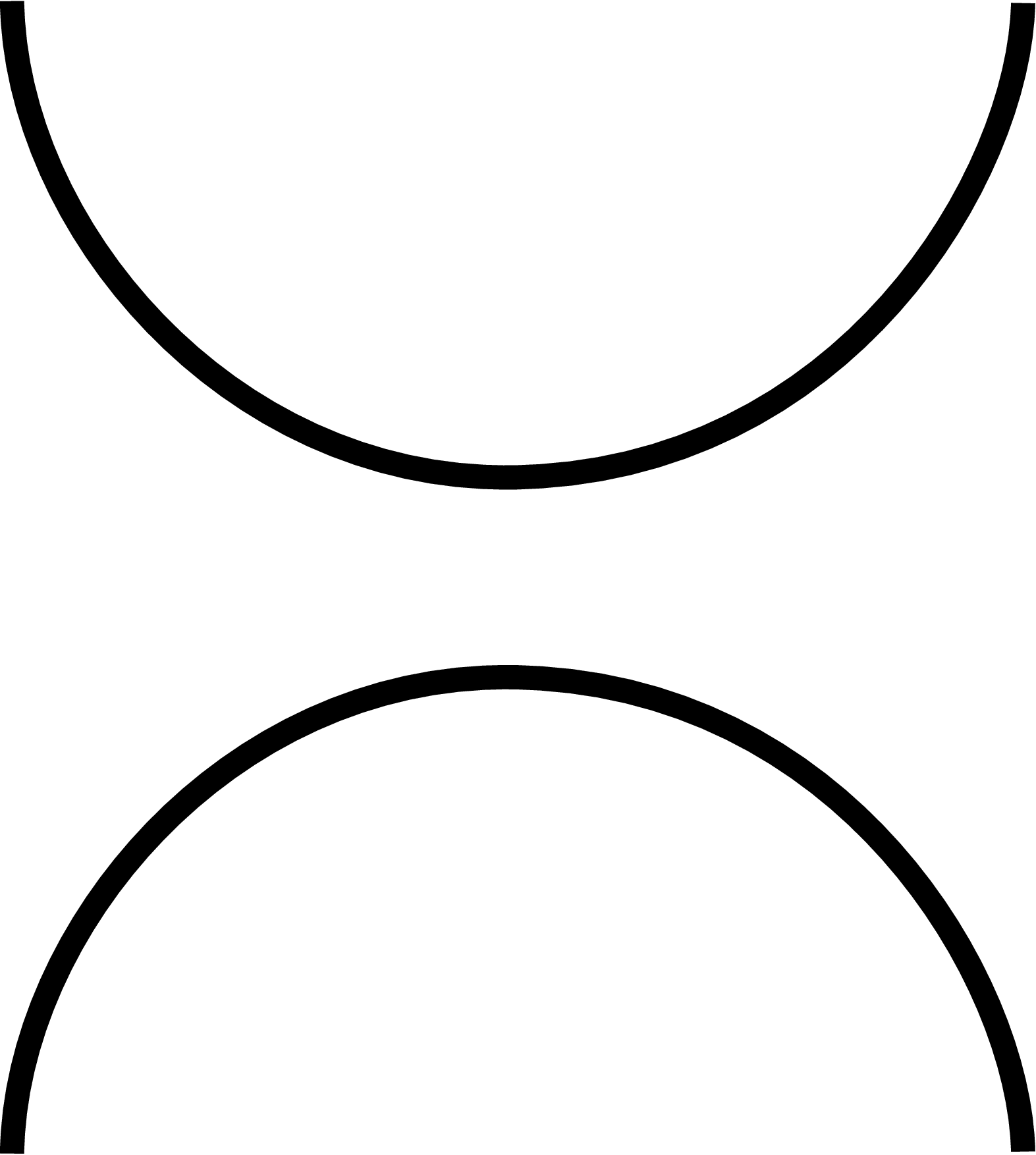}}
   \end{minipage}
, \hspace{20 mm}
  (2)\hspace{3 mm} L\sqcup
   \begin{minipage}[h]{0.05\linewidth}
        \vspace{0pt}
        \scalebox{0.02}{\includegraphics{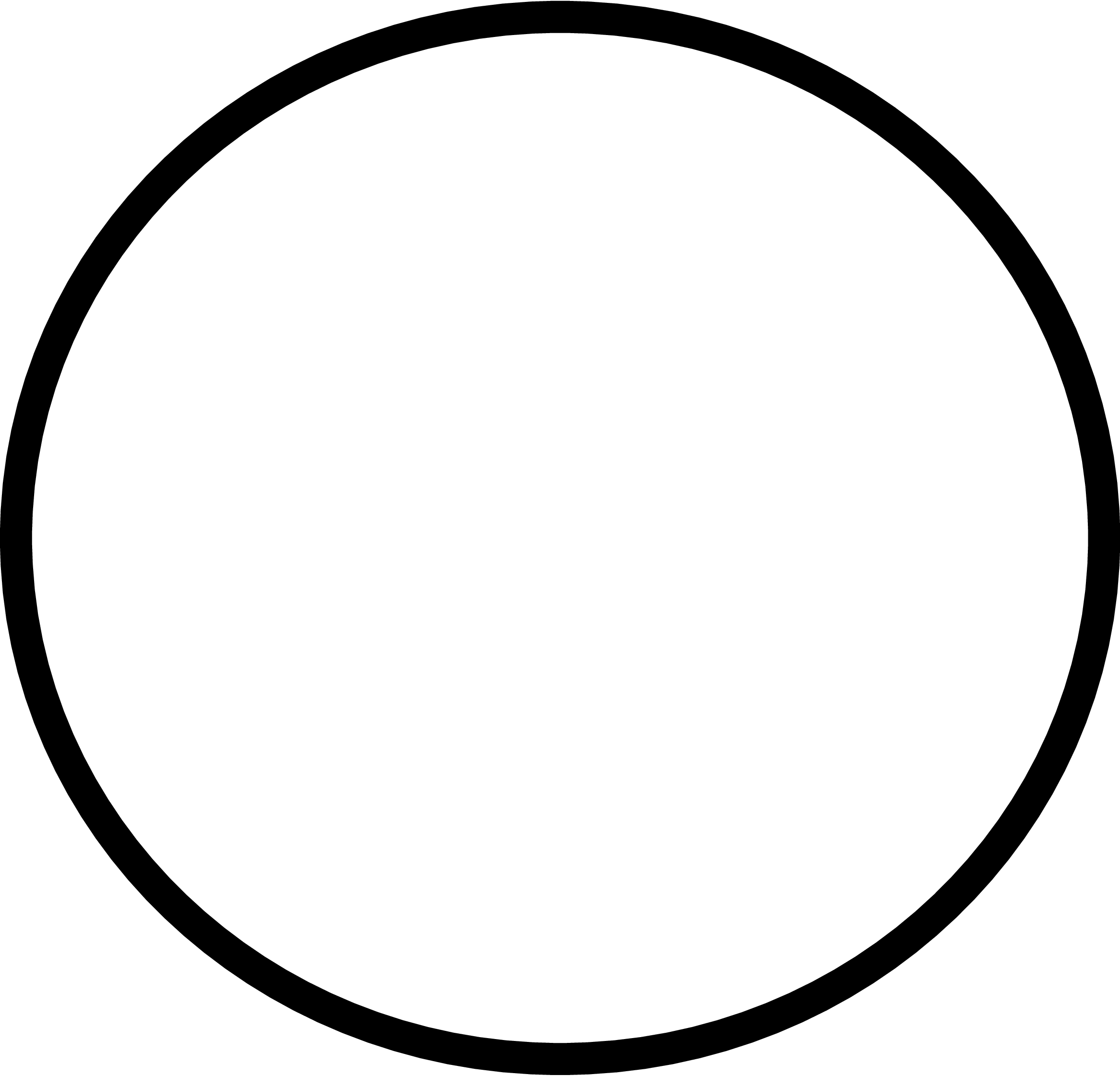}}
   \end{minipage}
  +
  (A^{2}+A^{-2})L. 
  \end{eqnarray*}
where $L\sqcup$ \begin{minipage}[h]{0.05\linewidth}
        \vspace{0pt}
        \scalebox{0.02}{\includegraphics{simple-circle}}
   \end{minipage}  consists of a framed link $L$ in $M$ and the trivial framed knot 
   \begin{minipage}[h]{0.05\linewidth}
        \vspace{0pt}
        \scalebox{0.02}{\includegraphics{simple-circle}}
   \end{minipage}.
The Kauffman bracket skein module, $\mathcal{S}(M;\R,A)$, is defined to be the quotient module
$\mathcal{S}(M;\R,A)=\R\mathcal{L}_M/K(M).$

\end{definition}  
A relative version of the Kauffman bracket skein module can be defined when $M$ has a boundary. The definition is extended as follows. We specify a finite (possibly empty) set of points $x_1,x_2,...,x_{2n}$ on the boundary of $M$. A band is a surface that is homeomorphic to $I\times I$. An element in the set $\mathcal{L}_M$ is an isotopy class of an oriented surface embedded into $M$ and decomposed into a union of finite number of framed links and bands joining the designated boundary points. Let $K(M)$ be the smallest submodule of $\R\mathcal{L}_M$ that is generated by Kauffman relations specified above. The Kauffman bracket skein module is defined to be the quotient module $\mathcal{S}(M;\R,A,\{x_i\}_{i=1}^{2n})=\R\mathcal{L}_M/K(M).$ The relative Kauffman bracket skein module depends only on the distribution of the points $\{x_i\}_{i=1}^{2n}$ among the different connected components of $\partial M$ and it does not depend on the exact position of the points $\{x_i\}_{i=1}^{2n}$. In particular, if $\partial M$ is connected then the definition of the relative Kauffman bracket skein module is independent of the choice of the exact position of the points $\{x_i\}_{i=1}^{2n}$. For more details see \cite{Przytycki} and \cite{Przytycki2}.
\\ \\ 
When the manifold $M$ is homeomorphic to $F\times [0,1]$, where $F$ is a surface with a finite set of points (possibly empty) in its boundary $\partial F$, one could define the (relative) Kauffman bracket skein module of $F$. In this case one considers an appropriate version of link diagrams in $F$ instead of framed links in $M$. The isomorphism between $\mathcal{S}(M;\R,A)$ and $\mathcal{S}(F;\R,A)$ that sends a framed link to its link diagram will be used to identify these two skein modules. \\
\\
We will work with the skein module of the sphere $\mathcal{S}(S^{2};\R,A)$. This skein module is freely generated a one generator. Let $D$ be any diagram in $S^2$. Using the definition of the normalized Kauffman bracket \cite{Kauffman2}, we can write $D = < D >\phi$ in $\mathcal{S}(S^{2};\R,A)$, where $\phi$ denotes the empty link. This provides an isomorphism $<>:\mathcal{S}(S^{2};\R,A) \longrightarrow \R$, induced by sending $D$ to $< D >$. In particular this isomorphism sends the empty link $\phi$ in $\mathcal{S}(S^{2};\R,A)$ to the identity in $\R$. We will also work with the relative skein module $\mathcal{S}([0,1]\times [0,1];\R,A,\{x_i\}_{i=1}^{2n})$, where the rectangular disk $[0,1]\times [0,1]$ has $n$ designated points $\{x_i\}_{i=1}^{n}$ on the top edge , where $x_i=(1,\frac{i}{n+1})$ for $1\leq i \leq n$,  and $n$ designated points on the bottom edge $\{x_i\}_{i=n+1}^{2n}$, where $x_i=(0,\frac{i-n}{n+1})$ for $n+1\leq i \leq 2n$. As we mentioned above, the relative skein module $\mathcal{S}([0,1]\times [0,1];\R,A,\{x_i\}_{i=1}^{2n})$ does not depend on the exact position of the points $\{x_i\}_{i=1}^{2n}$. However, we are making a choice for the position of the points here because we will define an algebra structure on $\mathcal{S}([0,1]\times [0,1];\R,A,\{x_i\}_{i=1}^{2n})$ where the position of these points is required. This relative skein module can be thought of as the $\R$-module generated by all $(n,n)$-tangle diagrams in $[0,1]\times [0,1]$ modulo the Kauffman relations. In fact the module $\mathcal{S}([0,1]\times [0,1];\R,A,\{x_i\}_{i=1}^{2n})$ is a free $\R$-module on $\frac{1}{n+1} {2n \choose n}$ free generators. For a proof of this fact see \cite{Przytycki}. The relative skein module $\mathcal{S}([0,1]\times [0,1];\R,A,\{x_i\}_{i=1}^{2n})$ admits a multiplication given by juxtaposition of two diagrams in $[0,1]\times [0,1]$. More precisely, let $D_1$ and $D_2$ be two diagrams in $[0,1] \times[0,1]$ such that $\partial D_j$, where $j=1,2$, consists of the points $\{x_i\}_{i=1}^{2n}$ specified above. Define $D_1.D_2$ to be the diagram in $[0,1]\times[0,1]$ obtained by attaching $D_1$ on the top of $D_2$ and then compress the result to $[0,1]\times[0,1]$. This multiplication on diagrams extends to a well-defined multiplication on isotopy classes of diagrams in $[0,1]\times[0,1]$. Finally it extends by linearity to a multiplication on $\mathcal{S}([0,1]\times [0,1];\R,A,\{x_i\}_{i=1}^{2n})$. With this multiplication $\mathcal{S}([0,1]\times [0,1];\R,A,\{x_i\}_{i=1}^{2n})$ is an associative algebra over $\R$ known as the \textit{$n^{th}$ Temperley-Lieb algebra} $TL_n$.
For each $n$ there exists an idempotent $f^{(n)}$ in $TL_n$ that plays a central role in the Witten-Reshetikhin-Turaev Invariants for $SU(2)$. See \cite{kaufflinks}, \cite{Lickorish3}, and \cite{Resh}. The idempotent $f^{(n)}$, known as the $n^{th}$ Jones-Wenzl idempotent, was discovered by Jones \cite{Jones} and it has a recursive formula due to Wenzl \cite{Wenzl}. We will use this recursive formula to define $f^{(n)}$. Further, we will adapt a graphical notation for $f^{(n)}$ which is due Lickorish \cite{Lickorish2}. In this graphical notation one thinks of $f^{(n)}$ as an empty box with $n$ strands entering and $n$ strands leaving the opposite side. \textit{The Jones-Wenzl idempotent} is defined by:
\begin{align}
\label{recursive}
  \begin{minipage}[h]{0.05\linewidth}
        \vspace{0pt}
        \scalebox{0.12}{\includegraphics{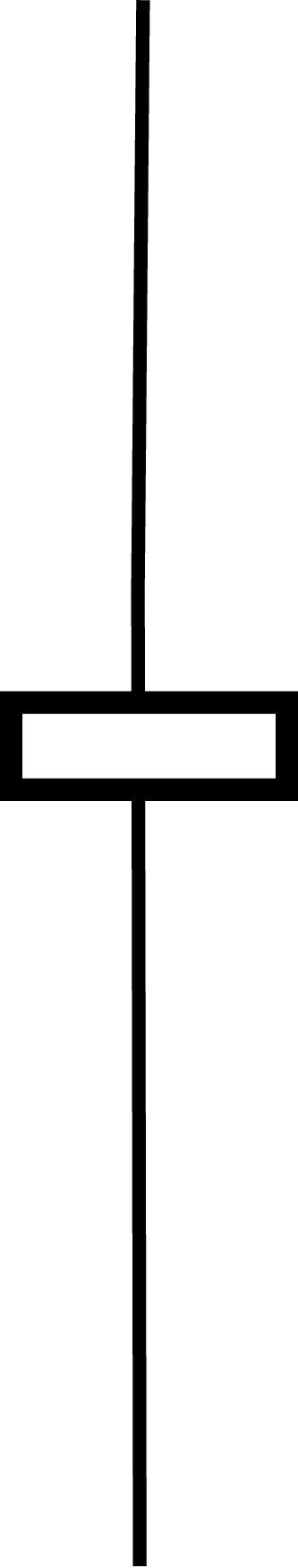}}
         \put(-20,+70){\footnotesize{$n$}}
   \end{minipage}
   =
  \begin{minipage}[h]{0.08\linewidth}
        \hspace{8pt}
        \scalebox{0.12}{\includegraphics{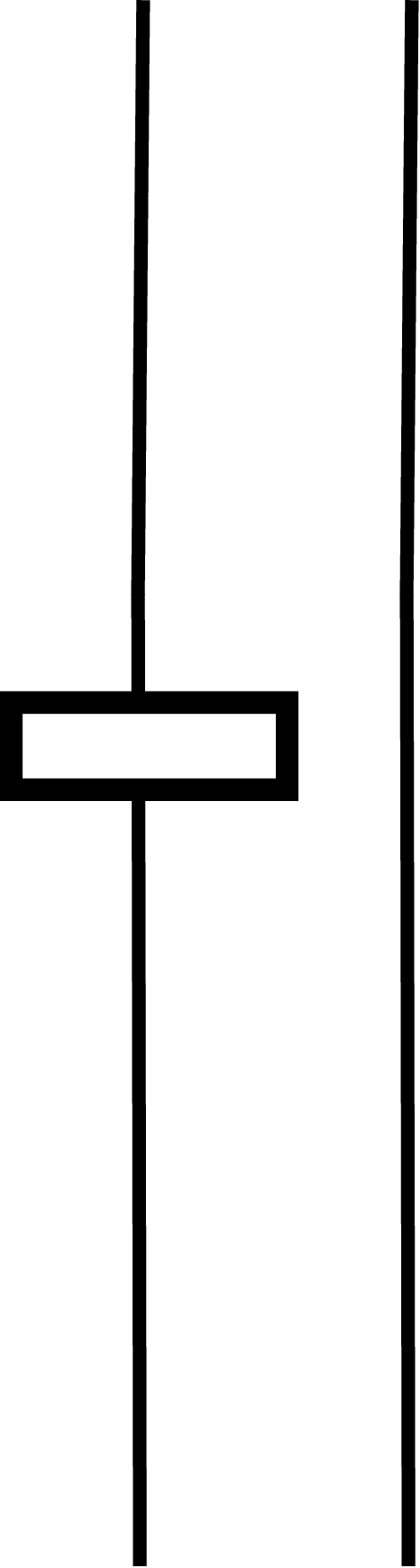}}
        \put(-42,+70){\footnotesize{$n-1$}}
        \put(-8,+70){\footnotesize{$1$}}
   \end{minipage}
   \hspace{9pt}
   -
 \Big( \frac{\Delta_{n-2}}{\Delta_{n-1}}\Big)
  \hspace{9pt}
  \begin{minipage}[h]{0.10\linewidth}
        \vspace{0pt}
        \scalebox{0.12}{\includegraphics{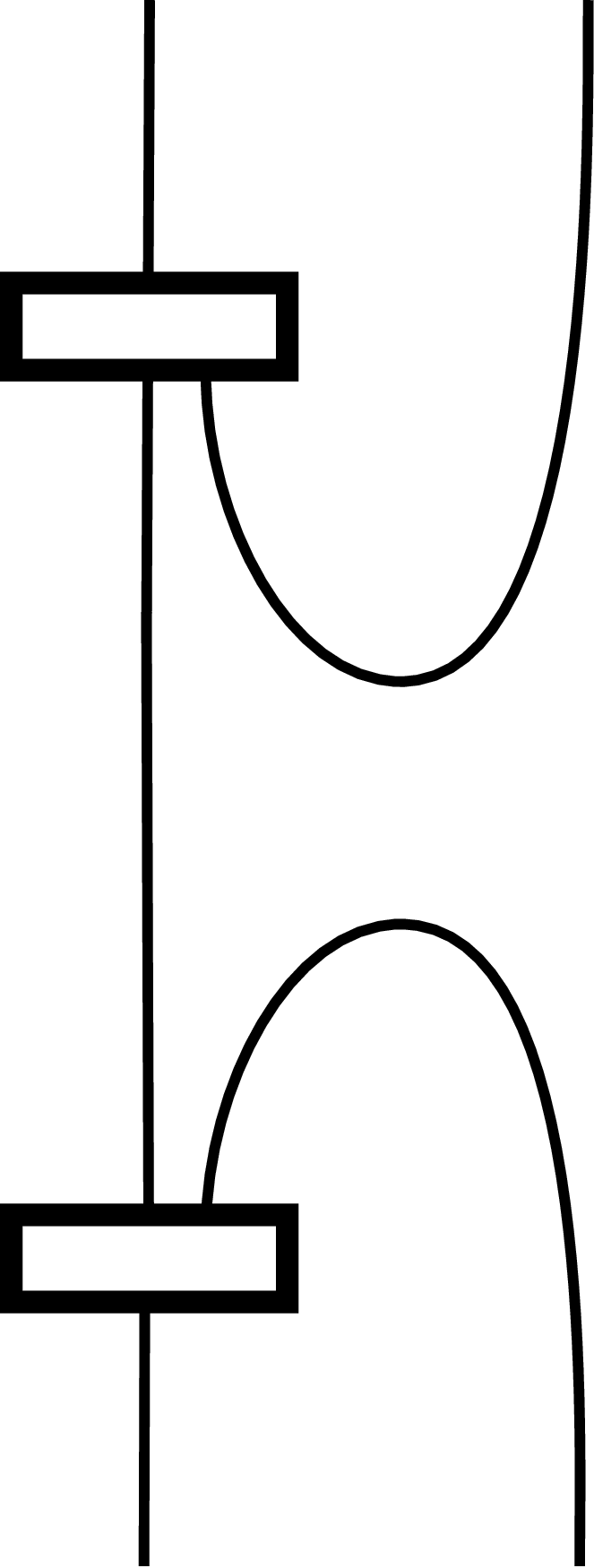}}
         \put(2,+85){\footnotesize{$1$}}
         \put(-52,+87){\footnotesize{$n-1$}}
         \put(-25,+47){\footnotesize{$n-2$}}
         \put(2,+10){\footnotesize{$1$}}
         \put(-52,+5){\footnotesize{$n-1$}}
   \end{minipage}
  , \hspace{20 mm}
    \begin{minipage}[h]{0.05\linewidth}
        \vspace{0pt}
        \scalebox{0.12}{\includegraphics{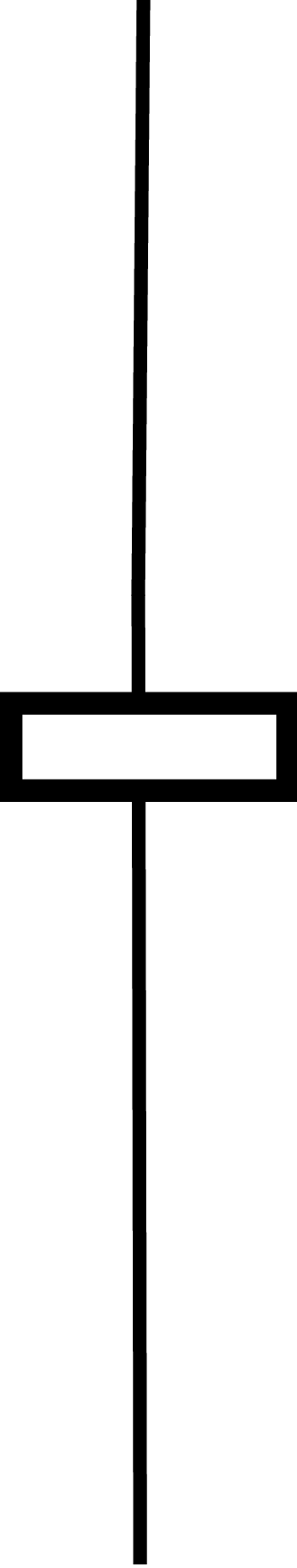}}
        \put(-20,+70){\footnotesize{$1$}}
   \end{minipage}
  =
  \begin{minipage}[h]{0.05\linewidth}
        \vspace{0pt}
        \scalebox{0.12}{\includegraphics{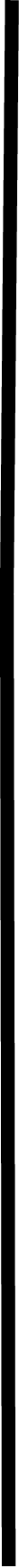}}
   \end{minipage}   
  \end{align}
  where 
\begin{equation*}
 \Delta_{n}=(-1)^{n}\frac{A^{2(n+1)}-A^{-2(n+1)}}{A^2-A^{-2}}.
\end{equation*}
The polynomial $\Delta_{n}$ is related to $[n+1]$, the $(n+1)^{th}$ quantum integer, by $\Delta _{n}=(-1)^{n}[n+1]$. It is common to use the substitution $a=A^2$ when one works with quantum integers.\\\\ 
We assume that $f^{(0)}$ is the empty diagram. The Jones-Wenzl idempotent satisfies
\begin{eqnarray}
\label{properties}
\hspace{0 mm}
\Delta_{n}=
  \begin{minipage}[h]{0.1\linewidth}
        \vspace{0pt}
        \scalebox{0.12}{\includegraphics{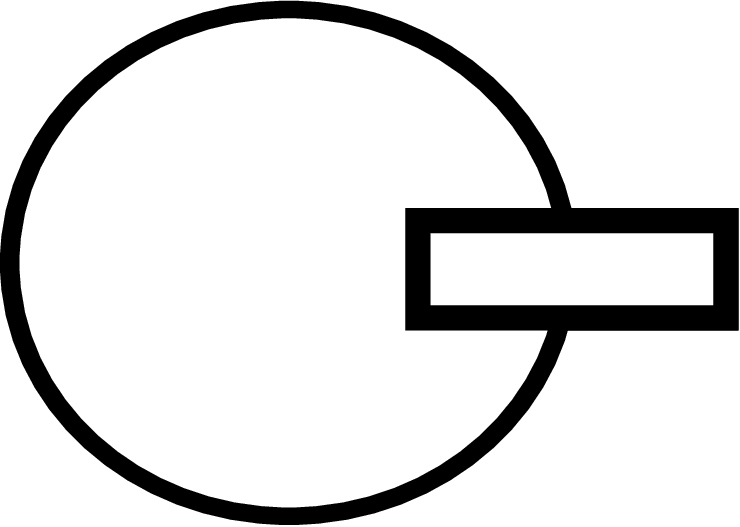}}
        \put(-29,+34){\footnotesize{$n$}}
   \end{minipage}
   , \hspace{14 mm}
     \begin{minipage}[h]{0.08\linewidth}
        \vspace{0pt}
        \scalebox{0.115}{\includegraphics{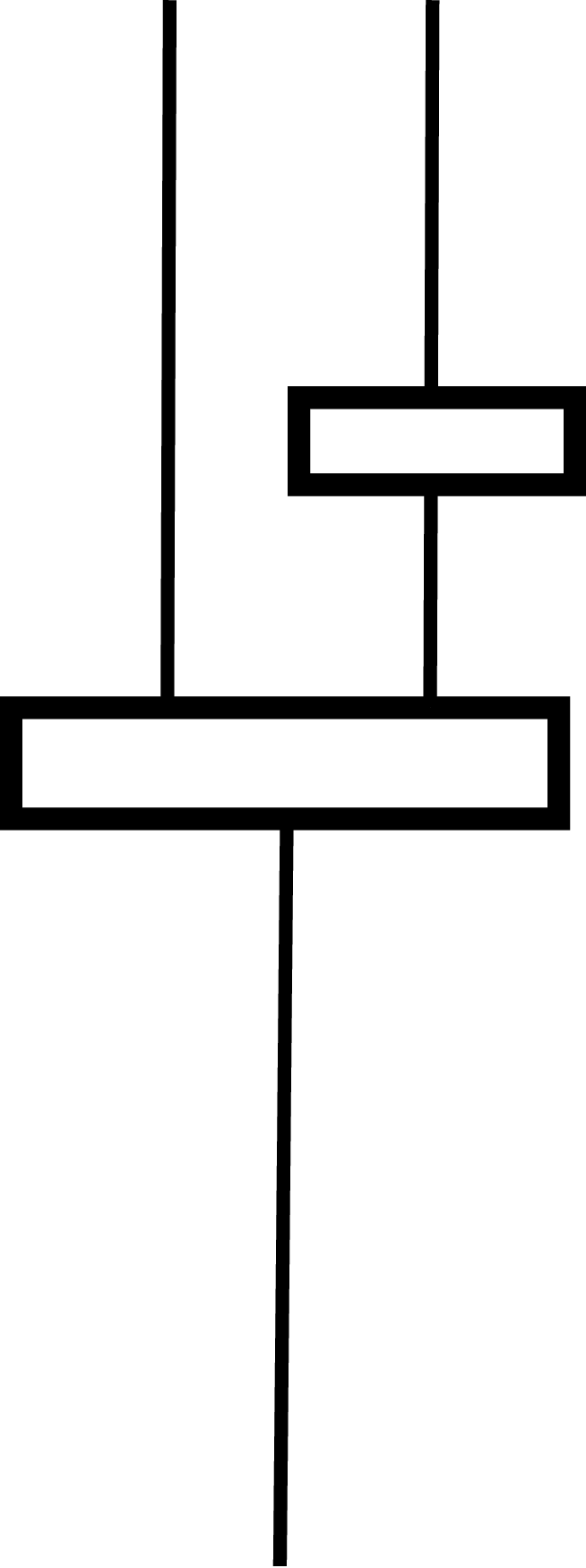}}
        \put(-34,+82){\footnotesize{$n$}}
        \put(-19,+82){\footnotesize{$m$}}
        \put(-46,20){\footnotesize{$m+n$}}
   \end{minipage}
  =
     \begin{minipage}[h]{0.09\linewidth}
        \vspace{0pt}
        \scalebox{0.115}{\includegraphics{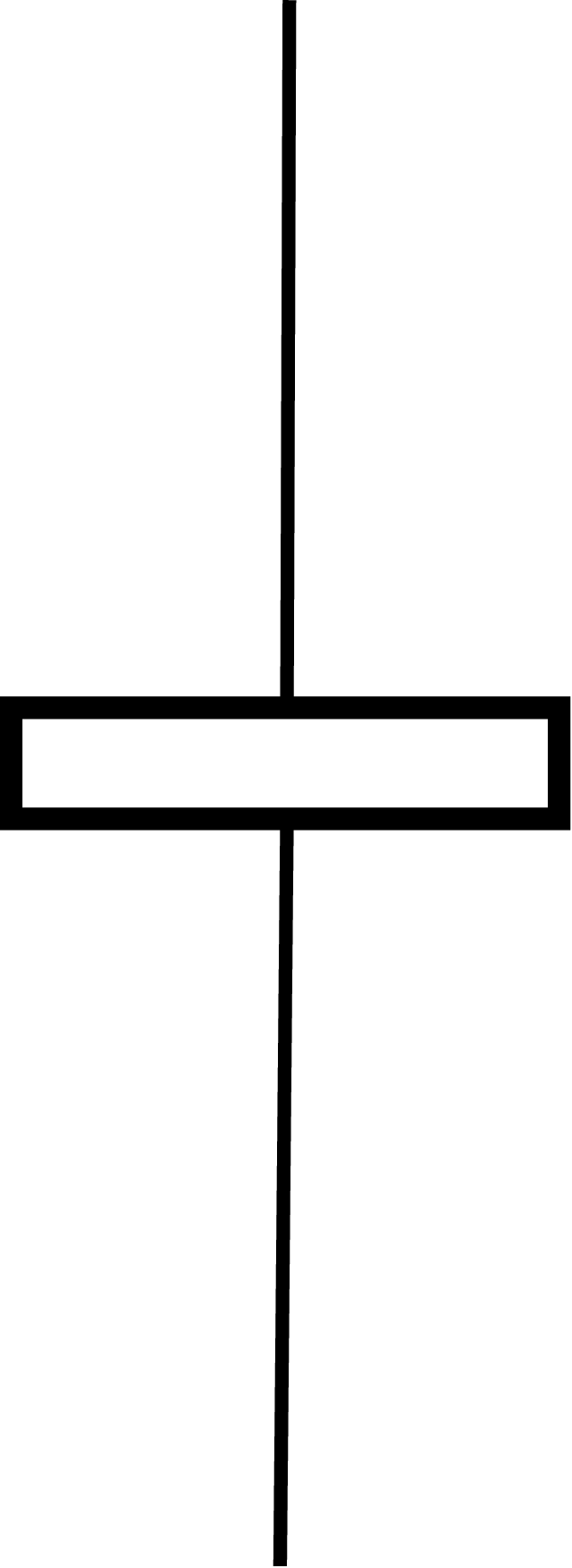}}
        \put(-46,20){\footnotesize{$m+n$}}
   \end{minipage}
    , \hspace{15 mm}
    \begin{minipage}[h]{0.09\linewidth}
        \vspace{0pt}
        \scalebox{0.115}{\includegraphics{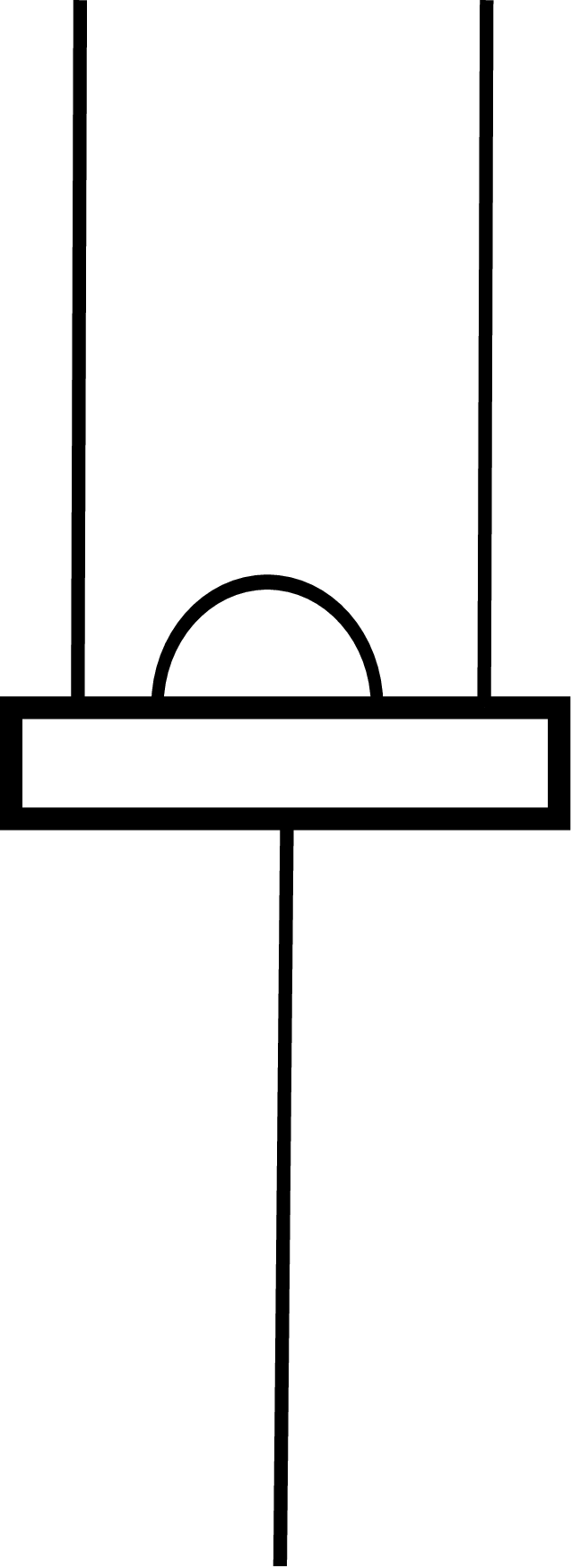}}
         \put(-40,+82){\footnotesize{$n$}}
         \put(-20,+64){\footnotesize{$1$}}
        \put(-2,+82){\footnotesize{$m$}}
        \put(-61,20){\footnotesize{$m+n+2$}}
   \end{minipage}
   =0,
  \end{eqnarray}
  and
  \begin{eqnarray}
\label{properties2}
    \begin{minipage}[h]{0.09\linewidth}
        \vspace{0pt}
        \scalebox{0.115}{\includegraphics{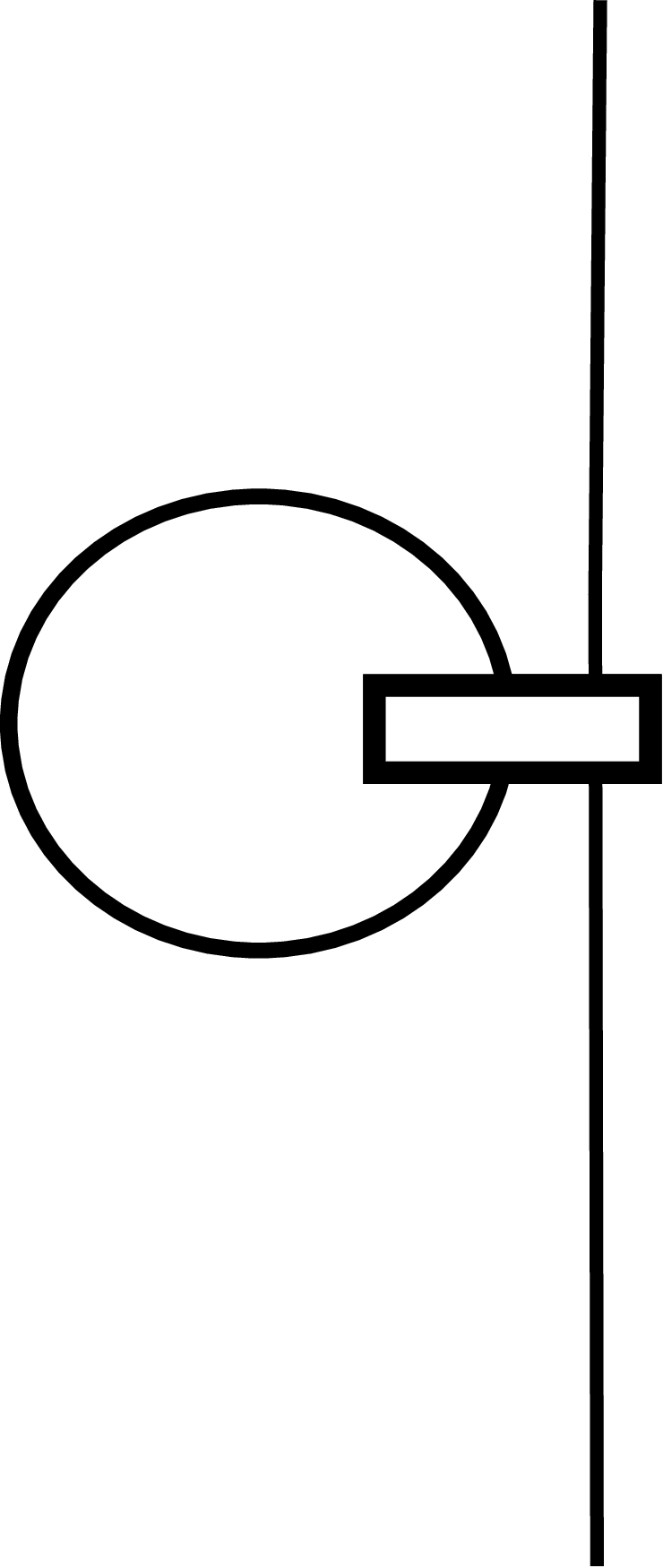}}
          \put(-13,+82){\footnotesize{$n$}}
        \put(-29,+70){\footnotesize{$m$}}
   \end{minipage}
   =\frac{\Delta_{m+n}}{\Delta_{n}}\hspace{1 mm}  
    \begin{minipage}[h]{0.06\linewidth}
        \vspace{0pt}
        \scalebox{0.115}{\includegraphics{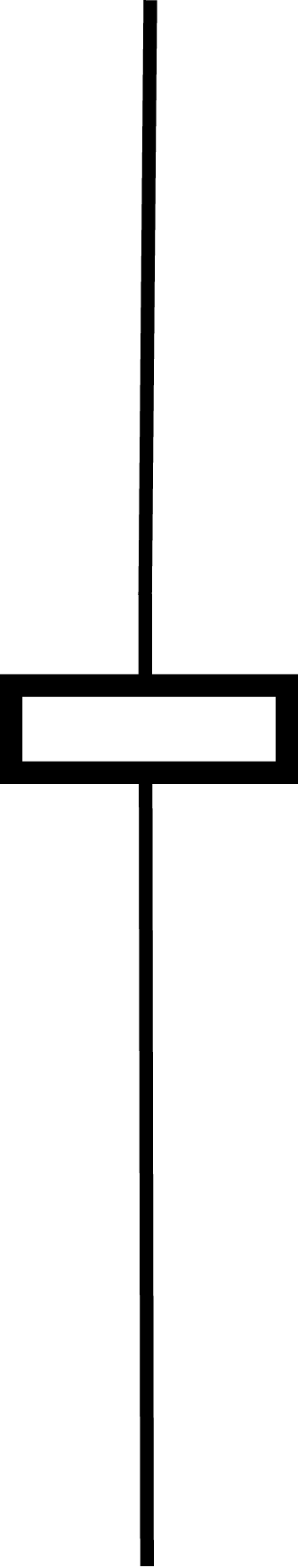}}
           \put(-19,+82){\footnotesize{$n$}}
   \end{minipage}
   \end{eqnarray}

Let $m, n, m^{\prime}, n^{\prime} \geq 0 $. The last relative Kauffman bracket skein module that we will need here is a certain submodule of the relative skein module of $D^2$ with $m+n+m^{\prime}+n^{\prime}$ specified points on its boundary. To define this submodule we partition this set of points into four sets of $m$, $n$, $m^{\prime}$, and $n^{\prime}$ points, respectively, and at each cluster of these points we place an appropriate idempotent, i.e., the one whose color matches the cardinality of this cluster (see Figure \ref{disk}). We will denote this skein module by $\mathscr{D}^{m,n}_{m^{\prime},n^{\prime}}$. Note that if $m+n+m^{\prime}+n^{\prime}$ is odd then $\mathscr{D}^{m,n}_{m^{\prime},n^{\prime}}=0$ 

\begin{figure}[H]
  \centering
   {\includegraphics[scale=0.13]{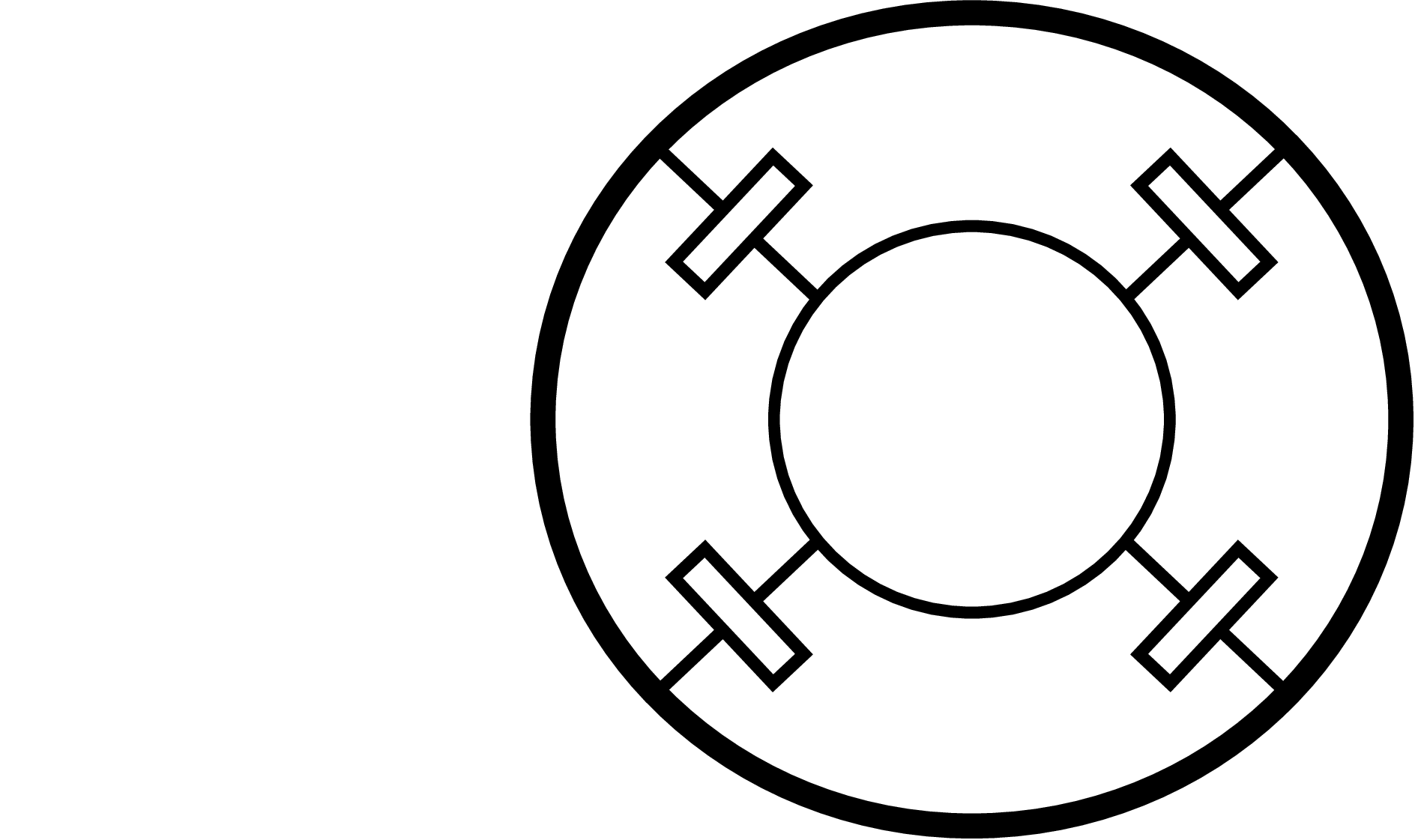}
    \put(-75,+62){$m$}
          \put(-10,+62){$n$}
          \put(-75,0){$m^{\prime}$}
          \put(-10,0){$n^{\prime}$}
            \caption{The relative skein module $\mathscr{D}^{m,n}_{m^{\prime},n^{\prime}}$}
  \label{disk}}
\end{figure}
We assume that a diagram is zero if it has a strand labeled by a negative number.\\

As we mentioned in the introduction, we are interested in a particular skein element in the module $\mathscr{D}^{m,n}_{m^{\prime},n^{\prime}}$. This element is shown in Figure \ref{bubble} and it is denoted by $\mathcal{B}_{m^{\prime},n^{\prime}}^{m,n}(k,l)$, where $k,l \geq 1$. We will call such an element in $\mathscr{D}^{m,n}_{m^{\prime},n^{\prime}}$ a \textit{bubble skein element}. For every bubble skein element $\mathcal{B}_{m^{\prime},n^{\prime}}^{m,n}(k,l)$, the integers $m,n,m^{\prime},n^{\prime}$, $k,l \geq 0$ and they satisfy $m+k=m^{\prime}+l$ and $n+k=n^{\prime}+l$.
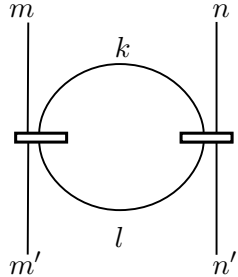
\begin{figure}[H]
  \centering
   {\includegraphics[scale=0.13]{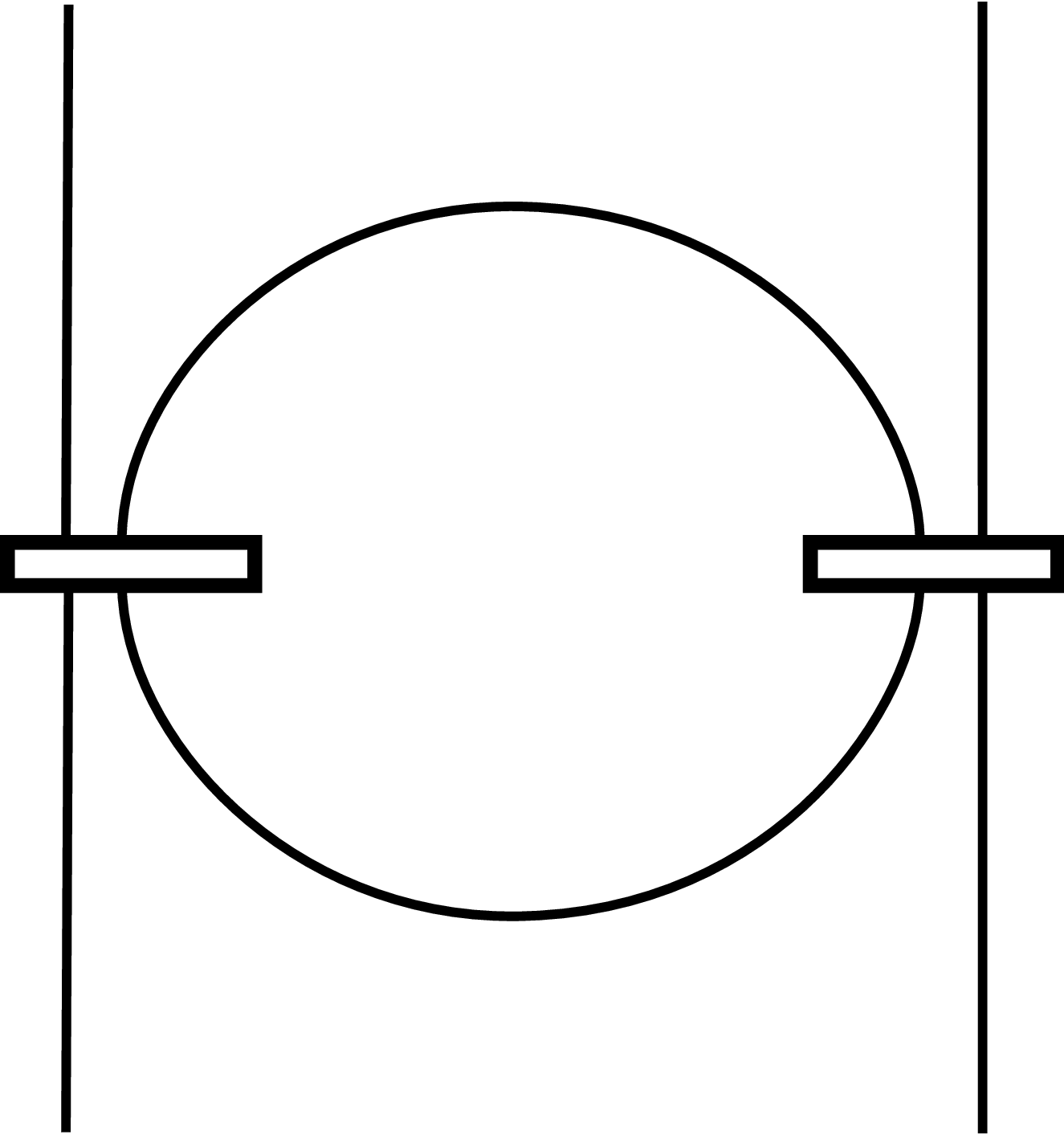}
   \put(-85,+90){$m$}
          \put(-8,+90){$n$}
          \put(-85,-7){$m^{\prime}$}
          \put(-8,-7){$n^{\prime}$}
           \put(-45,+75){$k$}
         \put(-45,+3){$l$}
  \caption{The bubble skein element $\mathcal{B}_{m^{\prime},n^{\prime}}^{m,n}(k,l)$}
  \label{bubble}}
\end{figure}

In this paper $\R$ will be $ \mathbb{Q}(A)$, the field generated by the indeterminate $A$ over the rational numbers.
\section{Main results}
\label{mainresults}
The main work of this paper gives an expansion the bubble skein element $\mathcal{B}_{m^{\prime},n^{\prime}}^{m,n}(k,l)$, defined in the previous section, in terms of a set of $ \mathbb{Q}(A)$-linearly independent skein elements in the module $\mathscr{D}^{m,n}_{m^{\prime},n^{\prime}}$ and gives an explicit determination of the coefficients obtained from this expansion.\\

Recall that the quantum binomial coefficients are defined by 
\begin{equation*}
{l \brack i}_{q}=\frac{(q;q)_l}{(q;q)_i(q;q)_{l-i}}.
\end{equation*}
where $(a;q)_n$ is $q$-Pochhammer symbol which is defined as 
\begin{equation*}
(a;q)_n=\prod\limits_{j=0}^{n-1}(1-aq^j).
\end{equation*}.
\begin{theorem}(The bubble expansion formula)
\label{main111}
Let $m,n,m^{\prime},n^{\prime} \geq 0$, and $k\geq l$;  $k,l \geq 1$. Then\\
{\small
\begin{eqnarray*}
\label{bubble expansion formula121}
  \begin{minipage}[h]{0.15\linewidth}
        \vspace{0pt}
        \scalebox{0.11}{\includegraphics{second-lemma-main-bubble}}
        \put(-68,+80){$m$}
          \put(-8,+80){$n$}
          \put(-68,-7){$m^{\prime}$}
          \put(-8,-7){$n^{\prime}$}
           \put(-37,+67){$k$}
         \put(-37,+3){$l$}
   \end{minipage}
   =\displaystyle\sum\limits_{i=0}^{\min(m,n,l)}
   \left\lceil 
\begin{array}{cc}
m & n \\ 
k & l%
\end{array}%
\right\rceil _{i}
    \begin{minipage}[h]{0.15\linewidth}
        \vspace{0pt}
        \scalebox{0.11}{\includegraphics{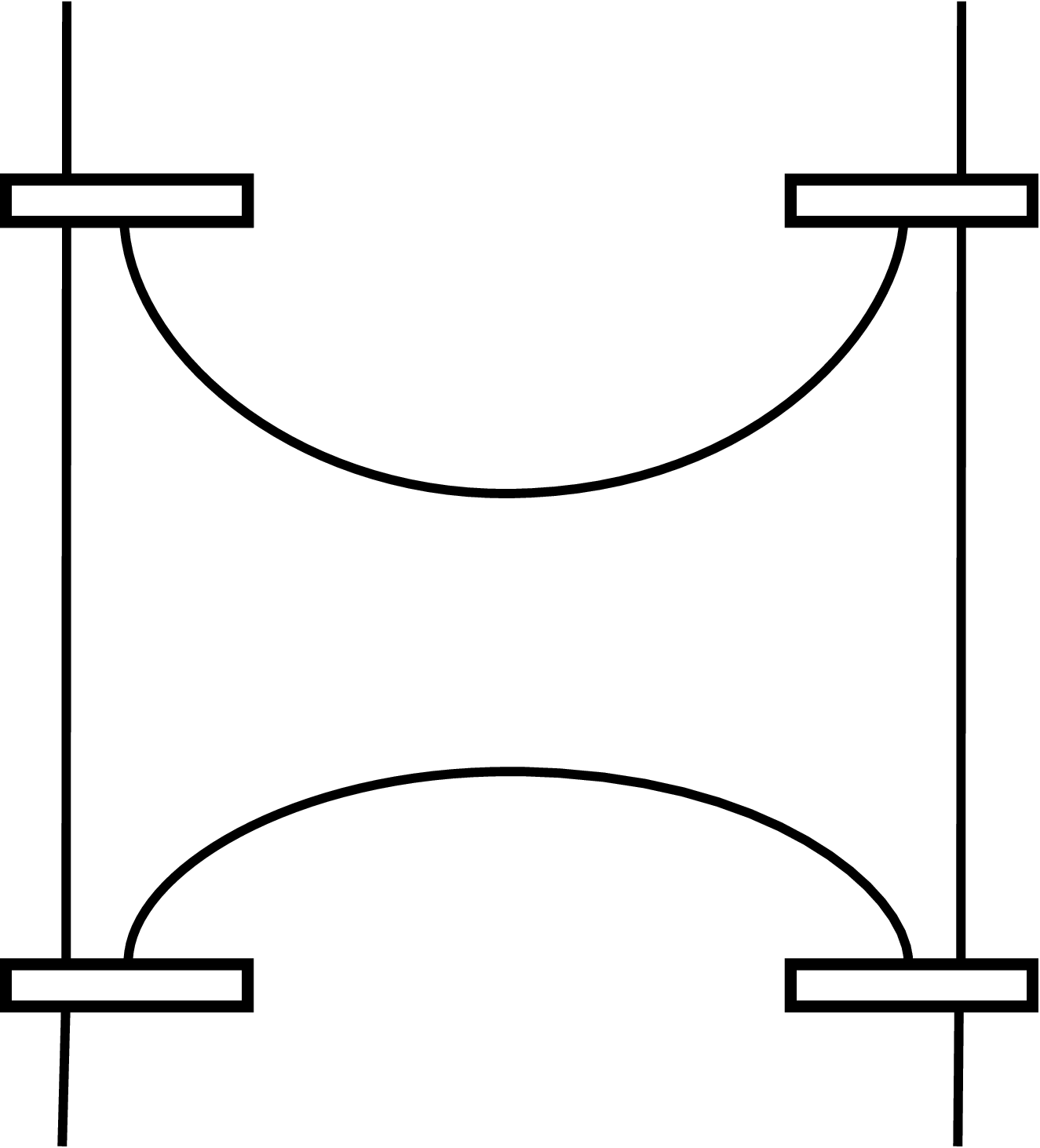}}
        \put(-68,+80){$m$}
          \put(-8,+80){$n$}
          \put(-68,-7){$m^{\prime}$}
          \put(-8,-7){$n^{\prime}$}
           \put(-37,+50){$i$}
         \put(-52,28){$k-l+i$}
   \end{minipage}.
  \end{eqnarray*}
}
where 
\begin{equation*}
\left\lceil 
\begin{array}{cc}
m & n \\ 
k & l%
\end{array}%
\right\rceil _{i}:=(-A^2)^{i(i-l)}\frac{\displaystyle\prod_{j=0}^{l-i-1}\Delta
_{k-j-1}\prod_{s=0}^{i-1}\Delta _{n-s-1}\Delta _{m-s-1}}{%
\displaystyle\prod_{t=0}^{l-1}\Delta _{n+k-t-1}\Delta _{m+k-t-1}}{l \brack i}_{A^4}\prod_{j=0}^{l-i-1}\Delta_{m+n+k-i-j}.
\end{equation*}
\end{theorem}
We give two applications of the previous theorem. The first one gives a relation between the coefficient $\left\lceil 
\begin{array}{cc}
m & n \\ 
k & k%
\end{array}%
\right\rceil _{0}$ and the theta graph $\Lambda(m,n,k)$ (see Figure \ref{theta}) in $\mathcal{S}(S^2)$.
\begin{proposition}
\begin{equation*}
  \begin{minipage}[h]{0.125\linewidth}
       \scalebox{0.10}{\includegraphics{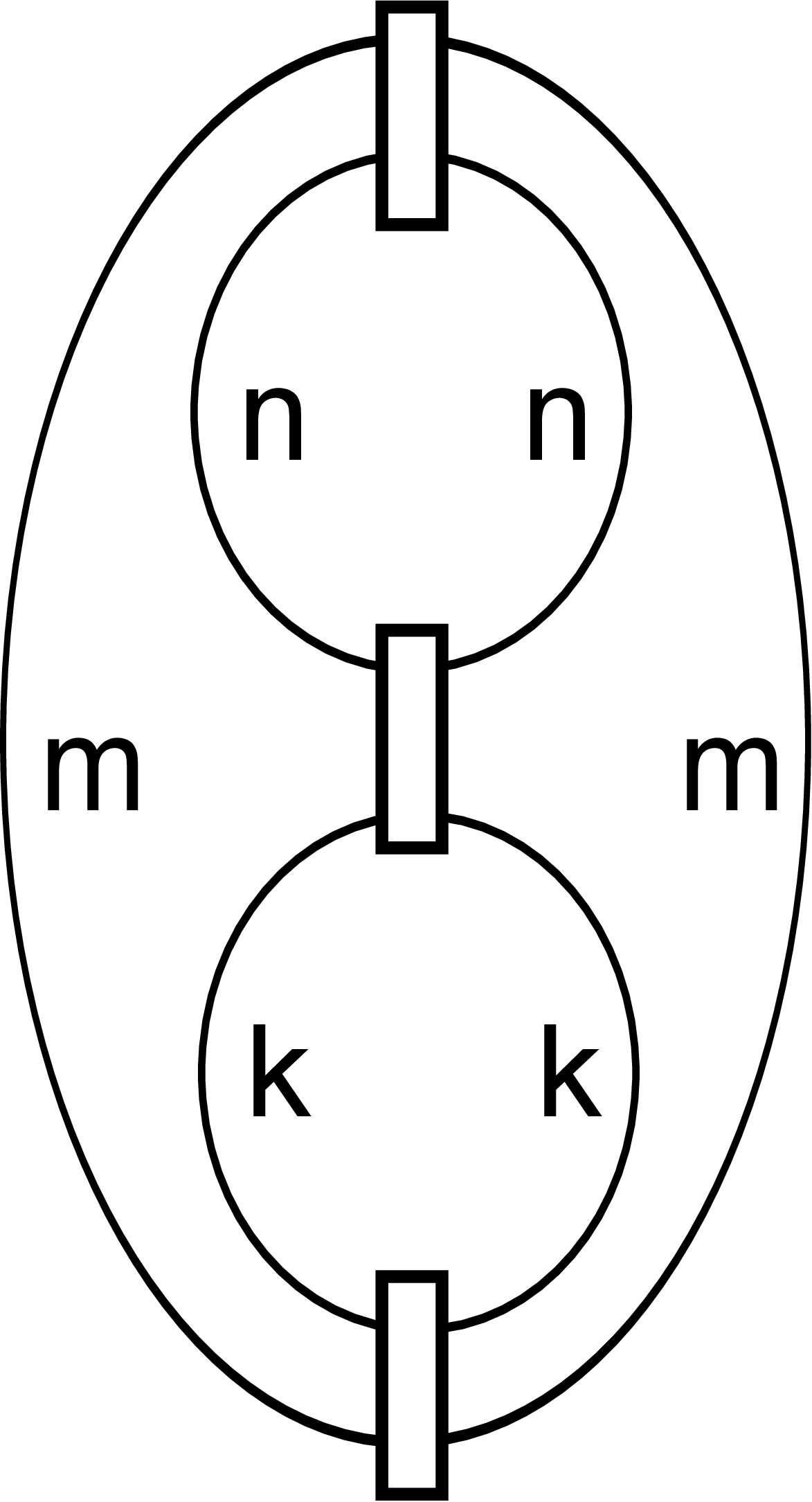}}
   \end{minipage}
   =\left\lceil 
\begin{array}{cc}
m & n \\ 
k & k%
\end{array}%
\right\rceil _{0}\Delta_{m+n}.
   \end{equation*}
\end{proposition}
The tail of an alternating link $L$ is a $q$-series associated to $L$. The second application that we give to the bubble expansion formula is showing that this expansion can be used to study and compute the tail of alternating links. In particular, we give a simple formula for the tail of the knot $8_5$:
\begin{proposition}
\begin{eqnarray*}
T_{8_5}(q)=(q;q)_{\infty}^2\sum\limits_{k=0}^{\infty}\frac{q^{k+k^{2}}}{%
(q;q)_{k}}(\sum\limits_{i=0}^{k}q^{(-2i(k-i))}\left[ 
\begin{array}{c}
k \\ 
i%
\end{array}%
\right] _{q}^{2}).
\end{eqnarray*}
\end{proposition}
\section{A recursive formula for the bubble skein element}
\label{section3}
In this section we use the recursive definition of the Jones-Wenzl idempotent to obtain a recursive formula for the bubble skein element $\mathcal{B}_{m^{\prime},n^{\prime}}^{m,n}(k,l)$ in the module $\mathscr{D}^{m,n}_{m^{\prime},n^{\prime}}$. To obtain a recursive formula for the element $\mathcal{B}_{m^{\prime},n^{\prime}}^{m,n}(k,l)$ we start by expanding the simple bubble element $\mathcal{B}_{m^{\prime},n^{\prime}}^{m,n}(k,1)$ in Lemma \ref{lemma1} and then use this expansion to obtain a recursion equation for an arbitrary bubble skein element in Lemma \ref{lemma2}. The recursive formula of $\mathcal{B}_{m^{\prime},n^{\prime}}^{m,n}(k,l)$ obtained in this section will be used in Theorem \ref{main} to write a bubble skein element as a $ \mathbb{Q}(A)$-linear sum of linearly independent elements in $\mathscr{D}^{m,n}_{m^{\prime},n^{\prime}}$.\\

We denote the rational functions
$\frac{\Delta
_{n+k}\Delta _{m+k-1}-\Delta _{n}\Delta _{m-1}}{\Delta _{n+k-1}\Delta
_{m+k-1}}$ and $\frac{\Delta _{m-1}\Delta _{n-1}}{\Delta
_{n+k-1}\Delta _{m+k-1}}$ by $\alpha_{m,n}^k$ and $\beta_{m,n}^k$, respectively.
\begin{lemma}
\label{lemma1}
Let $m,n,m^{\prime},n^{\prime} \geq 0$; $k\geq 1$. Then we have

\begin{eqnarray*}
   \begin{minipage}[h]{0.15\linewidth}
        \vspace{0pt}
        \scalebox{0.11}{\includegraphics{second-lemma-main-bubble}}
         \put(-68,+80){$m$}
          \put(-8,+80){$n$}
          \put(-68,-7){$m^{\prime}$}
          \put(-8,-7){$n^{\prime}$}
           \put(-37,+67){$k$}
         \put(-37,+3){$1$}
   \end{minipage}
   =
     \alpha_{m,n}^k
  \begin{minipage}[h]{0.15\linewidth}
        \vspace{0pt}
        \scalebox{0.11}{\includegraphics{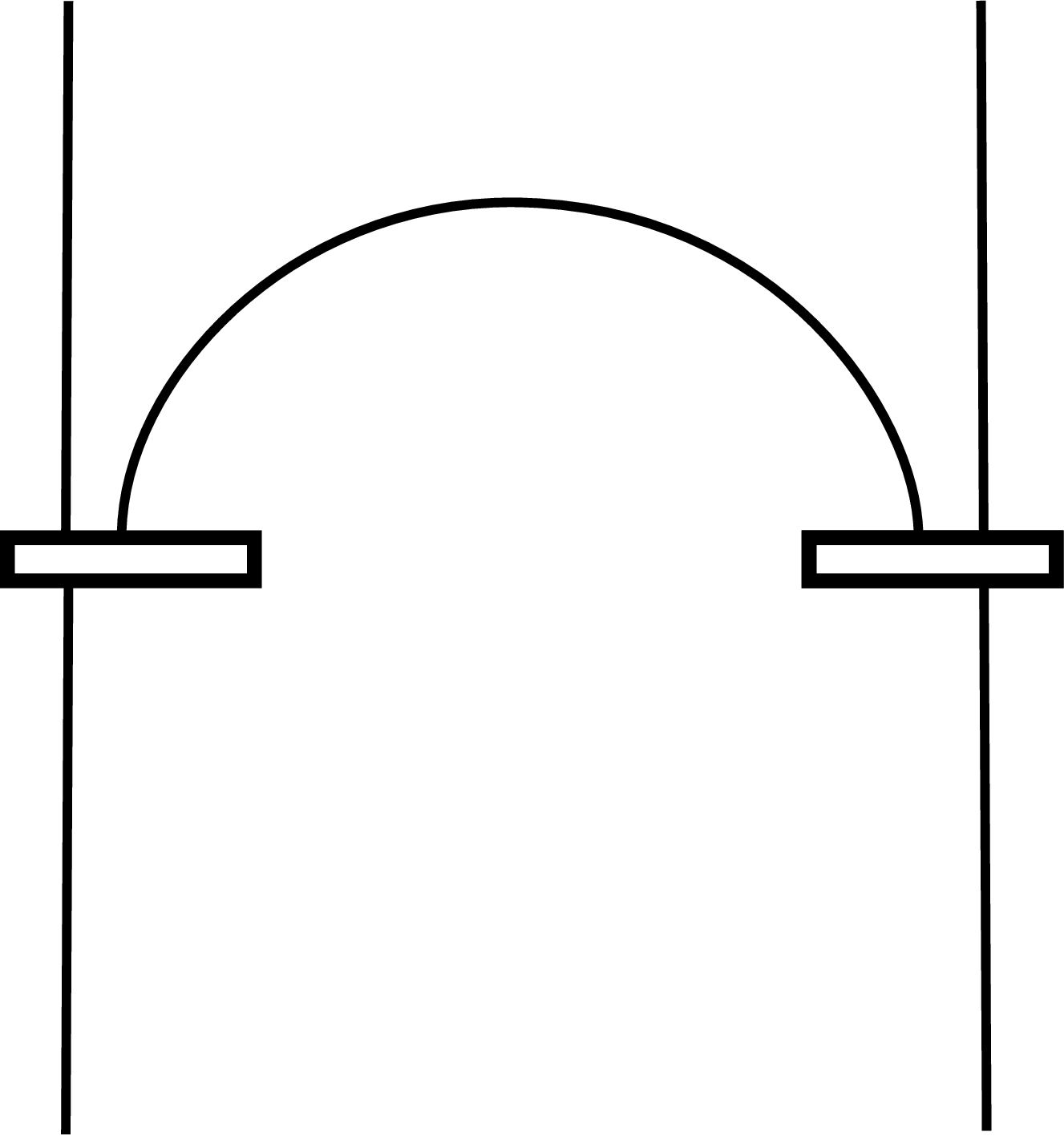}}
        \put(-68,+80){$m$}
          \put(-8,+80){$n$}
          \put(-68,-7){$m^{\prime}$}
          \put(-8,-7){$n^{\prime}$}
           \put(-45,+67){$k-1$}
          \end{minipage}
   +
  \beta_{m,n}^k
  \begin{minipage}[h]{0.15\linewidth}
        \vspace{0pt}
        \scalebox{0.11}{\includegraphics{bubbl-expansion-diagram}}
        \put(-68,+80){$m$}
          \put(-8,+80){$n$}
          \put(-68,-7){$m^{\prime}$}
          \put(-8,-7){$n^{\prime}$}
           \put(-38,+50){$1$}
         \put(-38,28){$k$}
   \end{minipage}.
\end{eqnarray*}

\end{lemma}
\begin{proof}
First we consider the trivial cases when one of the integers $m,n,m^{\prime},n^{\prime}$ is zero. Observe that $m+k=1+m^{\prime}$ and $n+k=1+n^{\prime}$. Since $k\geq 1$ , we know that $m\leq m^{\prime}$ and $n\leq n^{\prime}$. If $m^{\prime}=n^{\prime}=0$, then this implies that $m$ and $n$ must be zero and $k$ is $1$. Hence $\mathcal{B}_{0,0}^{0,0}(1,1) =\alpha_{0,0}^1=\Delta _{1}$ and we are done. Here we used our convention that a diagram is zero if it has a strand colored by a negative number. If $\min(m,n)=0$ or $\min(m^{\prime},n^{\prime})=0$ , then the result follows from (\ref{properties2}).
\\\\
When $\min(m,n)\neq 0$, we use induction on $k$. For $k=1$ we apply the recursive definition of the Jones-Wenzel idempotent (\ref{recursive}) on the projector $f^{(m+1)}$ that appears in $\mathcal{B}_{m^{\prime},n^{\prime}}^{m,n}(1,1)$ to obtain
{\small
\begin{eqnarray}
\label{special0}
  \begin{minipage}[h]{0.15\linewidth}
        \vspace{0pt}
        \scalebox{0.11}{\includegraphics{second-lemma-main-bubble}}
         \put(-68,+80){$m$}
          \put(-8,+80){$n$}
          \put(-68,-7){$m^{\prime}$}
          \put(-8,-7){$n^{\prime}$}
           \put(-37,+67){$1$}
         \put(-37,+3){$1$}
   \end{minipage}
   =\hspace{2pt}
    \begin{minipage}[h]{0.15\linewidth}
        \vspace{0pt}
        \scalebox{0.11}{\includegraphics{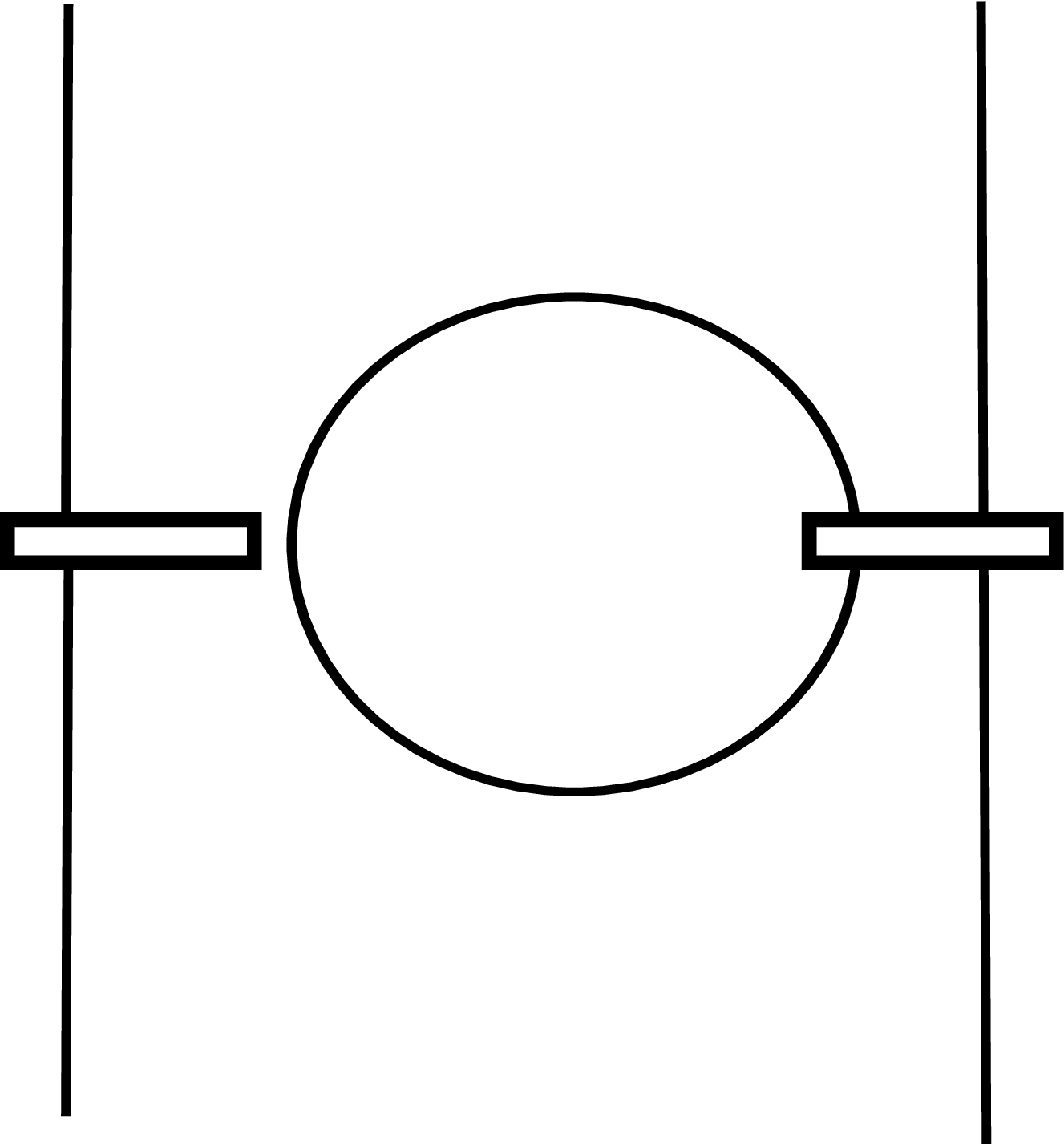}}
         \put(-68,+80){$m$}
          \put(-8,+80){$n$}
          \put(-68,-7){$m^{\prime}$}
          \put(-8,-7){$n^{\prime}$}
        \put(-37,+13){$1$}
   \end{minipage}
   -
  \frac{\Delta _{m-1}}{\Delta _{m}}
  \begin{minipage}[h]{0.16\linewidth}
        \vspace{0pt}
        \scalebox{0.11}{\includegraphics{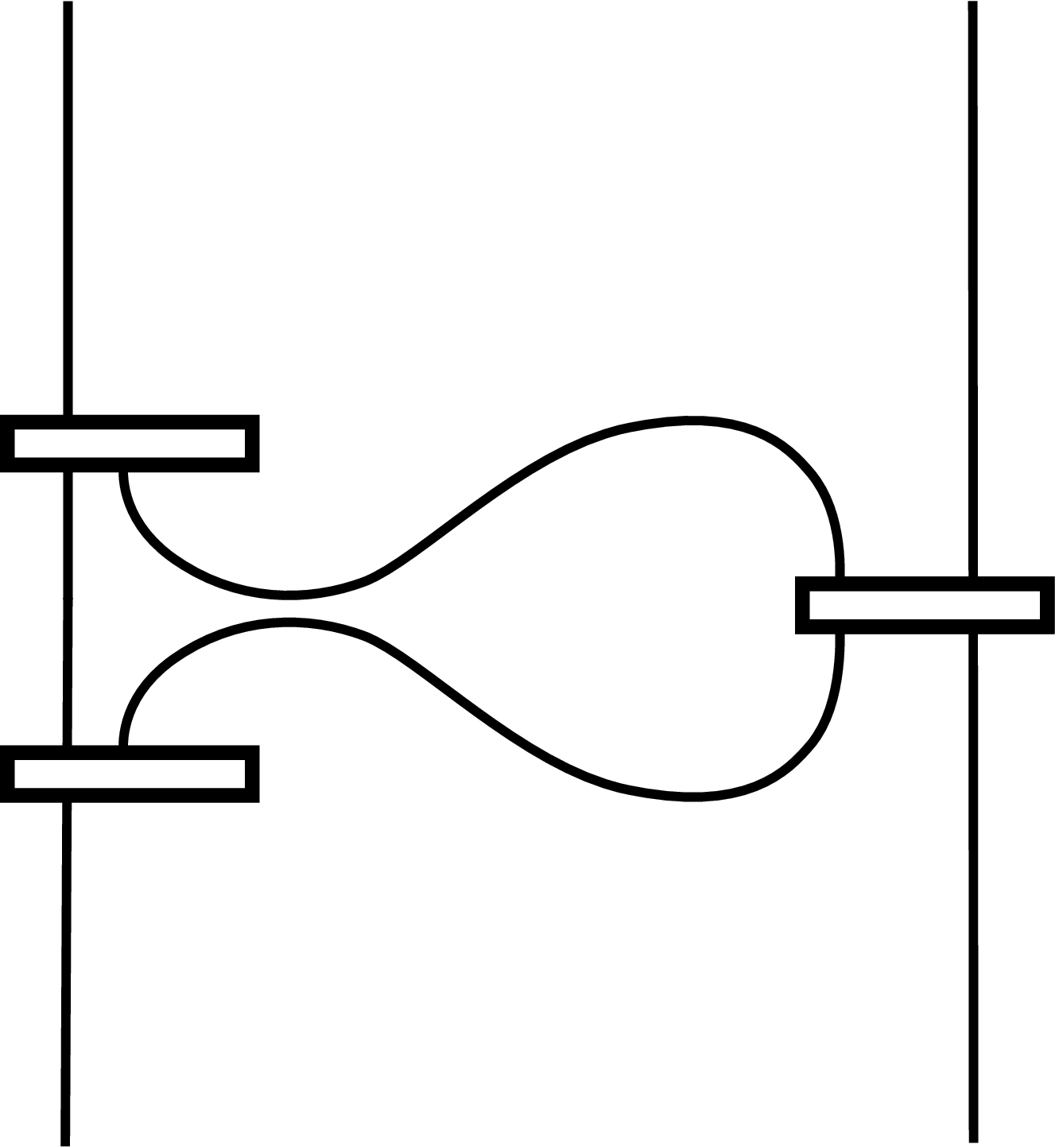}}
        \put(-68,+80){$m$}
          \put(-8,+80){$n$}
          \put(-68,-7){$m^{\prime}$}
          \put(-8,-7){$n^{\prime}$}
          \put(-37,+14){$1$}
         \put(-40,+48){$1$}
   \end{minipage}.
   \end{eqnarray}
}
 Using identity (\ref{properties2}) in the first term and expanding the projector $f^{(n+1)}$ the second term in (\ref{special0}) implies 
{\scriptsize
\begin{eqnarray}
\label{11}
   \begin{minipage}[h]{0.15\linewidth}
        \vspace{0pt}
        \scalebox{0.11}{\includegraphics{second-lemma-main-bubble}}
         \put(-68,+80){$m$}
          \put(-8,+80){$n$}
          \put(-68,-7){$m^{\prime}$}
          \put(-8,-7){$n^{\prime}$}
           \put(-37,+67){$1$}
         \put(-37,+3){$1$}
   \end{minipage}
   =
    \frac{\Delta _{n+1}\Delta _{m}-\Delta _{n}\Delta _{m-1}}{\Delta _{n}\Delta _{m}}
  \begin{minipage}[h]{0.15\linewidth}
        \vspace{0pt}
        \scalebox{0.11}{\includegraphics{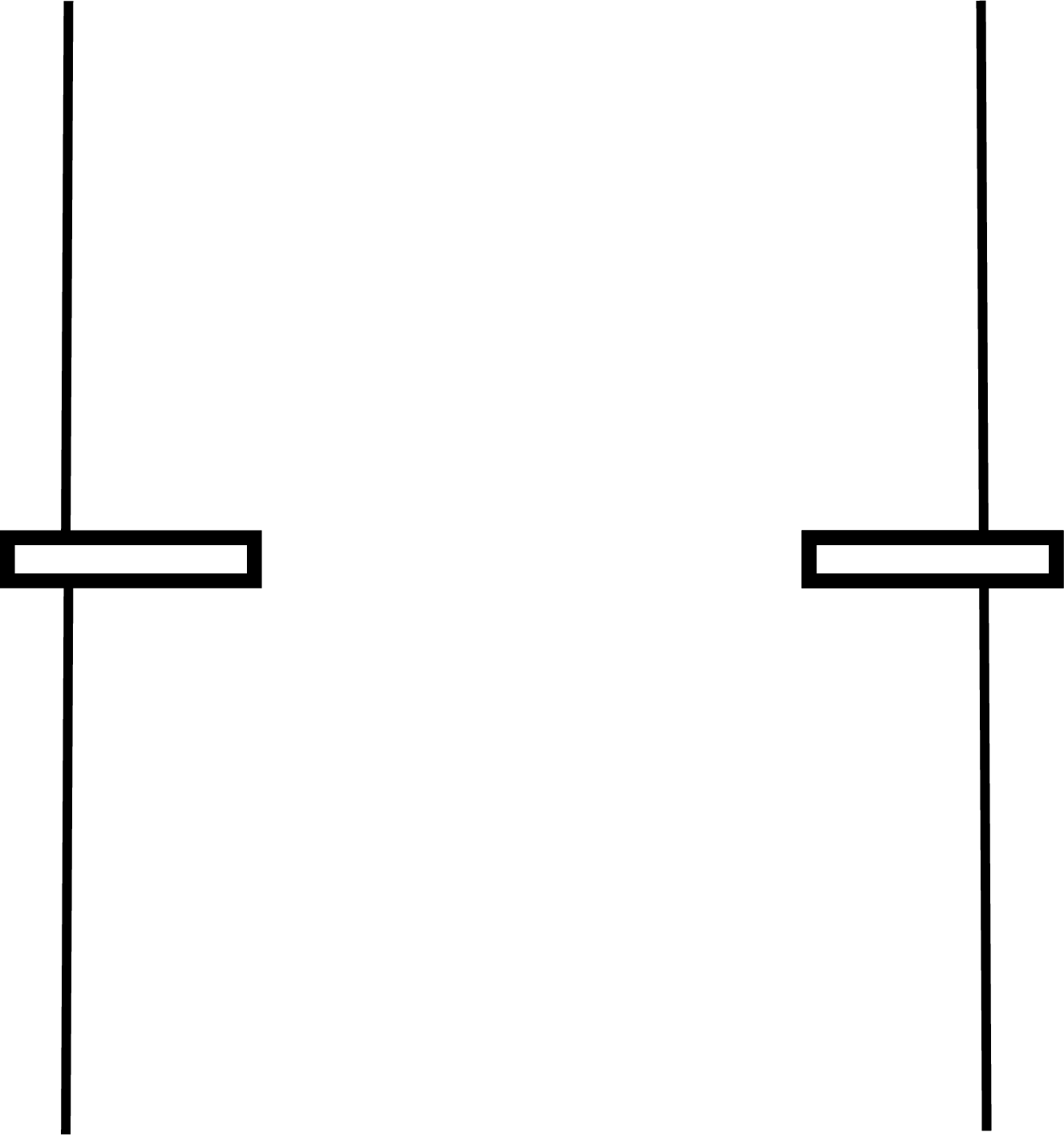}}
        \put(-68,+80){$m$}
          \put(-8,+80){$n$}
          \put(-68,-7){$m^{\prime}$}
          \put(-8,-7){$n^{\prime}$}
          \end{minipage}+
  \frac{\Delta _{m-1}\Delta _{n-1}}{\Delta _{n}\Delta _{m}}
  \begin{minipage}[h]{0.15\linewidth}
        \vspace{0pt}
        \scalebox{0.11}{\includegraphics{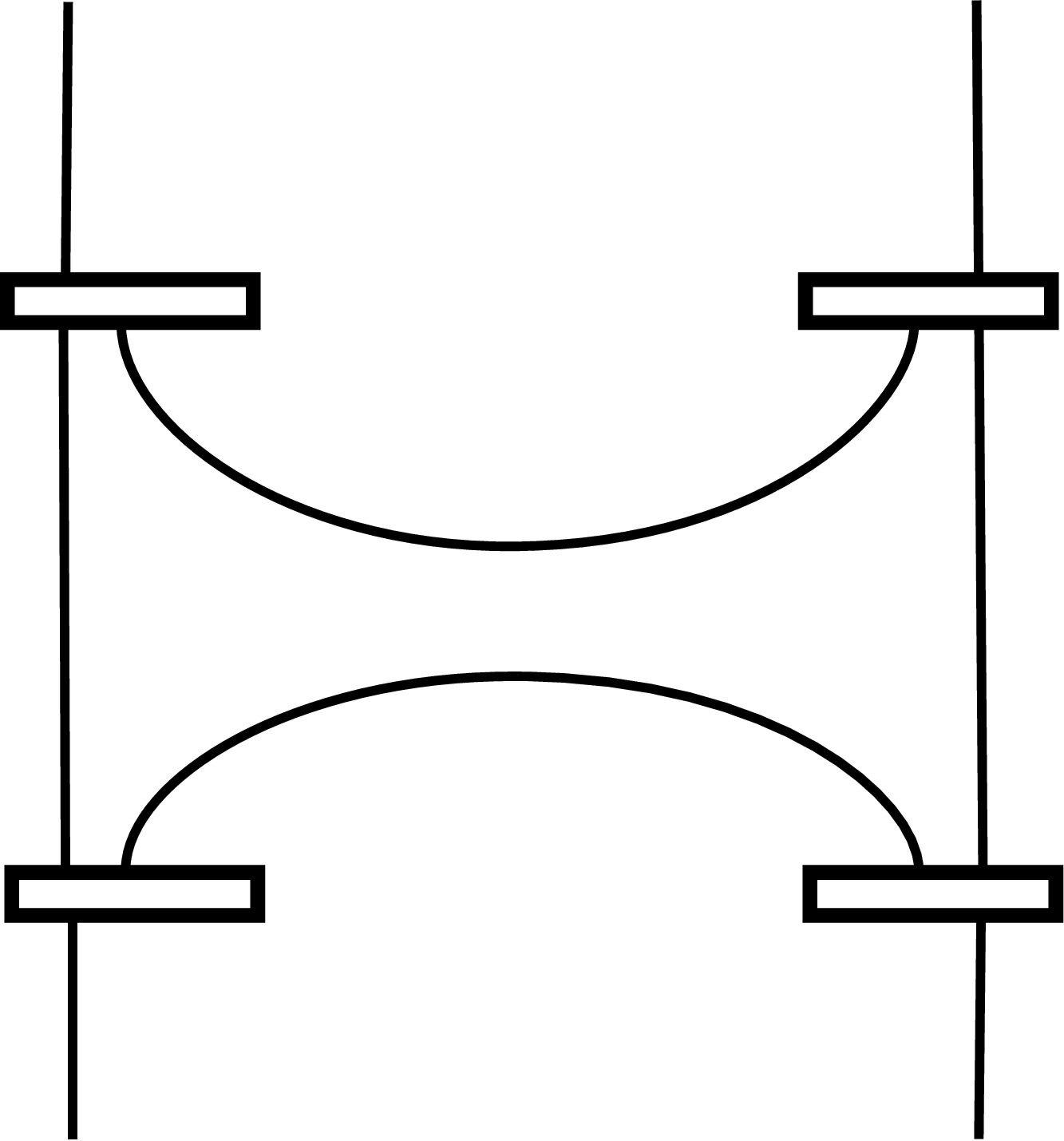}}
         \put(-68,+80){$m$}
          \put(-8,+80){$n$}
          \put(-68,-7){$m^{\prime}$}
          \put(-8,-7){$n^{\prime}$}
         \put(-38,+18){$1$}
         \put(-38,+42){$1$}
   \end{minipage}.
  \end{eqnarray}
}
For $k\geq2$, we use the recursive definition (\ref{recursive}) on the projector $f^{(m+k)}$ that appears in $\mathcal{B}_{m^{\prime},n^{\prime}}^{m,n}(k,1)$. Hence
{\small
\begin{eqnarray}
\label{special}
  \begin{minipage}[h]{0.15\linewidth}
        \vspace{0pt}
        \scalebox{0.11}{\includegraphics{second-lemma-main-bubble}}
         \put(-68,+80){$m$}
          \put(-8,+80){$n$}
          \put(-68,-7){$m^{\prime}$}
          \put(-8,-7){$n^{\prime}$}
           \put(-37,+67){$k$}
         \put(-37,+3){$1$}
   \end{minipage}
   =\hspace{2pt}
    \begin{minipage}[h]{0.15\linewidth}
        \vspace{0pt}
        \scalebox{0.11}{\includegraphics{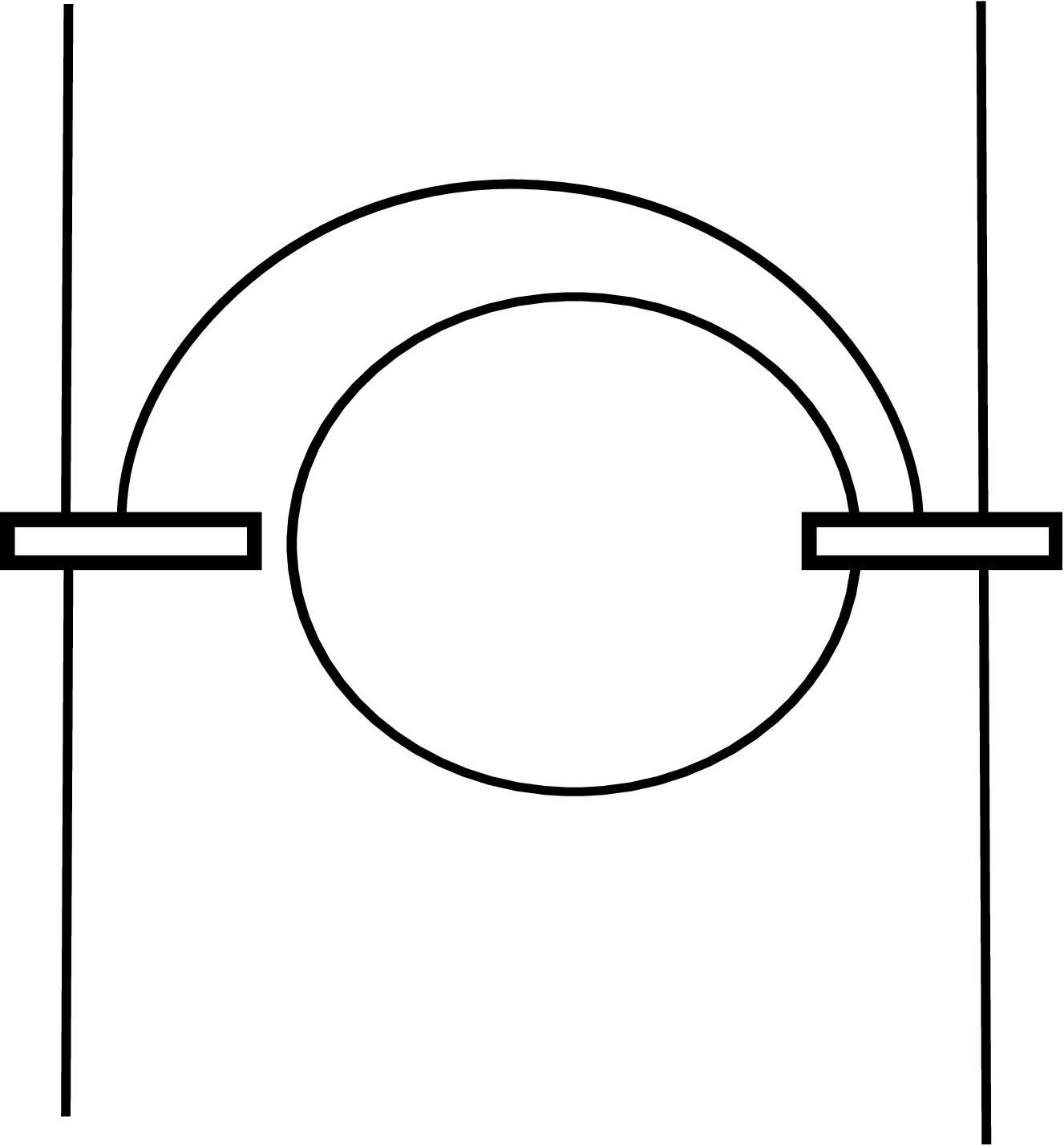}}
         \put(-68,+80){$m$}
          \put(-8,+80){$n$}
          \put(-68,-7){$m^{\prime}$}
          \put(-8,-7){$n^{\prime}$}
           \put(-45,+67){$k-1$}
         \put(-37,+13){$1$}
   \end{minipage}
   -
  \frac{\Delta _{m+k-2}}{\Delta _{m+k-1}}
  \begin{minipage}[h]{0.17\linewidth}
        \vspace{0pt}
        \scalebox{0.11}{\includegraphics{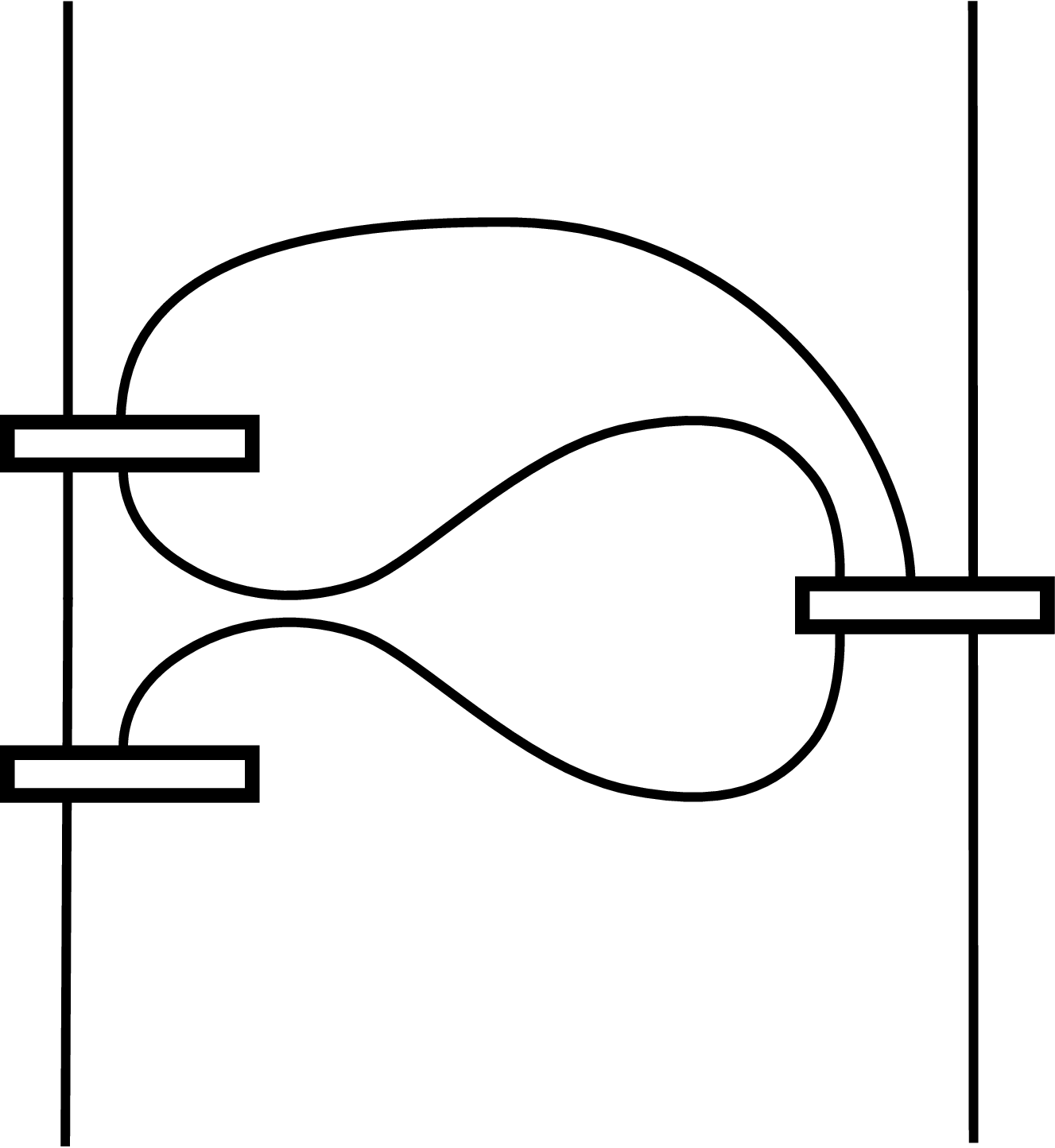}}
        \put(-68,+80){$m$}
          \put(-8,+80){$n$}
          \put(-68,-7){$m^{\prime}$}
          \put(-8,-7){$n^{\prime}$}
           \put(-45,+67){$k-1$}
         \put(-37,+14){$1$}
         \put(-40,+48){$1$}
   \end{minipage}.
   \end{eqnarray}
}
 Using identity (\ref{properties2}) and expanding the projector $f^{(n+k)}$ in (\ref{special}) implies 
{\scriptsize
\begin{eqnarray}
\label{11}
   \begin{minipage}[h]{0.15\linewidth}
        \vspace{0pt}
        \scalebox{0.11}{\includegraphics{second-lemma-main-bubble}}
         \put(-68,+80){$m$}
          \put(-8,+80){$n$}
          \put(-68,-7){$m^{\prime}$}
          \put(-8,-7){$n^{\prime}$}
           \put(-37,+67){$k$}
         \put(-37,+3){$1$}
   \end{minipage}
   =
    \frac{\Delta _{n+k}\Delta _{m+k-1}-\Delta _{n+k-1}\Delta _{m+k-2}}{\Delta _{m+k-1}\Delta _{n+k-1}}
  \begin{minipage}[h]{0.15\linewidth}
        \vspace{0pt}
        \scalebox{0.11}{\includegraphics{bubblefirstterm}}
        \put(-68,+80){$m$}
          \put(-8,+80){$n$}
          \put(-68,-7){$m^{\prime}$}
          \put(-8,-7){$n^{\prime}$}
           \put(-45,+67){$k-1$}
          \end{minipage}+
  \frac{\Delta _{m+k-2}\Delta _{n+k-2}}{\Delta _{m+k-1}\Delta _{n+k-1}}
  \begin{minipage}[h]{0.15\linewidth}
        \vspace{0pt}
        \scalebox{0.11}{\includegraphics{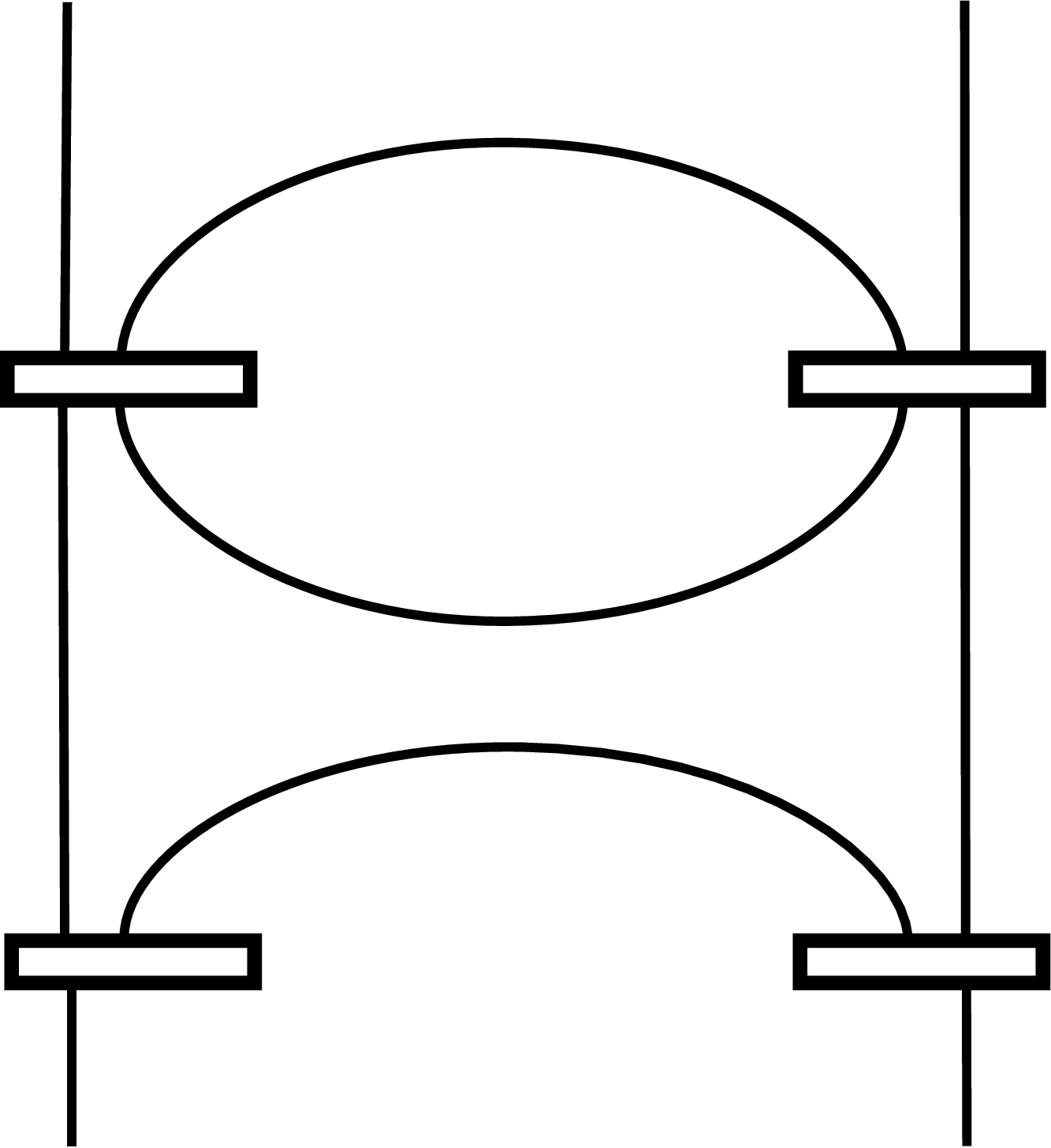}}
         \put(-68,+80){$m$}
          \put(-8,+80){$n$}
          \put(-68,-7){$m^{\prime}$}
          \put(-8,-7){$n^{\prime}$}
           \put(-45,+70){$k-1$}
         \put(-38,+15){$1$}
         \put(-38,+38){$1$}
   \end{minipage}.
  \end{eqnarray}
}
We apply the induction hypothesis on the bubble skein element $\mathcal{B}_{m^{\prime},n^{\prime}}^{m,n}(k-1,1)$ that appears in the second term of the last equation:

\begin{eqnarray*}
  \begin{minipage}[h]{0.15\linewidth}
        \vspace{0pt}
        \scalebox{0.11}{\includegraphics{second-lemma-main-bubble}}
         \put(-68,+80){$m$}
          \put(-8,+80){$n$}
          \put(-68,-7){$m^{\prime}$}
          \put(-8,-7){$n^{\prime}$}
           \put(-37,+67){$k$}
         \put(-37,+3){$1$}
   \end{minipage}
   &=&
     \frac{\Delta _{n+k}\Delta _{m+k-1}-\Delta _{n+k-1}\Delta _{m+k-2}}{\Delta _{m+k-1}\Delta _{n+k-1}}
  \begin{minipage}[h]{0.15\linewidth}
        \vspace{0pt}
        \scalebox{0.11}{\includegraphics{bubblefirstterm}}
        \put(-68,+80){$m$}
          \put(-8,+80){$n$}
          \put(-68,-7){$m^{\prime}$}
          \put(-8,-7){$n^{\prime}$}
           \put(-45,+67){$k-1$}
          \end{minipage}\\
   &+&
  \frac{\Delta _{m+k-2}\Delta _{n+k-2}}{\Delta _{m+k-1}\Delta _{n+k-1}}
  \bigg( \frac{\Delta _{n+k-1}\Delta _{m+k-2}-\Delta _{n}\Delta _{m-1}}{\Delta _{m+k-2}\Delta _{n+k-2}}
  \begin{minipage}[h]{0.15\linewidth}
        \vspace{0pt}
        \scalebox{0.11}{\includegraphics{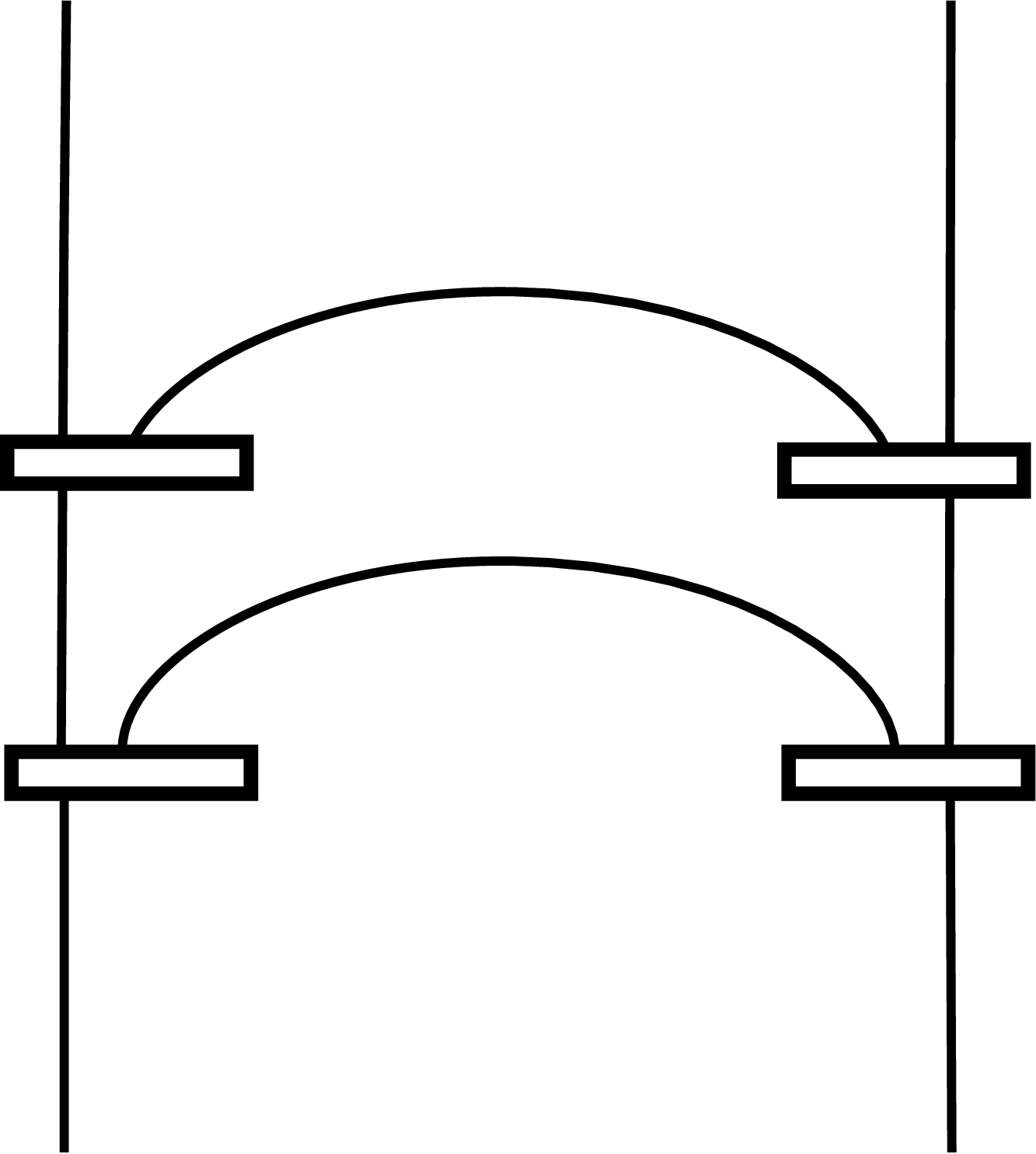}}
         \put(-68,+80){$m$}
          \put(-8,+80){$n$}
          \put(-68,-7){$m^{\prime}$}
          \put(-8,-7){$n^{\prime}$}
           \put(-47,+62){$k-2$}
           \put(-41,+42){$1$}
   \end{minipage}\\
   &+&\frac{\Delta _{m-1}\Delta _{n-1}}{\Delta _{m+k-2}\Delta _{n+k-2}}
   \begin{minipage}[h]{0.16\linewidth}
        \vspace{0pt}
        \scalebox{0.11}{\includegraphics{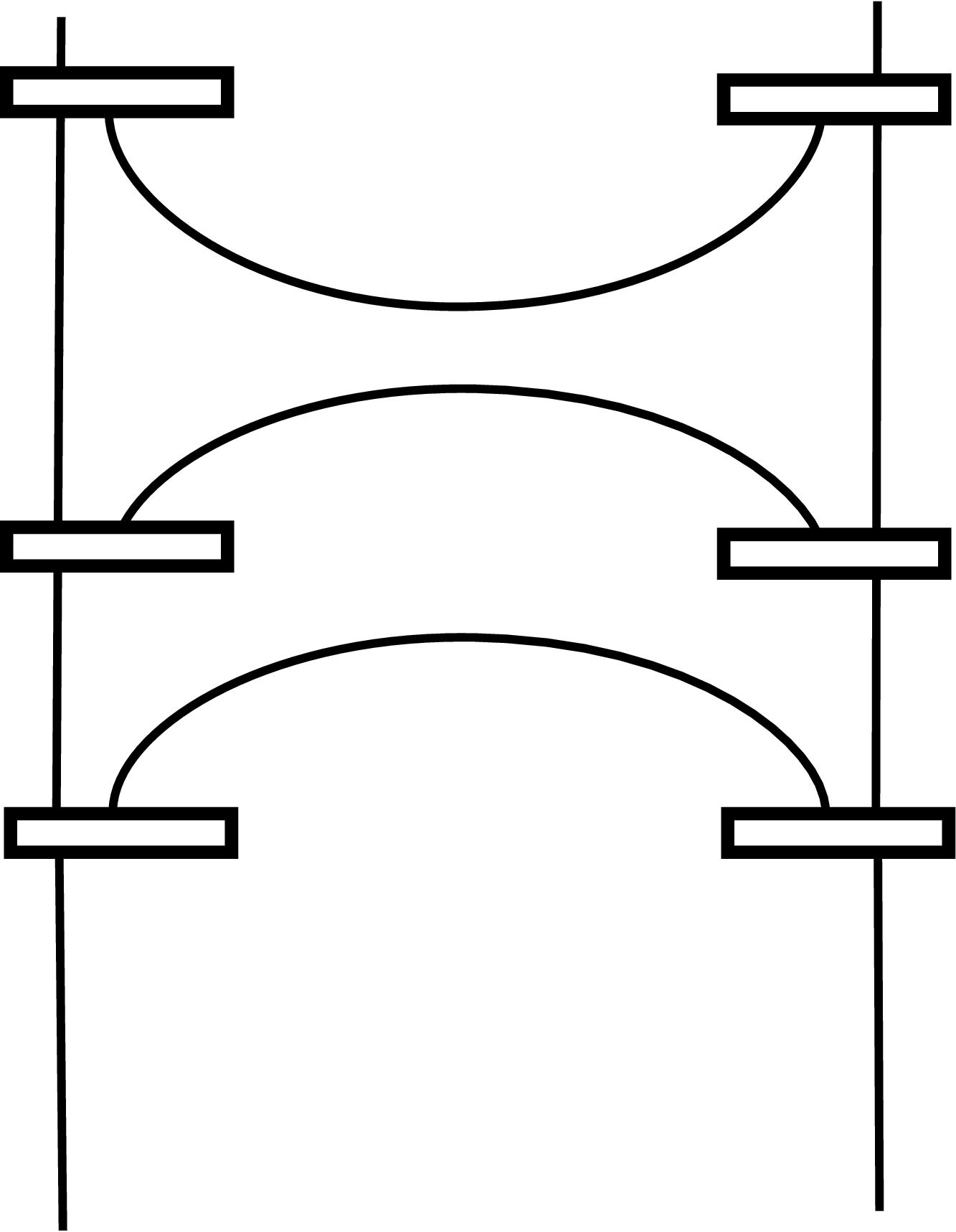}}
         \put(-68,+92){$m$}
          \put(-8,+92){$n$}
          \put(-68,-7){$m^{\prime}$}
          \put(-8,-7){$n^{\prime}$}
           \put(-40,+70){$1$}
         \put(-47,+52){$k-1$}
          \put(-41,+35){$1$}
   \end{minipage}\bigg).\\
   \end{eqnarray*}
   
   Collecting similar terms together, we obtain
   {\small
\begin{eqnarray*}
   \begin{minipage}[h]{0.15\linewidth}
        \vspace{0pt}
        \scalebox{0.11}{\includegraphics{second-lemma-main-bubble}}
         \put(-68,+80){$m$}
          \put(-8,+80){$n$}
          \put(-68,-7){$m^{\prime}$}
          \put(-8,-7){$n^{\prime}$}
           \put(-37,+67){$k$}
         \put(-37,+3){$1$}
   \end{minipage}
   =
     \frac{\Delta
_{n+k}\Delta _{m+k-1}-\Delta _{n}\Delta _{m-1}}{\Delta _{n+k-1}\Delta
_{m+k-1}}
  \begin{minipage}[h]{0.15\linewidth}
        \vspace{0pt}
        \scalebox{0.11}{\includegraphics{bubblefirstterm}}
        \put(-68,+80){$m$}
          \put(-8,+80){$n$}
          \put(-68,-7){$m^{\prime}$}
          \put(-8,-7){$n^{\prime}$}
           \put(-45,+67){$k-1$}
          \end{minipage}
   +
\frac{\Delta _{m-1}\Delta _{n-1}}{\Delta
_{n+k-1}\Delta _{m+k-1}}
  \begin{minipage}[h]{0.15\linewidth}
        \vspace{0pt}
        \scalebox{0.11}{\includegraphics{bubbl-expansion-diagram}}
        \put(-68,+80){$m$}
          \put(-8,+80){$n$}
          \put(-68,-7){$m^{\prime}$}
          \put(-8,-7){$n^{\prime}$}
           \put(-38,+50){$1$}
         \put(-38,28){$k$}
   \end{minipage}.
\end{eqnarray*}
  }

\end{proof}
\begin{remark}
Note that, while the symmetry with respect to the variables $m$ and $n$ in the function $\beta_{m,n}^k$ is clear, the rational function $\alpha_{m,n}^k$ appears to be asymmetric with respect to these variables. It is routine to verify that $\Delta
_{n+k}\Delta _{m+k-1}-\Delta _{n}\Delta _{m-1}=\Delta _{m+k}\Delta _{n+k-1}-\Delta _{m}\Delta _{n-1}$ and hence $\alpha_{m,n}^k$ is also symmetric with respect to the variables $m$ and $n$. This also could be seen to follow from the proof of Lemma \ref{lemma1}. To see this note that in Lemma \ref{lemma1} we obtain (\ref{11}) by expanding $f^{(m+k)}$ first and $f^{(n+k)}$ second. Further, the coefficients $\alpha_{m,n}^k$ and $\beta_{m,n}^k$ are a result of an iterative application of (\ref{11}). We could have done the expansion of the bubble in the opposite order, by expanding $f^{(n+k)}$ first and $f^{(m+k)}$ second, and obtain a version of the identity (\ref{11}) except $m$ and $n$ are swapped. An iteration of this identity would yield the same coefficient $\beta_{m,n}^k$ and forces $\alpha_{m,n}^k$ to be the same as $\alpha_{n,m}^k$.
\end{remark}
We will need the following identity. A proof can be found in \cite{kaufflinks} and \cite{Masbaum1}.
\begin{lemma}
$\Delta _{m+k}\Delta _{n+k-1}-\Delta _{m}\Delta _{n-1}=\Delta _{m+n+k}\Delta
_{k-1}$.
\end{lemma}
Note that the previous identity can be used to obtain a more manageable formula for  $\alpha_{m,n}^k$.

\begin{lemma}
\label{lemma2}
Let $m,n,m^{\prime},n^{\prime} \geq 0$; $k,l \geq 1$ and $k\geq l$. Then we have

\begin{eqnarray}
\label{main lemma}
  \begin{minipage}[h]{0.15\linewidth}
        \vspace{0pt}
        \scalebox{0.11}{\includegraphics{second-lemma-main-bubble}}
         \put(-68,+80){$m$}
          \put(-8,+80){$n$}
          \put(-68,-7){$m^{\prime}$}
          \put(-8,-7){$n^{\prime}$}
           \put(-37,+67){$k$}
         \put(-37,+3){$l$}
   \end{minipage}
   =
     \alpha_{m,n}^k 
  \begin{minipage}[h]{0.15\linewidth}
        \vspace{0pt}
        \scalebox{0.11}{\includegraphics{second-lemma-main-bubble}}
         \put(-68,+80){$m$}
          \put(-8,+80){$n$}
          \put(-68,-7){$m^{\prime}$}
          \put(-8,-7){$n^{\prime}$}
           \put(-45,+67){$k-1$}
         \put(-45,+3){$l-1$}
   \end{minipage}
  +
  \beta_{m,n}^k 
  \begin{minipage}[h]{0.15\linewidth}
        \vspace{0pt}
        \scalebox{0.11}{\includegraphics{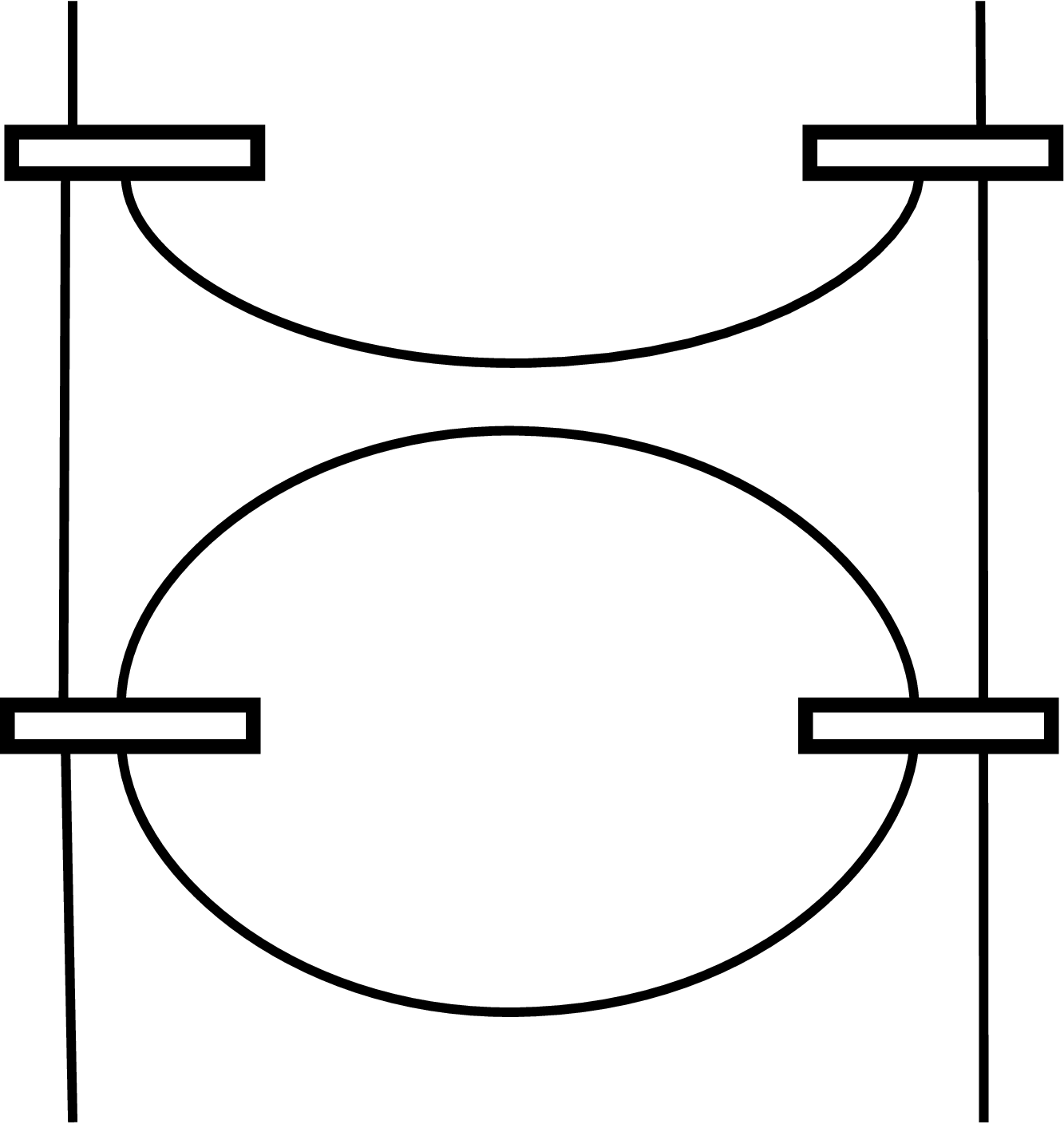}}
        \put(-68,+80){$m$}
          \put(-8,+80){$n$}
          \put(-68,-7){$m^{\prime}$}
          \put(-8,-7){$n^{\prime}$}
           \put(-40,+35){$k$}
           \put(-40,+55){$1$}
         \put(-45,-3){$l-1$}
   \end{minipage}.
\end{eqnarray}

\end{lemma}
\begin{proof}
We start again by considering the trivial cases when one of the integers $m,n,m^{\prime},n^{\prime}$ is zero. Note that, as in Lemma \ref{lemma1}, $m\leq m^{\prime}$ and $n\leq n^{\prime}$. If $m^{\prime}=n^{\prime}=0$, then $\mathcal{B}_{0,0}^{0,0}(k,k)=\alpha_{0,0}^k \Delta _{k-1} =\Delta _{k}$. If $\min(m,n)=0$ or $\min(m^{\prime},n^{\prime})=0$, then the result follows from (\ref{properties2}).
Now suppose that $\min(m,n)\neq 0$ and apply the recursive definition of the Jones-Wenzel idempotent on the projector $f^{(m+k)}$ that appears in $\mathcal{B}_{m^{\prime},n^{\prime}}^{m,n}(k,l)$
\begin{eqnarray}
\label{33}
  \begin{minipage}[h]{0.15\linewidth}
        \vspace{0pt}
        \scalebox{0.11}{\includegraphics{second-lemma-main-bubble}}
         \put(-68,+80){$m$}
          \put(-8,+80){$n$}
          \put(-68,-7){$m^{\prime}$}
          \put(-8,-7){$n^{\prime}$}
           \put(-37,+67){$k$}
         \put(-37,+3){$l$}
   \end{minipage}
   =
    \begin{minipage}[h]{0.15\linewidth}
        \vspace{0pt}
        \scalebox{0.11}{\includegraphics{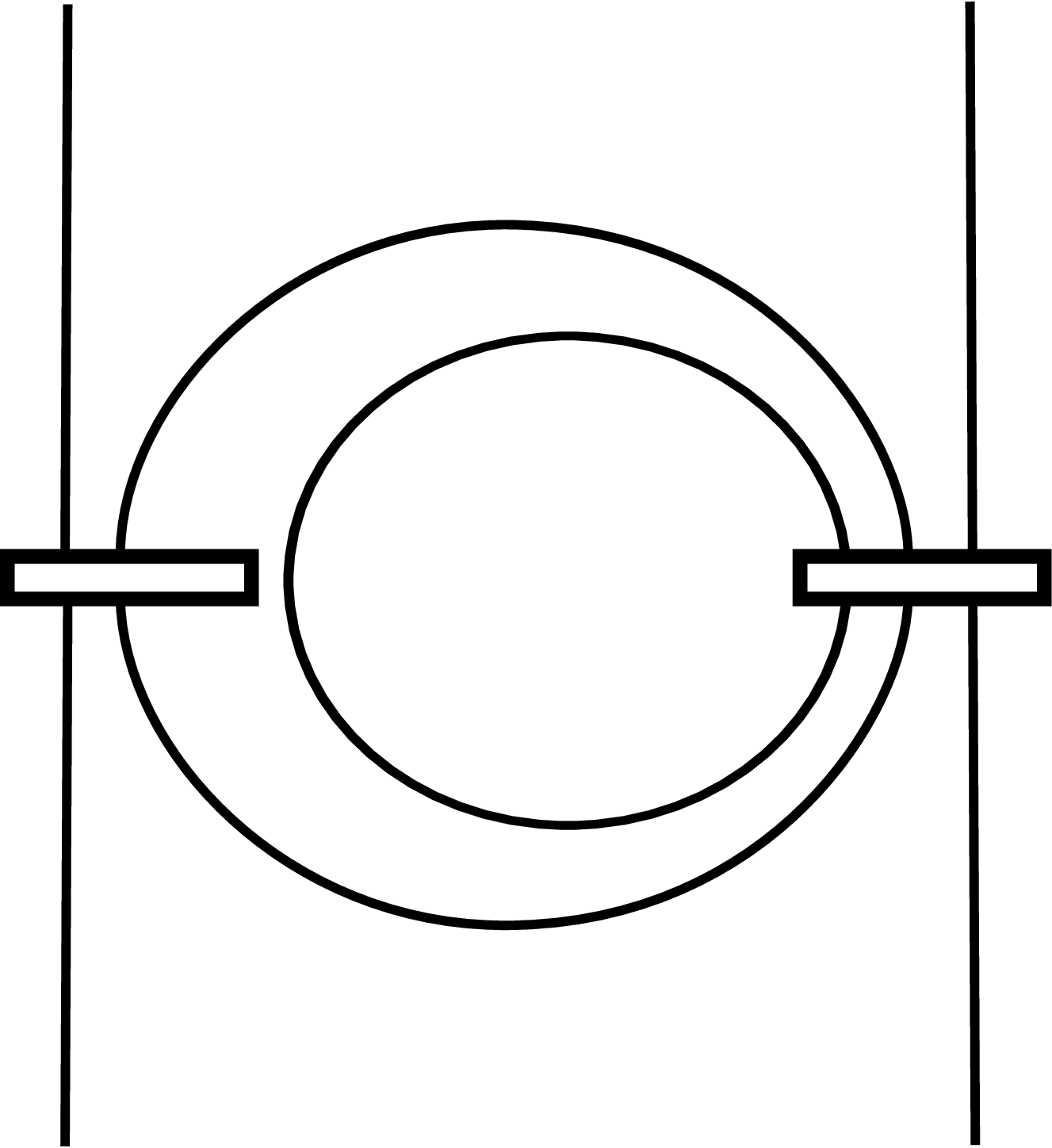}}
         \put(-68,+80){$m$}
          \put(-8,+80){$n$}
          \put(-68,-7){$m^{\prime}$}
          \put(-8,-7){$n^{\prime}$}
           \put(-45,+67){$k-1$}
         \put(-45,+3){$l-1$}
         \put(-43,+42){$1$}
   \end{minipage}
   -
  \frac{\Delta _{m+k-2}}{\Delta _{m+k-1}}
  \begin{minipage}[h]{0.15\linewidth}
        \vspace{0pt}
        \scalebox{0.11}{\includegraphics{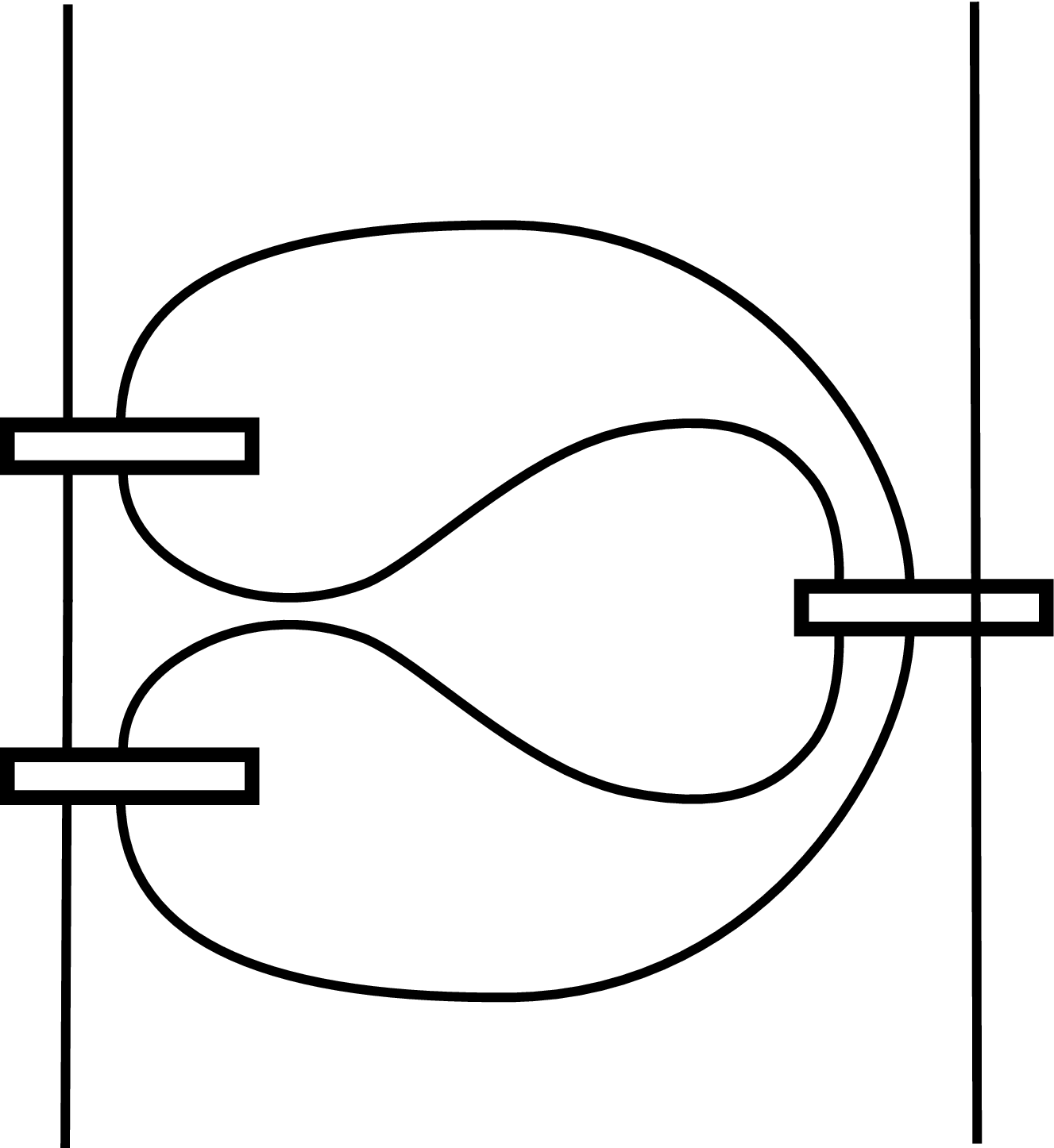}}
         \put(-68,+80){$m$}
          \put(-8,+80){$n$}
          \put(-68,-7){$m^{\prime}$}
          \put(-8,-7){$n^{\prime}$}
           \put(-46,+20){$1$}
            \put(-46,+45){$1$}
           \put(-45,+67){$k-1$}
         \put(-45,+0){$l-1$}
   \end{minipage}.
   \end{eqnarray}
Removing the loop that appears in the first term and expanding the projector $f^{(n+k)}$ in the second term of (\ref{33}) we obtain 
{\small
\begin{eqnarray}
\label{44}
  \begin{minipage}[h]{0.15\linewidth}
        \vspace{0pt}
        \scalebox{0.11}{\includegraphics{second-lemma-main-bubble}}
          \put(-68,+80){$m$}
          \put(-8,+80){$n$}
          \put(-68,-7){$m^{\prime}$}
          \put(-8,-7){$n^{\prime}$}
           \put(-37,+67){$k$}
         \put(-37,+3){$l$}
   \end{minipage}
   =
    \frac{\Delta _{n+k}\Delta _{m+k-1}-\Delta _{n+k-1}\Delta _{m+k-2}}{\Delta _{n+k-1}\Delta _{m+k-1}}
  \begin{minipage}[h]{0.15\linewidth}
        \vspace{0pt}
        \scalebox{0.11}{\includegraphics{second-lemma-main-bubble}}
     \put(-68,+80){$m$}
          \put(-8,+80){$n$}
          \put(-68,-7){$m^{\prime}$}
          \put(-8,-7){$n^{\prime}$}
           \put(-45,+67){$k-1$}
         \put(-45,+3){$l-1$}
          \end{minipage}
  +
  \frac{\Delta _{m+k-2}\Delta _{n+k-2}}{\Delta _{n+k-1}\Delta _{m+k-1}}
  \begin{minipage}[h]{0.15\linewidth}
        \vspace{0pt}
        \scalebox{0.11}{\includegraphics{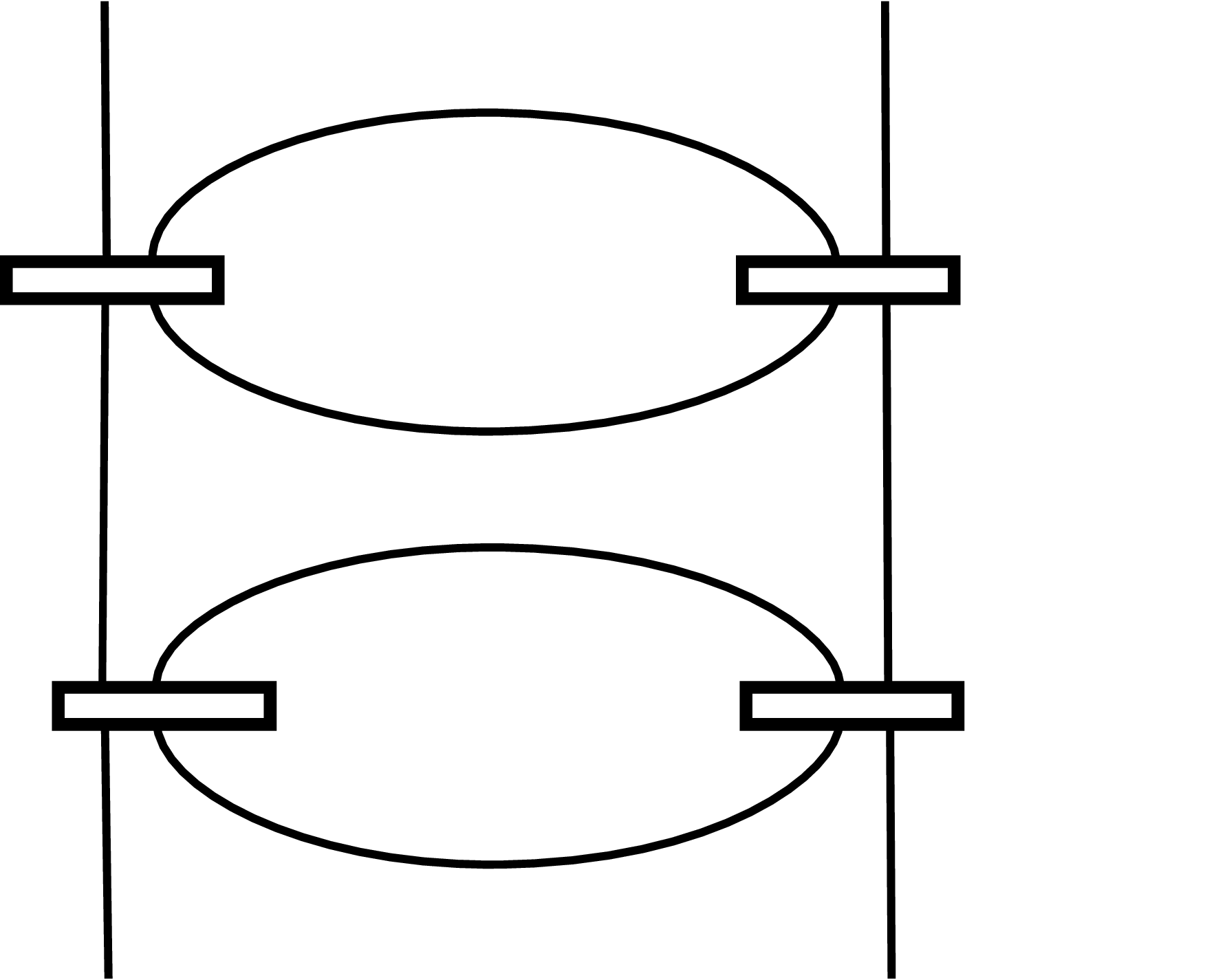}}
         \put(-90,+80){$m$}
          \put(-28,+80){$n$}
          \put(-90,-7){$m^{\prime}$}
          \put(-28,-7){$n^{\prime}$}
           \put(-64,+68){$k-1$}
         \put(-64,0){$l-1$}
          \put(-57,+25){$1$}  
           \put(-57,+44){$1$}
   \end{minipage}.
   \end{eqnarray}
}
If $l=1$ then the result follows from Lemma \ref{lemma1}. Otherwise we apply Lemma \ref{lemma1} to the skein element $\mathcal{B}_{m^{\prime},n^{\prime}}^{m,n}(k-1,1)$ appearing in the second term of (\ref{44}) to obtain

{\footnotesize
\begin{eqnarray*}
  \begin{minipage}[h]{0.15\linewidth}
        \vspace{0pt}
        \scalebox{0.11}{\includegraphics{second-lemma-main-bubble}}
         \put(-68,+80){$m$}
          \put(-8,+80){$n$}
          \put(-68,-7){$m^{\prime}$}
          \put(-8,-7){$n^{\prime}$}
           \put(-37,+67){$k$}
         \put(-37,+3){$l$}
   \end{minipage}
   &=&
     \frac{\Delta _{n+k}\Delta _{m+k-1}-\Delta _{n+k-1}\Delta _{m+k-2}}{\Delta _{n+k-1}\Delta _{m+k-1}}
  \begin{minipage}[h]{0.15\linewidth}
        \vspace{0pt}
        \scalebox{0.11}{\includegraphics{second-lemma-main-bubble}}
         \put(-68,+80){$m$}
          \put(-8,+80){$n$}
          \put(-68,-7){$m^{\prime}$}
          \put(-8,-7){$n^{\prime}$}
           \put(-45,+67){$k-1$}
         \put(-45,+3){$l-1$}
   \end{minipage}\\
   &+&
  \frac{\Delta _{m+k-2}\Delta _{n+k-2}}{\Delta _{n+k-1}\Delta _{m+k-1}}
  \bigg( \frac{\Delta _{n+k-1}\Delta _{m+k-2}-\Delta _{n}\Delta _{m-1}}{\Delta _{n+k-2}\Delta _{m+k-2}}
  \begin{minipage}[h]{0.15\linewidth}
        \vspace{0pt}
        \scalebox{0.11}{\includegraphics{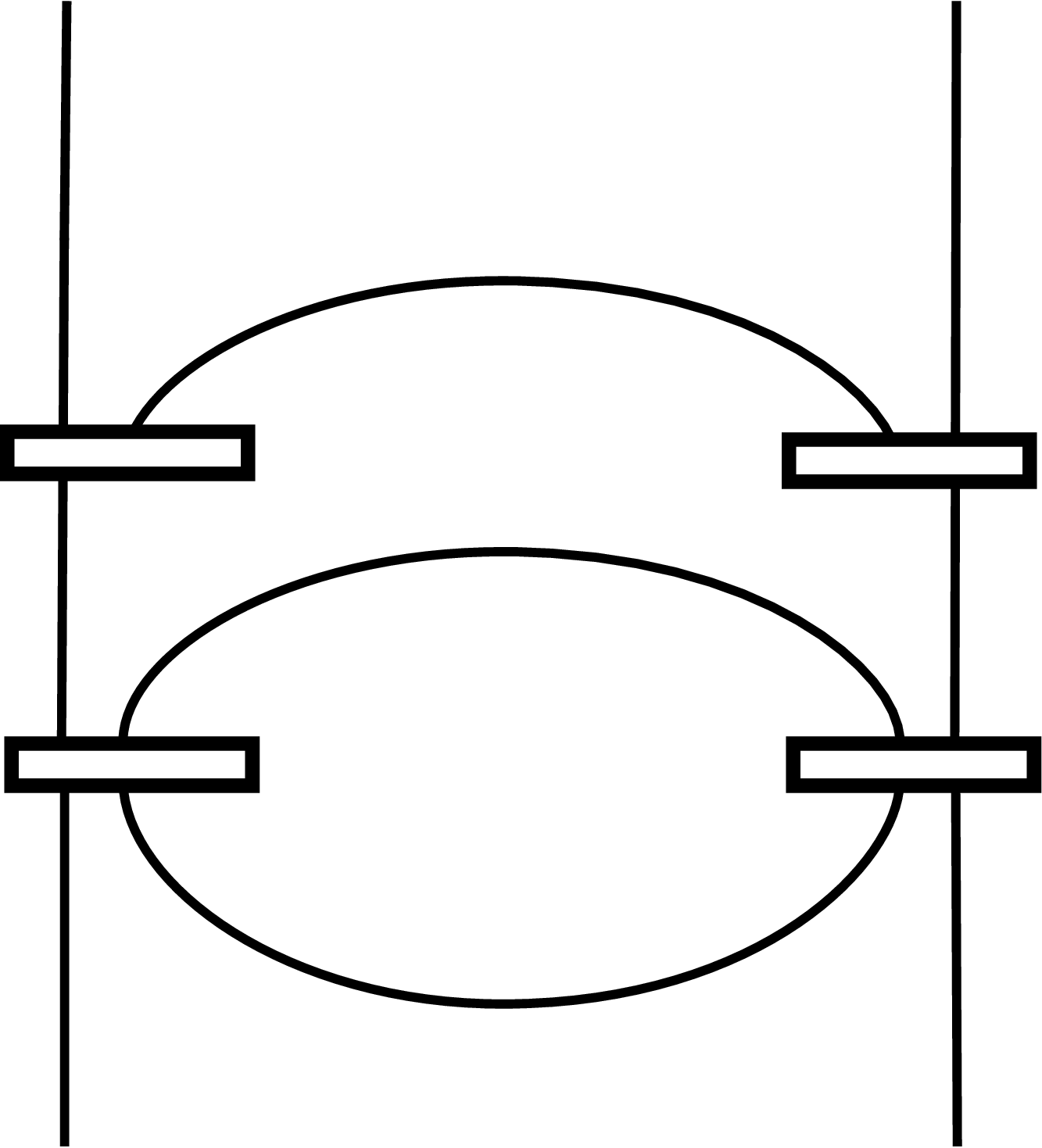}}
         \put(-68,+80){$m$}
          \put(-8,+80){$n$}
          \put(-68,-7){$m^{\prime}$}
          \put(-8,-7){$n^{\prime}$}
           \put(-47,+62){$k-2$}
           \put(-41,+42){$1$}
         \put(-44,1){$l-1$}
   \end{minipage}\\
   &+&\frac{\Delta _{m-1}\Delta _{n-1}}{\Delta _{n+k-2}\Delta _{m+k-2}}
   \begin{minipage}[h]{0.16\linewidth}
        \vspace{0pt}
        \scalebox{0.11}{\includegraphics{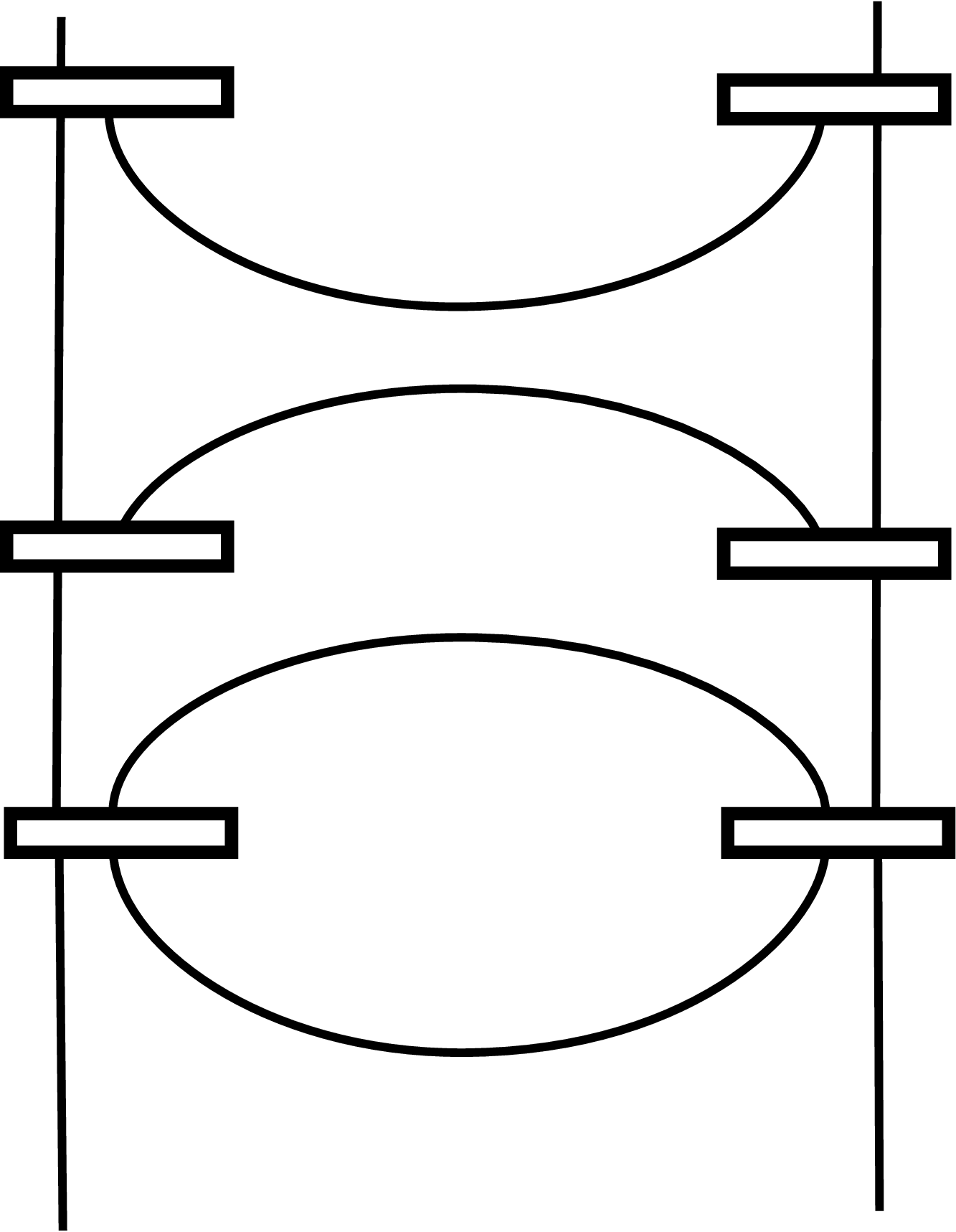}}
         \put(-68,+92){$m$}
          \put(-8,+92){$n$}
          \put(-68,-7){$m^{\prime}$}
          \put(-8,-7){$n^{\prime}$}
           \put(-40,+70){$1$}
         \put(-45,+3){$l-1$}
         \put(-47,+52){$k-1$}
          \put(-41,+35){$1$}
   \end{minipage}\bigg)\\
   &=&\frac{\Delta
_{n+k}\Delta _{m+k-1}-\Delta _{n}\Delta _{m-1}}{\Delta _{n+k-1}\Delta
_{m+k-1}}
   \begin{minipage}[h]{0.15\linewidth}
        \vspace{0pt}
        \scalebox{0.11}{\includegraphics{second-lemma-main-bubble}}
         \put(-68,+80){$m$}
          \put(-8,+80){$n$}
          \put(-68,-7){$m^{\prime}$}
          \put(-8,-7){$n^{\prime}$}
           \put(-45,+67){$k-1$}
         \put(-45,+3){$l-1$}
   \end{minipage}+\frac{\Delta _{m-1}\Delta _{n-1}}{\Delta
_{n+k-1}\Delta _{m+k-1}}
   \begin{minipage}[h]{0.15\linewidth}
        \vspace{0pt}
        \scalebox{0.11}{\includegraphics{second-lemma-main-bubble_secondterm.eps}}
        \put(-68,+80){$m$}
          \put(-8,+80){$n$}
          \put(-68,-7){$m^{\prime}$}
          \put(-8,-7){$n^{\prime}$}
           \put(-40,+35){$k$}
           \put(-40,+55){$1$}
         \put(-45,-3){$l-1$}
   \end{minipage}.
   \end{eqnarray*}
   }
  \end{proof}
\begin{remark}
\label{remark1}
Let $m,n,m^{\prime},n^{\prime} \geq 0$. Without loss of generality assume that $m+n^{\prime} \geq n+m^{\prime}$.
In the space in $\mathscr{D}^{m,n}_{m^{\prime},n^{\prime}}$, there is a one-to-one correspondence between the two set of element shown in Figure \ref{cor}. The correspondence is shown also in the same Figure and it can be parameterized by an integer $i$. Note that the diagonal line in the Figure represents $\frac{1}{2} (m+n^{\prime}-n-m^{\prime})$ parallel lines. The skein elements on the right-hand side of Figure \ref{cor} form a basis for the space $\mathscr{D}^{m,n}_{m^{\prime},n^{\prime}}$. For a proof of this fact see Lemma $14.9$ in \cite{Lickorish1}. On the other hand the skein elements of the left-hand side spans the space $\mathscr{D}^{m,n}_{m^{\prime},n^{\prime}}$. See also the proof of Lemma $14.9$ in \cite{Lickorish1}. One concludes that, by the correspondence shown in Figure \ref{cor} the skein elements on the left-hand side of Figure \ref{cor} must be linearly independent and hence they form a basis for the space $\mathscr{D}^{m,n}_{m^{\prime},n^{\prime}}$.
\begin{figure}[h]
  \centering
   {\includegraphics[scale=0.11]{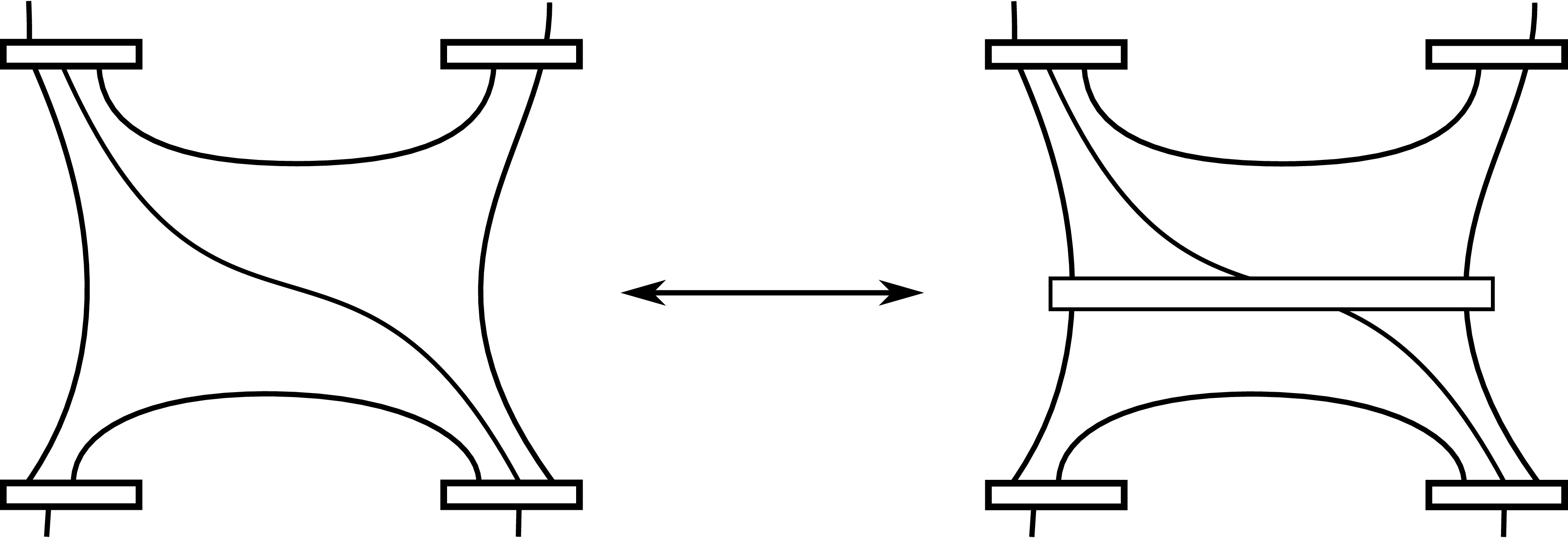}
   		\put(-68,+70){$m$}
          \put(-8,+70){$n$}
          \put(-68,-7){$m^{\prime}$}
          \put(-8,-7){$n^{\prime}$}
           \put(-37,+50){$i$}
           \put(-188,+70){$m$}
          \put(-128,+70){$n$}
          \put(-188,-7){$m^{\prime}$}
          \put(-128,-7){$n^{\prime}$}
           \put(-157,+50){$i$}
           \caption{Correspondence between two bases in the module $\mathscr{D}^{m,n}_{m^{\prime},n^{\prime}}$. }
  \label{cor}}
\end{figure}
\end{remark}
\begin{remark}
\label{remark2}
We are interested in coefficients the bubble skein element $\mathcal{B}_{m^{\prime},n^{\prime}}^{m,n}(k,l)$ in terms of the basis shown on the left-hand side in Figure \ref{cor}. Since $m+k=m^{\prime}+l$ and $n+k=n^{\prime}+l$ then we must have $m+n^{\prime}=n+m^{\prime}$ and hence the diagonal line in the basis shown on the left-hand side in Figure \ref{cor} will not appear when we expand the element $\mathcal{B}_{m^{\prime},n^{\prime}}^{m,n}(k,l)$ in terms of this basis. 
\end{remark}
\section{The bubble expansion formula}
\label{section4}
In this section we will use the recursive formula we obtained in Lemma \ref{lemma2} for $\mathcal{B}_{m^{\prime},n^{\prime}}^{m,n}(k,l)$ to expand this element as a $\mathbb{Q}(A)$-linear sum of certain linearly independent skein elements in $\mathscr{D}^{m,n}_{m^{\prime},n^{\prime}}$. Then we will use Theorem \ref{main} together with Lemma \ref{lemma2} to determine a recursive formula for the coefficients of $\mathcal{B}_{m^{\prime},n^{\prime}}^{m,n}(k,l)$ in terms of these linearly independent elements. Finally, the recursive formula will be used to determine a closed form of these coefficients.

\begin{theorem}
\label{main}
Let $m,n,m^{\prime},n^{\prime} \geq 0$; $k,l \geq 1$. Then\\
(1) For $k\geq l$:{\small
\begin{eqnarray}
\label{bubble expansion formula1}
  \begin{minipage}[h]{0.15\linewidth}
        \vspace{0pt}
        \scalebox{0.11}{\includegraphics{second-lemma-main-bubble}}
        \put(-68,+80){$m$}
          \put(-8,+80){$n$}
          \put(-68,-7){$m^{\prime}$}
          \put(-8,-7){$n^{\prime}$}
           \put(-37,+67){$k$}
         \put(-37,+3){$l$}
   \end{minipage}
   =\displaystyle\sum\limits_{i=0}^{\min(m,n,l)}
   \left\lceil 
\begin{array}{cc}
m & n \\ 
k & l%
\end{array}%
\right\rceil _{i}
    \begin{minipage}[h]{0.15\linewidth}
        \vspace{0pt}
        \scalebox{0.11}{\includegraphics{bubbl-expansion-diagram}}
        \put(-68,+80){$m$}
          \put(-8,+80){$n$}
          \put(-68,-7){$m^{\prime}$}
          \put(-8,-7){$n^{\prime}$}
           \put(-37,+50){$i$}
         \put(-52,28){$k-l+i$}
   \end{minipage}.
  \end{eqnarray}
}
(2) For $l\geq k$:{\small
\begin{eqnarray}
\label{bubble expansion formula2}
  \begin{minipage}[h]{0.15\linewidth}
        \vspace{0pt}
        \scalebox{0.11}{\includegraphics{second-lemma-main-bubble}}
        \put(-68,+80){$m$}
          \put(-8,+80){$n$}
          \put(-68,-7){$m^{\prime}$}
          \put(-8,-7){$n^{\prime}$}
           \put(-37,+67){$k$}
         \put(-37,+3){$l$}
   \end{minipage}
   =\displaystyle\sum\limits_{i=0}^{\min(m^{\prime},n^{\prime},k)}
\left\lceil 
\begin{array}{cc}
n^{\prime} & m^{\prime} \\ 
k & l%
\end{array}%
\right\rceil _{i}        
  \begin{minipage}[h]{0.15\linewidth}
        \vspace{0pt}
        \scalebox{0.11}{\includegraphics{bubbl-expansion-diagram}}
        \put(-68,+80){$m$}
          \put(-8,+80){$n$}
          \put(-68,-7){$m^{\prime}$}
          \put(-8,-7){$n^{\prime}$}
           \put(-52,+35){$l-k+i$}
         \put(-37,12){$i$}
   \end{minipage}.
  \end{eqnarray}
}
where $\left\lceil 
\begin{array}{cc}
m & n \\ 
k & l%
\end{array}%
\right\rceil _{i} :=\left\lceil 
\begin{array}{cc}
m & n \\ 
k & l%
\end{array}%
\right\rceil _{i} (A)$ and $\left\lceil 
\begin{array}{cc}
n^{\prime} & m^{\prime} \\ 
k & l%
\end{array}%
\right\rceil _{i} :=\left\lceil 
\begin{array}{cc}
n^{\prime} & m^{\prime} \\ 
k & l%
\end{array}%
\right\rceil _{i} (A)$ are rational functions.
\end{theorem}
\begin{proof}
(1) The trivial cases when $\min(m,n,l) \leq 1$ were discussed in Lemma \ref{lemma2}. Suppose that $\min(m,n,l) \geq 2$ and consider the identity (\ref{main lemma}). We apply this identity to the bubble skein elements appearing in the first and the second terms of (\ref{main lemma}). We obtain that $\mathcal{B}_{m^{\prime},n^{\prime}}^{m,n}(k,l)$ is equal to a $ \mathbb{Q}(A)$-linear sum of the three skein elements
\begin{eqnarray*}
   \begin{minipage}[h]{0.15\linewidth}
        \vspace{0pt}
        \scalebox{0.11}{\includegraphics{second-lemma-main-bubble}}
         \put(-68,+76){$m$}
          \put(-8,+76){$n$}
          \put(-68,-7){$m^{\prime}$}
          \put(-8,-7){$n^{\prime}$}
           \put(-48,+67){$k-2$}
         \put(-48,+3){$l-2$}
   \end{minipage}
   , \hspace{0 mm}
     \begin{minipage}[h]{0.15\linewidth}
        \vspace{1pt}
        \scalebox{0.11}{\includegraphics{second-lemma-main-bubble_secondterm.eps}}
        \put(-68,+76){$m$}
          \put(-8,+76){$n$}
          \put(-68,-7){$m^{\prime}$}
          \put(-8,-7){$n^{\prime}$}
           \put(-47,+33){$k-1$}
           \put(-40,+55){$1$}
         \put(-45,-3){$l-2$}
   \end{minipage},
  \hspace{ 0 mm}
   \begin{minipage}[h]{0.15\linewidth}
        \vspace{0pt}
        \scalebox{0.11}{\includegraphics{second-lemma-main-bubble_secondterm.eps}}
        \put(-68,+76){$m$}
          \put(-8,+76){$n$}
          \put(-68,-7){$m^{\prime}$}
          \put(-8,-7){$n^{\prime}$}
           \put(-40,+35){$k$}
           \put(-40,+55){$2$}
         \put(-45,-3){$l-2$}
   \end{minipage}.
   \end{eqnarray*}
Note that a bubble skein element appears in each of the three elements above. Moreover, each bubble has an index on at least one of its strands that is less than the index of the corresponding strand in (\ref{main lemma}). One could apply the identity (\ref{main lemma}) iteratively on each bubble that appears in the sum. This iterative process will eventually terminate all bubbles in each element in the summation. This yields that $\mathcal{B}_{m^{\prime},n^{\prime}}^{m,n}(k,l)$ is equal to a $\mathbb{Q}(A)$-linear sum of the $\mathscr{D}^{m,n}_{m^{\prime},n^{\prime}}$ skein elements
   \begin{eqnarray*}
  \begin{minipage}[h]{0.15\linewidth}
        \vspace{0pt}
        \scalebox{0.11}{\includegraphics{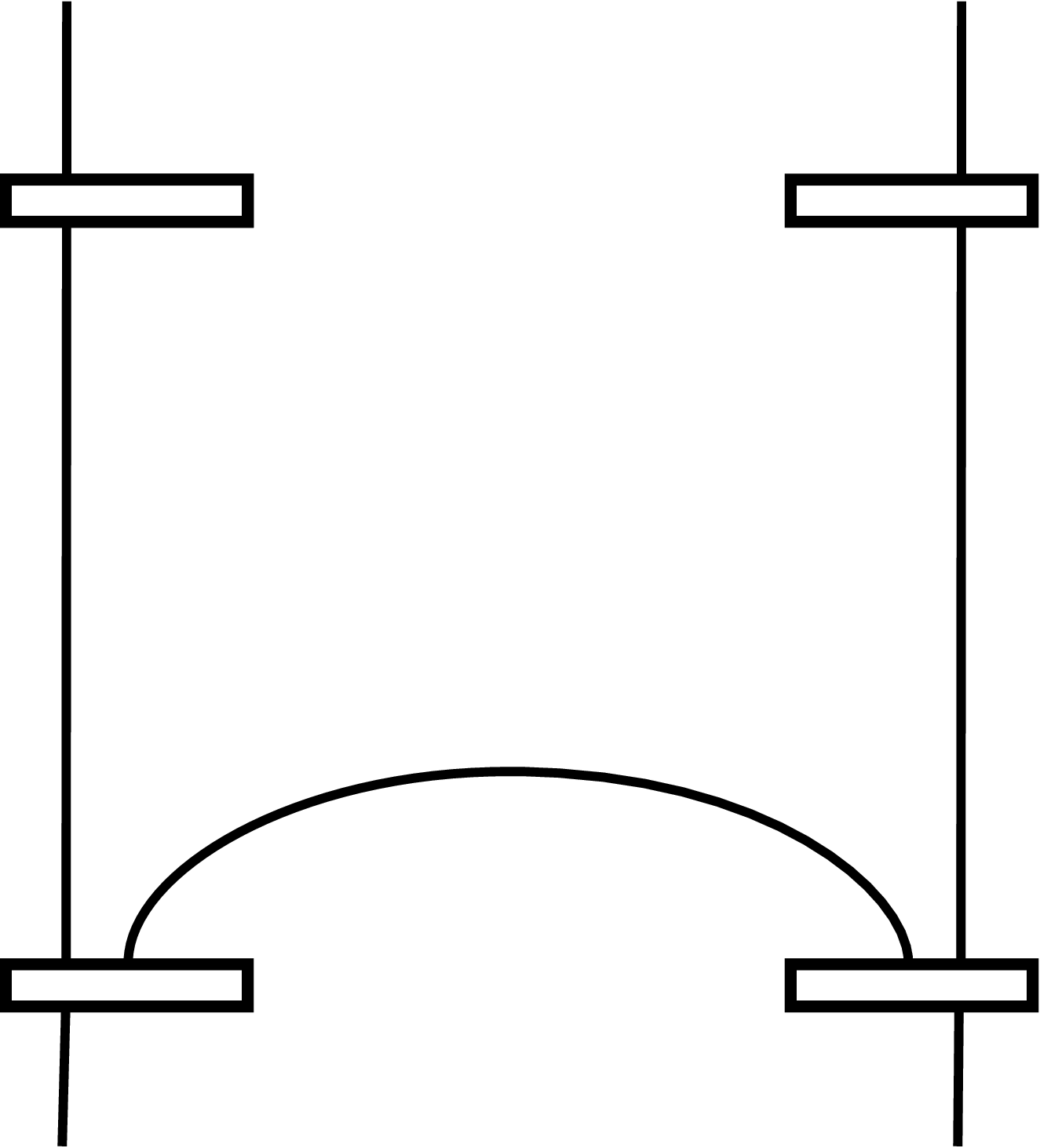}}
        \put(-68,+80){$m$}
          \put(-8,+80){$n$}
          \put(-68,-7){$m^{\prime}$}
          \put(-8,-7){$n^{\prime}$}
           \put(-45,+28){$k-l$}
          \end{minipage}
   ,\hspace{2pt}
 \begin{minipage}[h]{0.15\linewidth}
        \vspace{0pt}
        \scalebox{0.11}{\includegraphics{bubbl-expansion-diagram}}
        \put(-68,+80){$m$}
          \put(-8,+80){$n$}
          \put(-68,-7){$m^{\prime}$}
          \put(-8,-7){$n^{\prime}$}
           \put(-55,+28){$k-l+1$}
         \put(-37,47){$1$}
   \end{minipage}
   ,...\hspace{2pt},
   \begin{minipage}[h]{0.15\linewidth}
        \vspace{0pt}
        \scalebox{0.11}{\includegraphics{bubbl-expansion-diagram}}
        \put(-68,+80){$m$}
          \put(-8,+80){$n$}
          \put(-68,-7){$m^{\prime}$}
          \put(-8,-7){$n^{\prime}$}
           \put(-55,+28){$k-l+s$}
         \put(-37,47){$s$}
   \end{minipage}.
  \end{eqnarray*}
where $s= \min(m,n,l)$. Now the result follows by noticing that the variables $m,n,m^{\prime},n^{\prime},k$ and $l$ are related by the equations $m+k=m^{\prime}+l$ and $ n+k=n^{\prime}+l$ and hence it is sufficient to index the coefficients in the expansion of $\mathcal{B}_{m^{\prime},n^{\prime}}^{m,n}(k,l)$ in terms of the previous skein elements by five indices. (2) Follows from (1).
\end{proof}
Now we determine the coefficients $\left\lceil 
\begin{array}{cc}
m & n \\ 
k & l%
\end{array}%
\right\rceil _{i}$. 
\begin{proposition}
Let $m,n,m^{\prime},n^{\prime} \geq 0$, $k,l \geq 1$. Let $\mathcal{B}_{m^{\prime},n^{\prime}}^{m,n}(k,l)$  be a bubble skein element in $\mathscr{D}^{m,n}_{m^{\prime},n^{\prime}}$ such that $k \geq l$. Then, for $0\leq i \leq \min(m,n,l)$, the rational function  $   \left\lceil 
\begin{array}{cc}
m & n \\ 
k & l%
\end{array}%
\right\rceil _{i}$ satisfies the following recursive identity:

 \begin{equation}
   \label{coefrecur}
   \left\lceil 
\begin{array}{cc}
m & n \\ 
k & l%
\end{array}%
\right\rceil _{i}=\alpha_{m,n}^k\times\left\lceil 
\begin{array}{cc}
m & n \\ 
k-1 & l-1%
\end{array}%
\right\rceil _{i}
+ \beta_{m,n}^k\times
\left\lceil 
\begin{array}{cc}
m-1 & n-1 \\ 
k & l-1%
\end{array}%
\right\rceil _{i-1}.
   \end{equation}
\end{proposition}

\begin{proof}
Substitute (\ref{bubble expansion formula1}) in both sides of (\ref{main lemma}): 
{\small

\begin{eqnarray*}
 \displaystyle\sum\limits_{i=0}^{\min(m,n,l)}
     \left\lceil 
\begin{array}{cc}
m & n \\ 
k & l%
\end{array}%
\right\rceil _{i}
   \begin{minipage}[h]{0.15\linewidth}
        \vspace{0pt}
        \scalebox{0.11}{\includegraphics{bubbl-expansion-diagram}}
        \put(-68,+80){$m$}
          \put(-8,+80){$n$}
          \put(-68,-7){$m^{\prime}$}
          \put(-8,-7){$n^{\prime}$}
           \put(-37,+50){$i$}
         \put(-52,28){$k-l+i$}
   \end{minipage}&=&\alpha_{m,n}^k\displaystyle\sum\limits_{i=0}^{\min(m,n,l-1)}
     \left\lceil 
\begin{array}{cc}
m & n \\ 
k-1 & l-1%
\end{array}%
\right\rceil _{i}
   \begin{minipage}[h]{0.15\linewidth}
        \vspace{0pt}
        \scalebox{0.11}{\includegraphics{bubbl-expansion-diagram}}
        \put(-68,+80){$m$}
          \put(-8,+80){$n$}
          \put(-68,-7){$m^{\prime}$}
          \put(-8,-7){$n^{\prime}$}
           \put(-37,+50){$i$}
         \put(-52,28){$k-l+i$}
   \end{minipage}\\&+&\beta_{m,n}^k\displaystyle\sum\limits_{i=0}^{\min(m-1,n-1,l-1)}
     \left\lceil 
\begin{array}{cc}
m-1 & n-1 \\ 
k & l-1%
\end{array}%
\right\rceil _{i}
  \begin{minipage}[h]{0.15\linewidth}
        \vspace{0pt}
        \scalebox{0.11}{\includegraphics{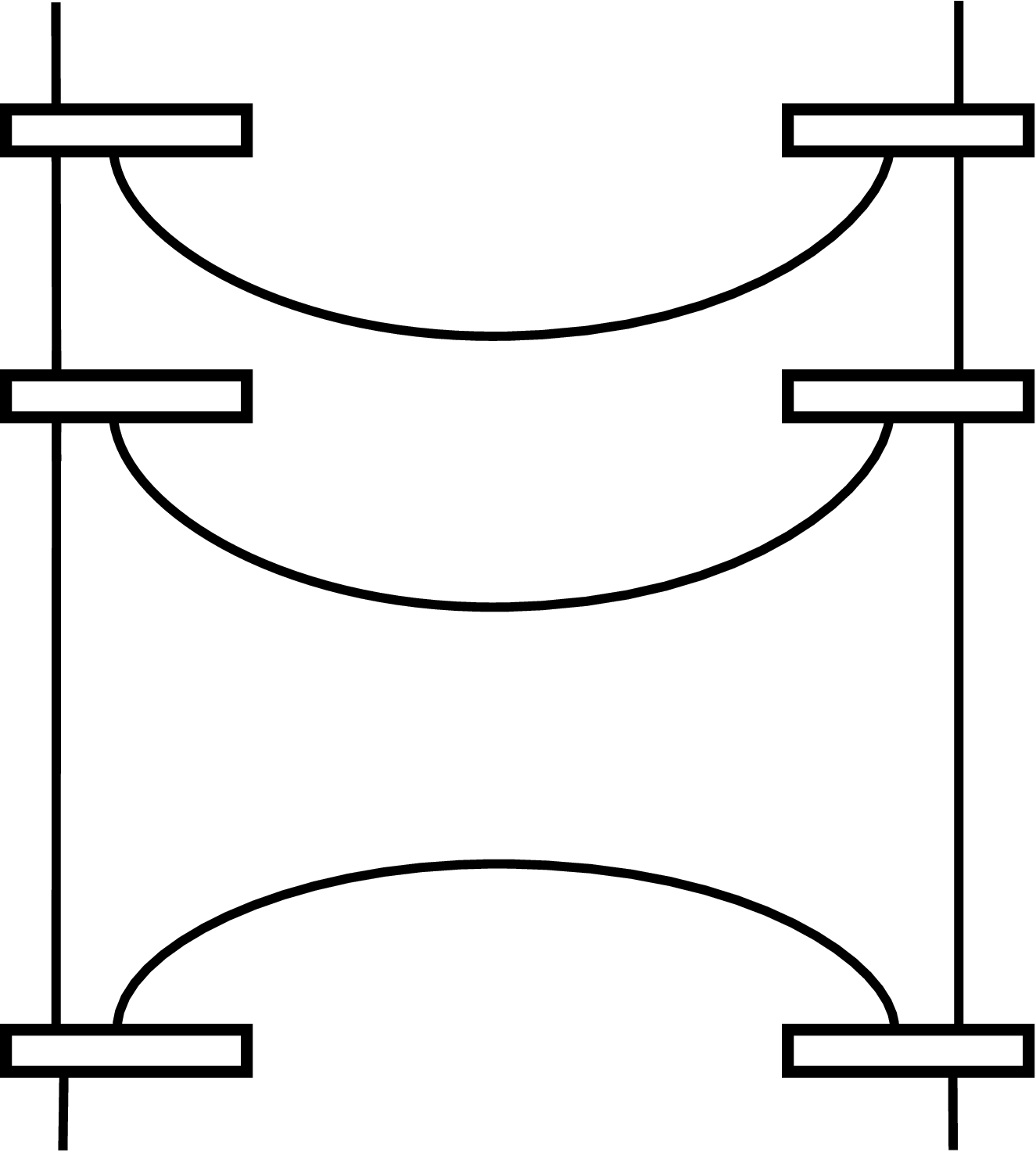}}
        \put(-68,+80){$m$}
          \put(-8,+80){$n$}
          \put(-68,-7){$m^{\prime}$}
          \put(-8,-7){$n^{\prime}$}
           \put(-37,+40){$i$}
         \put(-61,22){$k-l+i+1$}
          \put(-37,+58){$1$}
   \end{minipage}.
  \end{eqnarray*}
}
Hence
{\small
\begin{eqnarray*}
 \displaystyle\sum\limits_{i=0}^{\min(m,n,l)}
   \left\lceil 
\begin{array}{cc}
m & n \\ 
k & l%
\end{array}%
\right\rceil _{i}
   \begin{minipage}[h]{0.15\linewidth}
        \vspace{0pt}
        \scalebox{0.11}{\includegraphics{bubbl-expansion-diagram}}
        \put(-68,+80){$m$}
          \put(-8,+80){$n$}
          \put(-68,-7){$m^{\prime}$}
          \put(-8,-7){$n^{\prime}$}
           \put(-37,+50){$i$}
         \put(-52,28){$k-l+i$}
   \end{minipage}&=&\displaystyle\sum\limits_{i=0}^{\min(m,n,l-1)}
     \alpha_{m,n}^k\times
     \left\lceil 
\begin{array}{cc}
m & n \\ 
k-1 & l-1%
\end{array}%
\right\rceil _{i}
  \begin{minipage}[h]{0.15\linewidth}
        \vspace{0pt}
        \scalebox{0.11}{\includegraphics{bubbl-expansion-diagram}}
        \put(-68,+80){$m$}
          \put(-8,+80){$n$}
          \put(-68,-7){$m^{\prime}$}
          \put(-8,-7){$n^{\prime}$}
           \put(-37,+50){$i$}
         \put(-52,28){$k-l+i$}
   \end{minipage}\\&+&\displaystyle\sum\limits_{i=0}^{\min(m-1,n-1,l-1)}
     \beta_{m,n}^k\times\left\lceil 
\begin{array}{cc}
m-1 & n-1 \\ 
k & l-1%
\end{array}%
\right\rceil _{i}
 \begin{minipage}[h]{0.15\linewidth}
        \vspace{0pt}
        \scalebox{0.11}{\includegraphics{bubbl-expansion-diagram}}
        \put(-68,+80){$m$}
          \put(-8,+80){$n$}
          \put(-68,-7){$m^{\prime}$}
          \put(-8,-7){$n^{\prime}$}
           \put(-44,+50){$i+1$}
         \put(-61,28){$k-l+i+1$}
   \end{minipage}.
  \end{eqnarray*}
}
The latter can be written as
{\small
\begin{eqnarray*}
 \displaystyle\sum\limits_{i=0}^{\min(m,n,l)}
    \left\lceil 
\begin{array}{cc}
m & n \\ 
k & l%
\end{array}%
\right\rceil _{i}
   \begin{minipage}[h]{0.15\linewidth}
        \vspace{0pt}
        \scalebox{0.11}{\includegraphics{bubbl-expansion-diagram}}
        \put(-68,+80){$m$}
          \put(-8,+80){$n$}
          \put(-68,-7){$m^{\prime}$}
          \put(-8,-7){$n^{\prime}$}
           \put(-37,+50){$i$}
         \put(-52,28){$k-l+i$}
   \end{minipage}&=&\displaystyle\sum\limits_{i=0}^{\min(m,n,l-1)}
     \alpha_{m,n}^k\times
     \left\lceil 
\begin{array}{cc}
m & n \\ 
k-1 & l-1%
\end{array}%
\right\rceil _{i}
   \begin{minipage}[h]{0.15\linewidth}
        \vspace{0pt}
        \scalebox{0.11}{\includegraphics{bubbl-expansion-diagram}}
        \put(-68,+80){$m$}
          \put(-8,+80){$n$}
          \put(-68,-7){$m^{\prime}$}
          \put(-8,-7){$n^{\prime}$}
           \put(-37,+50){$i$}
         \put(-52,28){$k-l+i$}
   \end{minipage}\\&+&\displaystyle\sum\limits_{i=1}^{\min(m,n,l)}
     \beta_{m,n}^k\times
     \left\lceil 
\begin{array}{cc}
m-1 & n-1 \\ 
k & l-1%
\end{array}%
\right\rceil _{i-1}
   \begin{minipage}[h]{0.15\linewidth}
        \vspace{0pt}
        \scalebox{0.11}{\includegraphics{bubbl-expansion-diagram}}
        \put(-68,+80){$m$}
          \put(-8,+80){$n$}
          \put(-68,-7){$m^{\prime}$}
          \put(-8,-7){$n^{\prime}$}
           \put(-37,+50){$i$}
         \put(-52,28){$k-l+i$}
   \end{minipage}.
  \end{eqnarray*}
}
        Now note that the elements \begin{minipage}[h]{0.15\linewidth}
        \vspace{0pt}
        \scalebox{0.11}{\includegraphics{bubbl-expansion-diagram}}
        \put(-68,+80){$m$}
          \put(-8,+80){$n$}
          \put(-68,-7){$m^{\prime}$}
          \put(-8,-7){$n^{\prime}$}
           \put(-37,+50){$i$}
         \put(-52,28){$k-l+i$}
   \end{minipage}, where $0\leq i\leq \min(m,n,l)$, are linearly independent in the module $\mathscr{D}^{m,n}_{m^{\prime},n^{\prime}}$ by Remarks \ref{remark1} and \ref{remark2}. Hence we conclude that 
   \begin{equation*}
   \left\lceil 
\begin{array}{cc}
m & n \\ 
k & l%
\end{array}%
\right\rceil _{i}=\alpha_{m,n}^k\times\left\lceil 
\begin{array}{cc}
m & n \\ 
k-1 & l-1%
\end{array}%
\right\rceil _{i}
+ \beta_{m,n}^k\times
\left\lceil 
\begin{array}{cc}
m-1 & n-1 \\ 
k & l-1%
\end{array}%
\right\rceil _{i-1}.
   \end{equation*}
   
\end{proof}
\begin{remark}
The coefficients $\left\lceil 
\begin{array}{cc}
m & n \\ 
k & l%
\end{array}%
\right\rceil _{i}$ behave like the binomial coefficients ${l \choose i}$ in the sense that $\left\lceil 
\begin{array}{cc}
m & n \\ 
k & l%
\end{array}%
\right\rceil _{i}=0$ when $i<0$ or $i>l$. Note also that the recursion formula for $\left\lceil 
\begin{array}{cc}
m & n \\ 
k & l%
\end{array}%
\right\rceil _{i}$, when one focuses the variables $l$ and $i$, is analogues to the recursion formula of the binomial coefficients ${l \choose i}$. 
\end{remark}
The following theorem gives a closed formula for the rational function $\left\lceil 
\begin{array}{cc}
m & n \\ 
k & l%
\end{array}%
\right\rceil _{i}$.
\begin{theorem}
Let $m,n,k,l \geq 0$, and $k \geq l$ and $0\leq i \leq \min(m,n,l)$. Then
\begin{equation}
\left\lceil 
\begin{array}{cc}
m & n \\ 
k & l%
\end{array}%
\right\rceil _{i}=(-A^2)^{i(i-l)}\frac{\displaystyle\prod_{j=0}^{l-i-1}\Delta
_{k-j-1}\prod_{s=0}^{i-1}\Delta _{n-s-1}\Delta _{m-s-1}}{%
\displaystyle\prod_{t=0}^{l-1}\Delta _{n+k-t-1}\Delta _{m+k-t-1}}{l \brack i}_{A^4}\prod_{j=0}^{l-i-1}\Delta_{m+n+k-i-j}.
\end{equation}
\end{theorem}
\begin{proof}
This equation agrees with the recursion identity (\ref{coefrecur}).
\end{proof}
The symmetry of the bubble skein element, or the previous formula for the coefficients $\left\lceil 
\begin{array}{cc}
m & n \\ 
k & l%
\end{array}%
\right\rceil _{i}$, implies immediately the following.
\begin{proposition}
Let $m,n,k,l \geq 0$. Let $k \geq l$. Then 
\begin{equation}
\left\lceil 
\begin{array}{cc}
m & n \\ 
k & l%
\end{array}%
\right\rceil _{i}=\left\lceil 
\begin{array}{cc}
n & m \\ 
k & l%
\end{array}%
\right\rceil _{i}.
\end{equation}
\end{proposition}
It is preferred sometimes to write the the coefficients $\left\lceil 
\begin{array}{cc}
m & n \\ 
k & l%
\end{array}%
\right\rceil _{i}$ in terms of quantum integers rather than deltas.
Recall that $\Delta _{n}=(-1)^{n}[n+1]$ and hence the sign of the term$\displaystyle
\prod_{j=0}^{l-i-1}[k-j]$ can be easily calculated to be $(-1)^{-(\frac{1}{%
2})(i-l)(-1+i+2k-l).}$ Similarly, the sign of $\displaystyle\prod_{s=0}^{i-1}[n-s]$
is $(-1)^{-(\frac{1}{2})i(1+i-2n)},$ the sign of $\displaystyle\prod_{t=0}^{l-1}[n+k-t]$
is $(-1)^{-(\frac{1}{2})l(1-2k+l-2n)}$, and the sign of $\displaystyle\prod_{j=0}^{l-i-1}[m+n+k-i-j+1]$ is $(-1)^{-(\frac{1}{2})(i-l)(1+i+2k-l+2m+2n)}$. Thus the exponent of $-1$ of the whole term is $
-i-2ik-l+2il+4kl-2l^2+2lm+2ln=i+l\pmod 2$. Hence the previous theorem can be rewritten as follows:

\begin{corollary}
\label{coeffff}
Let $m,n,k,l \geq 0$. Let $k \geq l$ and $0\leq i \leq \min(m,n,l)$. Then\\
\begin{equation*}
\left\lceil 
\begin{array}{cc}
m & n \\ 
k & l%
\end{array}%
\right\rceil _{i}=(-1)^{i+l}a^{i(i-l)}\frac{\displaystyle\prod_{j=0}^{l-i-1}[k-j]\prod_{s=0}^{i-1}[n-s][m-s]}{\displaystyle\prod_{t=0}^{l-1}[n+k-t][m+k-t]}
{l \brack i}_{a^2}\prod_{j=0}^{l-i-1}[m+n+k-i-j+1].
\end{equation*}
\end{corollary}

\section{Applications}
\label{section5}
\subsection{The theta graph}
A theta graph is a spin network in $S^2$ that plays an important role in computing arbitrary spin network in $S^2$. The evaluation of a theta graph in $\mathcal{S}(S^2)$ is equivalent to find the evaluation of the skein element in Figure \ref{theta}.
\begin{figure}[htb]
  \centering
   {\includegraphics[scale=0.10]{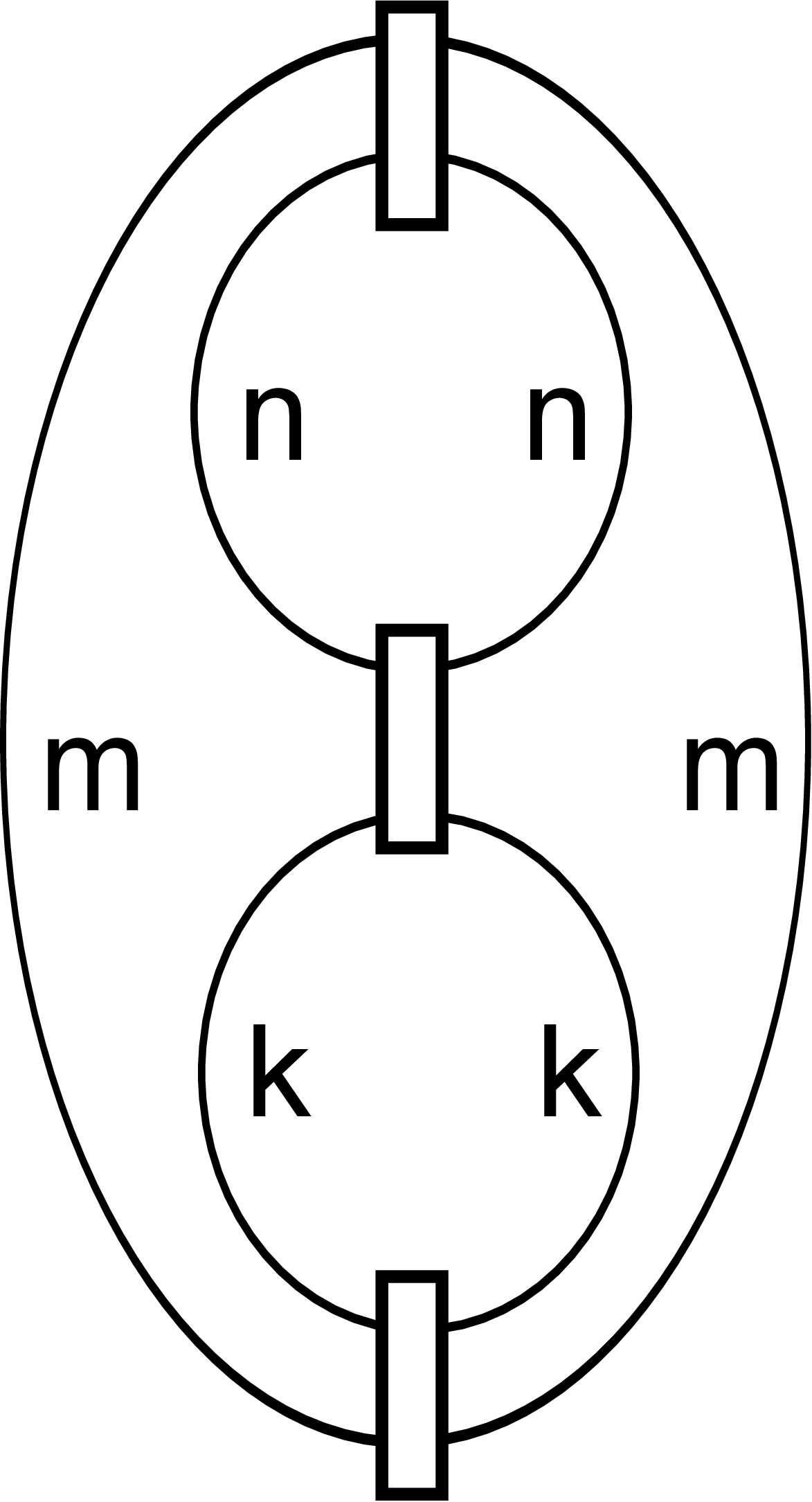}
  \caption{The skein element $\Lambda(m,n,k)$}
  \label{theta}}
\end{figure}
Explicit determination of this skein element is done in \cite{kaufflinks} and \cite{Masbaum1}. We will denote this skein element by $\Lambda(m,n,k)$.
Our first immediate application of the bubble expansion formula is computing this skein element. The following lemma shows that $\left\lceil 
\begin{array}{cc}
m & n \\ 
k & k%
\end{array}%
\right\rceil _{0}$ is almost equal to the evaluation of this skein element.
\begin{lemma}
\label{thetag}
\begin{equation}
  \begin{minipage}[h]{0.125\linewidth}
       \scalebox{0.10}{\includegraphics{theta1}}
   \end{minipage}
   =\left\lceil 
\begin{array}{cc}
m & n \\ 
k & k%
\end{array}%
\right\rceil _{0}\Delta_{m+n}.
   \end{equation}
\end{lemma}
\begin{proof} We apply the bubble skein formula (\ref{bubble expansion formula1}) on that bubble $\mathcal{B}_{m,n}^{m,n}(k,k)$ appears in Figure \ref{theta}:
\begin{align*}
  \begin{minipage}[h]{0.12\linewidth}
        \vspace{0pt}
        \scalebox{0.10}{\includegraphics{theta1}}
   \end{minipage}
   =\displaystyle\sum\limits_{i=0}^{\min(m,n,k)}
     \left\lceil 
\begin{array}{cc}
m & n \\ 
k & k%
\end{array}%
\right\rceil _{i}\hspace{3pt}
  \begin{minipage}[h]{0.12\linewidth}
        \vspace{0pt}
        \scalebox{0.10}{\includegraphics{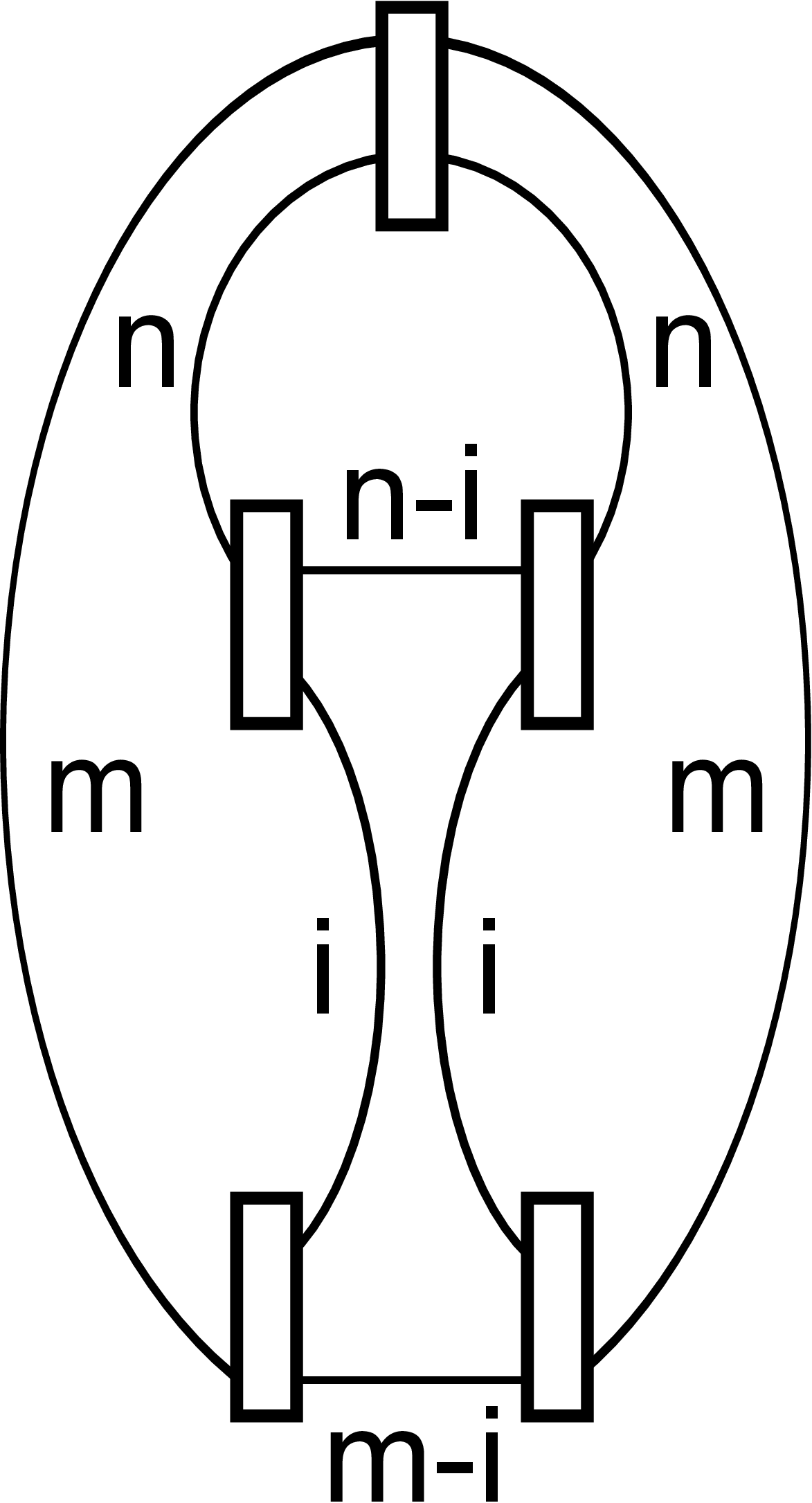}}
   \end{minipage}.
   \end{align*}
 The previous summation is zero except when $i=0$ and hence it reduces to
 \begin{align*}
 \left\lceil 
\begin{array}{cc}
m & n \\ 
k & k%
\end{array}%
\right\rceil _{0}\hspace{2pt}
  \begin{minipage}[h]{0.12\linewidth}
        \vspace{0pt}
        \scalebox{0.10}{\includegraphics{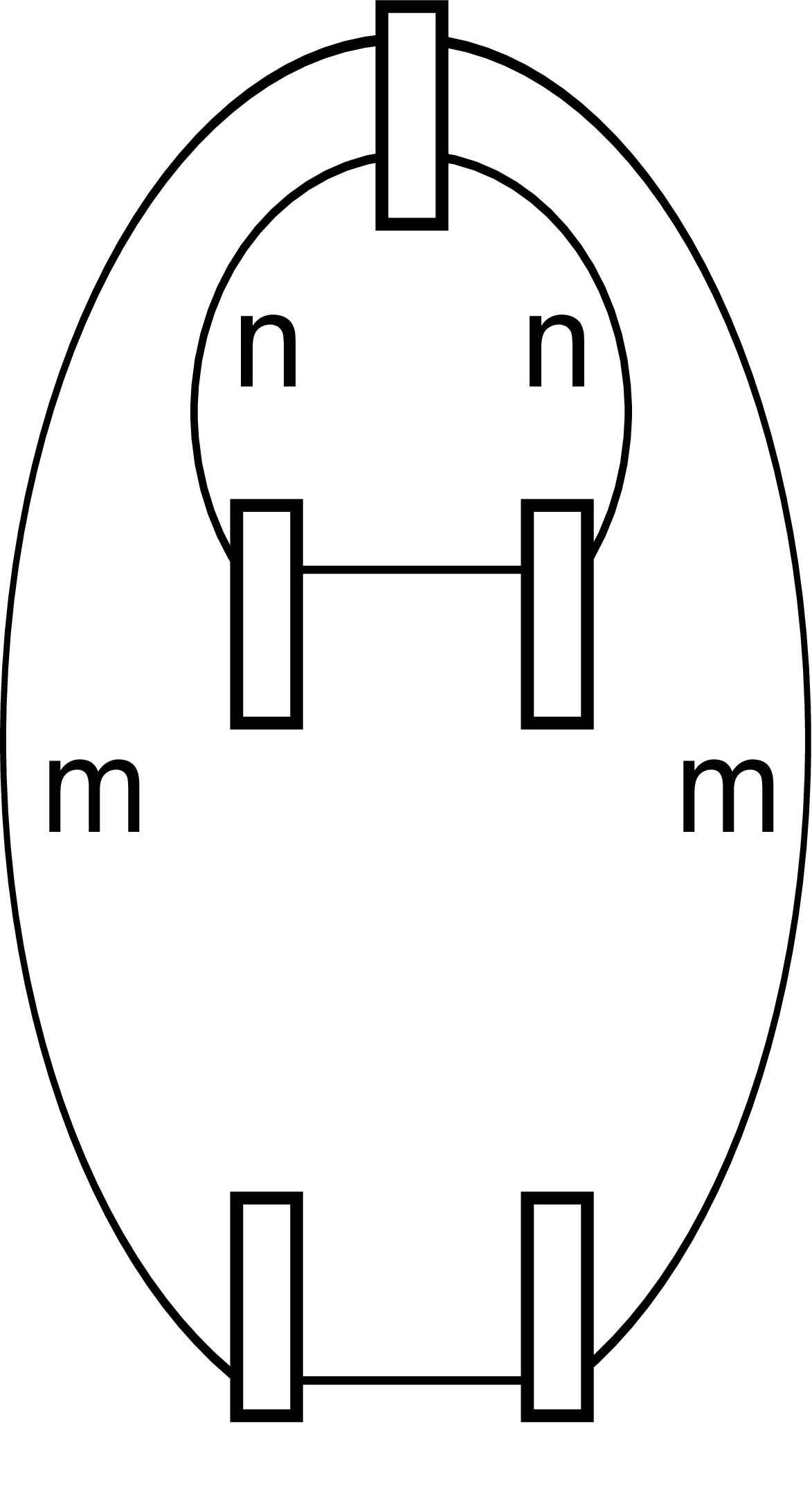}}
   \end{minipage}
   =\left\lceil 
\begin{array}{cc}
m & n \\ 
k & k%
\end{array}%
\right\rceil _{0}\Delta_{m+n}.
    \end{align*}
\end{proof}
\begin{remark}
In the previous lemma one could apply the bubble skein formula on the other bubbles in the skein element \ref{theta} and obtain
\begin{align*}
  \begin{minipage}[h]{0.12\linewidth}
        \vspace{0pt}
        \scalebox{0.10}{\includegraphics{theta1}}
   \end{minipage}
   =\left\lceil 
\begin{array}{cc}
m & n \\ 
k & k%
\end{array}%
\right\rceil _{0}\Delta_{m+n}=\left\lceil 
\begin{array}{cc}
m & k \\ 
n & n%
\end{array}%
\right\rceil _{0}\Delta_{m+k}=\left\lceil 
\begin{array}{cc}
n & k \\ 
m & m%
\end{array}%
\right\rceil _{0}\Delta_{n+k}.
   \end{align*}
 This reflects the symmetry of the theta graph. 
\end{remark}
\subsection{The head and the tail of alternating knots}
 We briefly review the basics of the head and the tail of the colored Jones polynomial. For more details see \cite{Cody} and \cite{Cody2}. Then we use the bubble expansion formula to give an example of calculation that demonstrates how one can use the bubble expansion formula to study certain skein elements in $\mathcal{S}(S^3)$ that determines the tail of a family of knots.  \\
\\ We start by recalling the definition of the unreduced colored Jones polynomial for links in $S^3$. Let $L$ be a framed link in $S^3$. Decorate every component of $L$, according to its framing, by the $n^{th}$ Jones-Wenzl idempotent and consider this decorated framed link as an element of $\mathcal{S}(S^3)$. Up to a power of $\pm A$, that depends on the framing of $L$, the value of this element is the $n^{th}$ (unreduced) colored Jones polynomial $\tilde{J}_{n,L}(A)$. In what follows we assume $A^4=q$.\\

Following \cite{Cody2}, write $P_1(q)\doteq_n P_2(q)$ ,where $P_1(q)$, $P_2(q)$ are two Laurent series if and only if $P_1(q)=q^{\pm s}P_2(q)\mod q^n$. It was proven in \cite{Cody} that the $k^{th}$ coefficient in the entire colored Jones polynomial of an alternating link $L$ does not depend on the color $n$ as long as $n \geq k$. Thus one could define a power series associated with the entire colored Jones polynomial whose $n^{th}$ coefficient is the $n^{th}$ coefficient of $\tilde{J}_{n,L}$. The tail of the colored Jones polynomial of a link $L$ is defined to be a series
$T_L(q)$, that satisfies  $T_L(q)\doteq_{n}\frac{\tilde{J}_{n,L}(q)}{\Delta_n(q)}$ for all $n$. In the same way, the head of the colored Jones polynomial of a link $L$, denoted by $H_L$, is defined to be the tail of $\frac{\tilde{J}_{n,L}(q^{-1}}{\Delta_n(q^{-1}})$. For more details see \cite{Cody}. Higher order stability of the coefficients of the colored Jones polynomial is studied by Garoufalidis and Le in \cite{klb}.\\\\ Let $L$ be a link in $S^3$ and $D$ be an alternating knot diagram of $L$. For any crossing in $D$ there are two ways to smooth this crossing, the $A$-smoothing and the $B$-smoothing (see Figure \ref{smoothings}).
\begin{figure}[H]
  \centering
   {\includegraphics[scale=0.14]{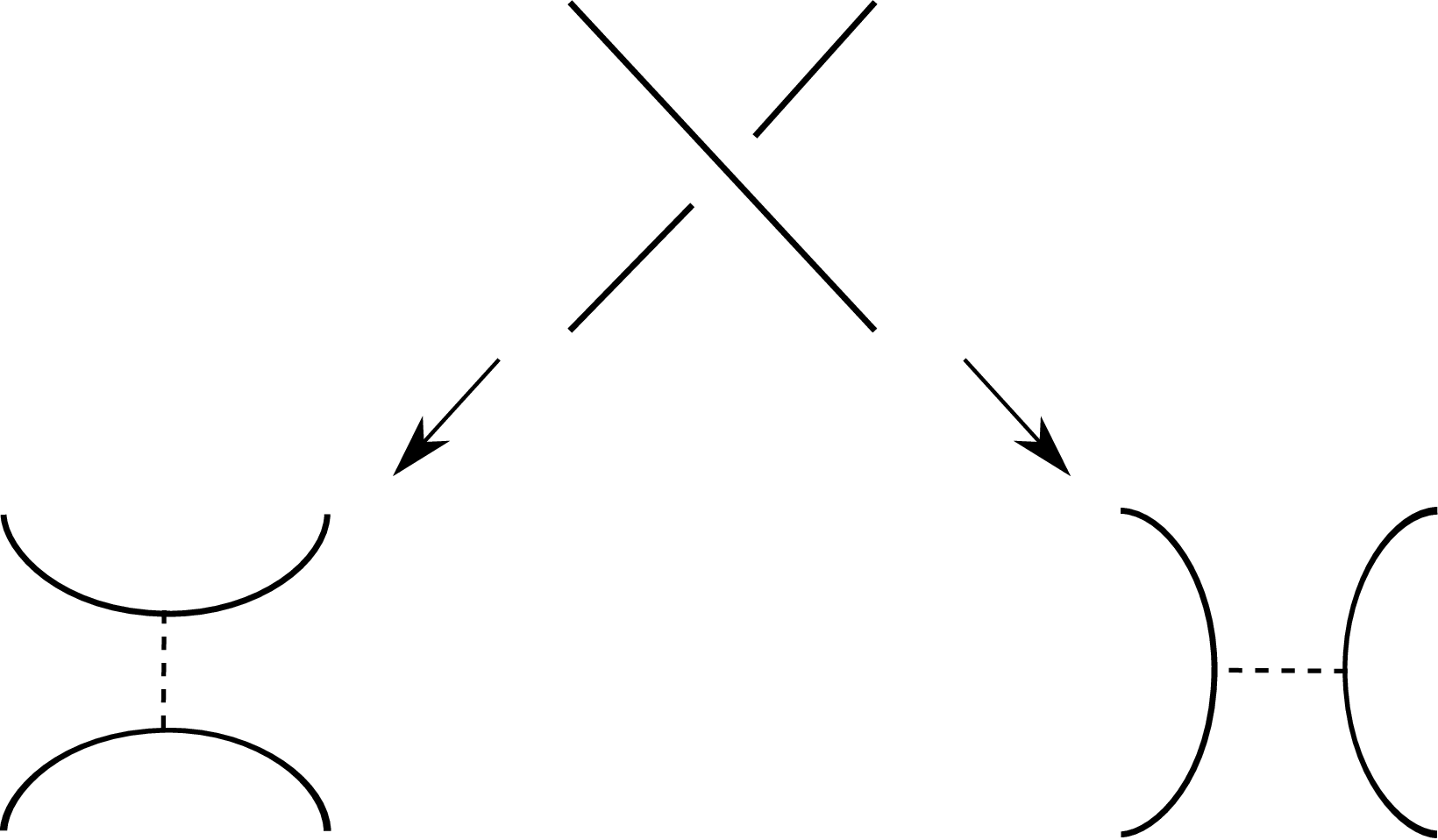}
    \put(-88,33){$A$}
          \put(-30,33){$B$}
    \caption{A and B smoothings}
  \label{smoothings}}
\end{figure}

We replace a crossing with a smoothing together with a dashed line joining the two arcs. After applying a smoothing to each crossing in $D$ we obtain diagram of a collection of disjoint circles in the plane. We call this diagram a \textit{state} for the diagram $D$. The all-$A$ smoothing state and the the all-$B$ smoothing for $D$ are of particular importance for us. The all-$A$ smoothing (all-$B$ smoothing) state of $D$ is the state obtained by replacing each crossing by a $A$ smoothing ($B$ smoothing). Write $S_A(D)$ and $S_B(D)$ to denote the all $A$ smoothing and all $B$ smoothing states of $D$ respectively.\\

If $D$ is a link diagram then the \textit{$A$-graph} $A(D)$ and the \textit{$B$-graph} $B(D)$ are two graphs associated to the all-$A$ smoothing and all-$B$ smoothing states of $D$. The set of vertices  of $A(D)$ is equal to the set of circles in $S_A(D)$. Moreover, an edge in the set of edges of $A(D)$ is obtained by joining two vertices of $A(D)$ for each crossing in $D$ between the corresponding circles.  We obtain the \textit{reduced $A$-graph} $A(D)^{\prime}$ by keeping the same set of vertices of $A(D)$ and replacing parallel edges by a single edge. See Figure \ref{thisisit} for an example. We define the $B$-graph $B(D)$ and reduced $B$-graph $B(D)^{\prime}$ similarly.
 
 \begin{figure}[H]
  \centering
    {\includegraphics[scale=0.06]{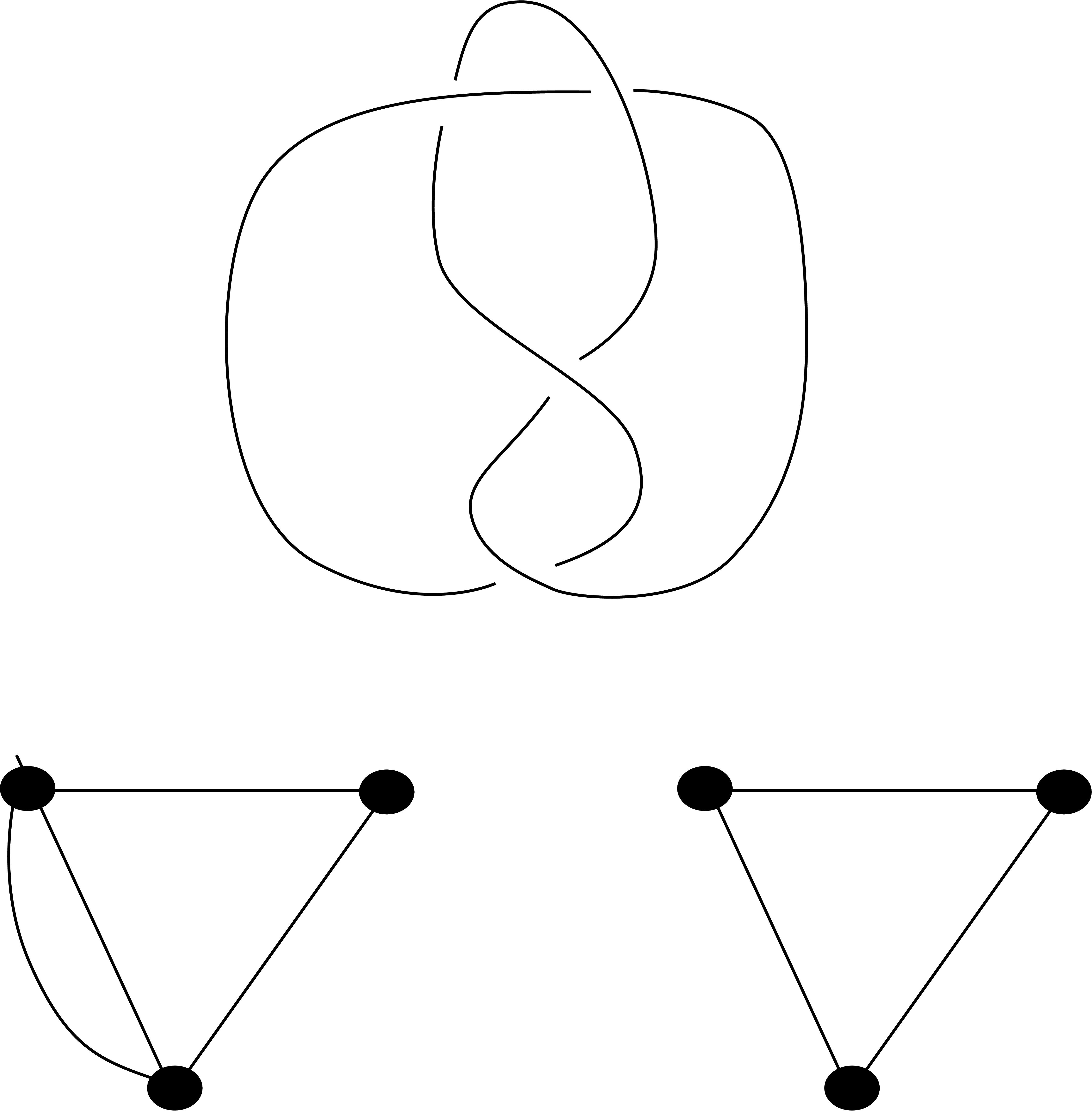}
    \caption{The knot $4_1$, its $A$-graph on the left and its reduced $A$-graph on the right}
    \label{thisisit}
 }
\end{figure} 
  
Now let $D$ be a link diagram and consider the skein element obtained from $S_B(D)$ by decorating each circle in $S_B(D)$ with the $n^{th}$ Jones-Wenzl idempotent and replacing each dashed line in $S_B(D)$ with the $(2n)^{th}$ Jones-Wenzl idempotent. See Figure \ref{allB1} for an example. Write $S_B^{(n)}(D)$ to denote this skein element.
\begin{figure}[H]
  \centering
   {\includegraphics[scale=0.08]{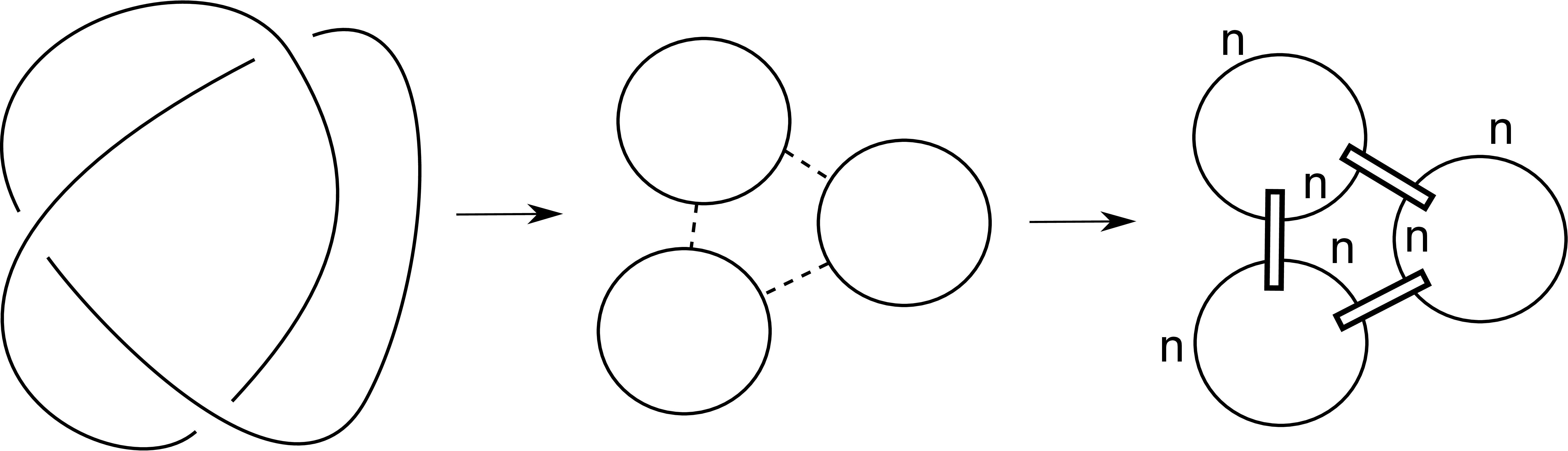}
    \caption{Obtaining $S_B^{(n)}(D)$ from a knot diagram $D$}
    \label{allB1}
 }
\end{figure}

We will need the following theorem.
\begin{theorem}(C. Armond \cite{Cody2})
\label{cody thm}
Let $L$ be a link in $S^3$ and $D$ be a reduced alternating knot diagram of $L$. Then
\begin{equation*}
\tilde{J}_{n,L}(A)\doteq_{4(n+1)}S_B^{(n)}(D).
\end{equation*} 
\end{theorem}
Theorem \ref{cody thm} has an important consequence, which basically tells us that the tail (the head) of an alternating link only depends on the reduced $A$-graph (reduced $B$-graph):  
\begin{theorem}(C. Armond, O. Dasbach \cite{Cody})
\label{cody thm2}
Let $L_1$ and $L_2$ be two alternating links with alternating diagrams $D_1$ and $D_2$. If the graph $A(D_1)^{\prime}$ coincides with $A(D_2)^{\prime}$, then $T_{L_1}=T_{L_2}$. Similarly, if $B(D_1)^{\prime}$ coincides with $B(D_2)^{\prime}$, then $H_{L_1}=H_{L_2}$.
\end{theorem}
Finding an exact form for the head and tail series is an interesting task. Explicit calculations were done on the knot table to determine these two power series in \cite{Cody}. Using multiple techniques Armond and Dasbach determined the head and tail for an infinite family of knots and links. The knot $8_5$, Figure \ref{85knot}, is the first knot on the knot table whose tail could not be determined by a direct application of techniques in \cite{Cody}.
\begin{figure}[H]
  \centering
 {\includegraphics[scale=0.055]{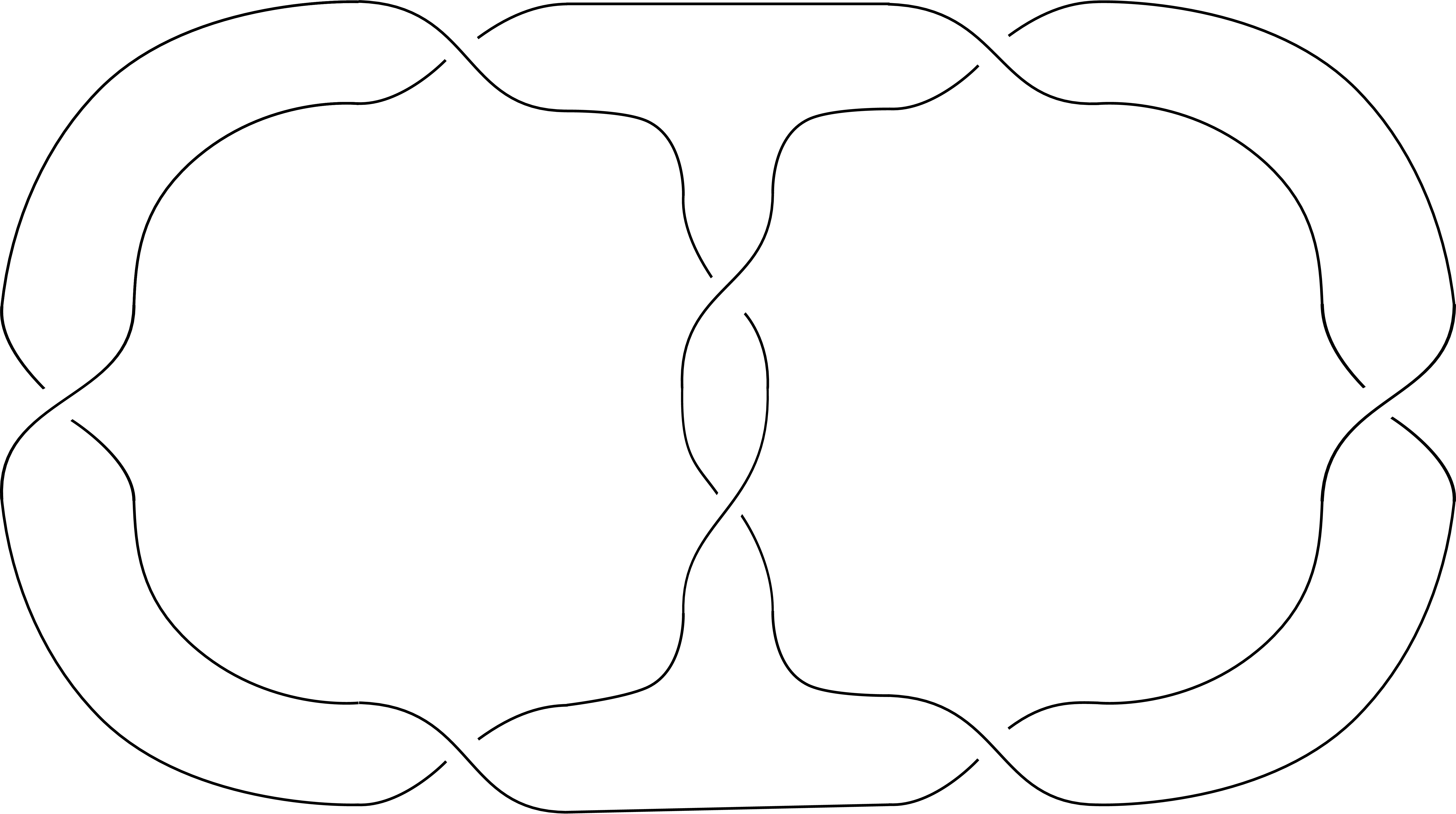}
  \caption{The knot $8_5$}
  \label{85knot}}
\end{figure}
 The techniques we developed here appear to be helpful in understanding the head and the tail for some knots. We demonstrate this by studying the family of knots in Figure \ref{85 gen knot}. Note that we obtain the knot $8_5$ from this family by replacing each crossing region by a single crossing.

\begin{figure}[H]
  \centering
 {\includegraphics[scale=0.065]{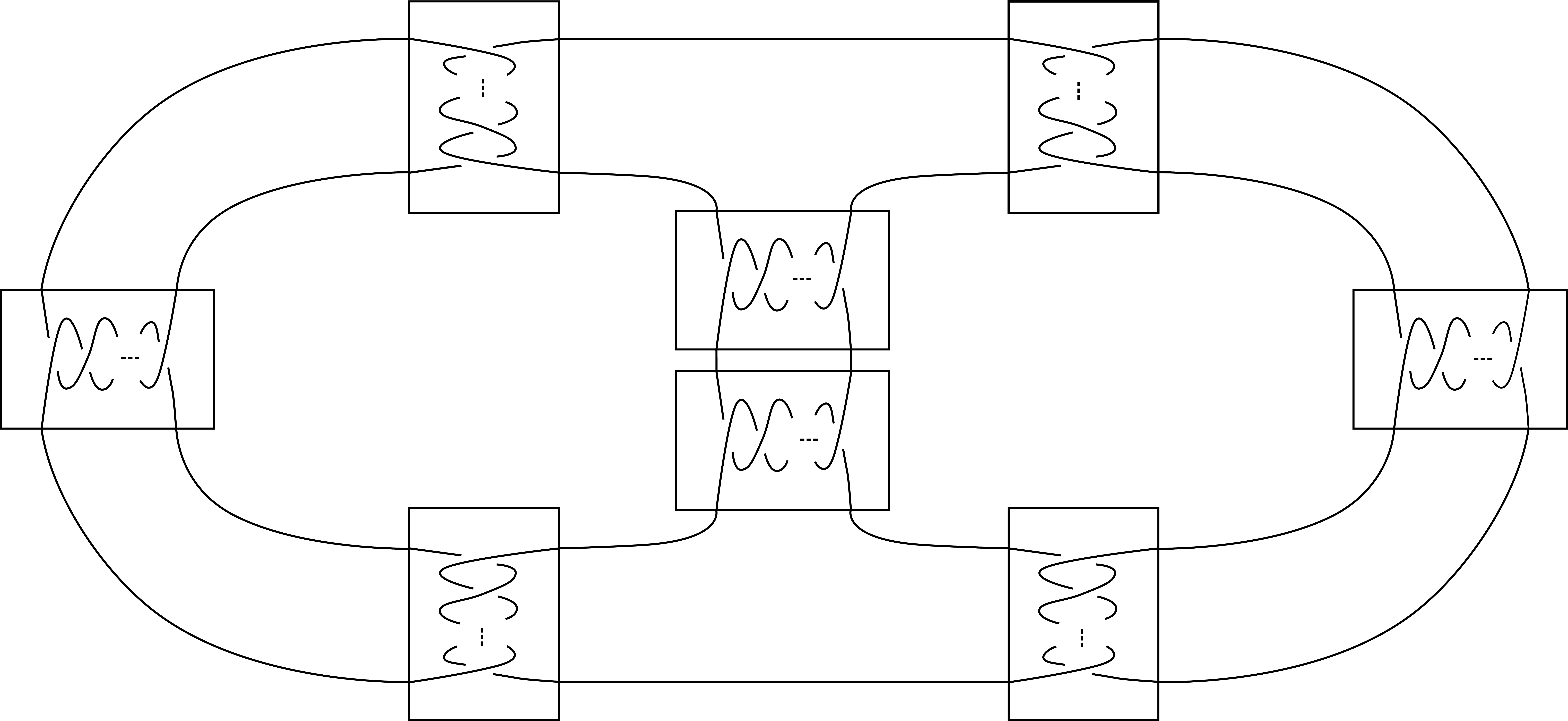}
  \caption{The knot $\Gamma$}
  \label{85 gen knot}}
\end{figure}
The reduced $B$-graph for each knot in this family can be easily seen to coincide with the graph in Figure \ref{graph}.
\begin{figure}[H]
  \centering
 {\includegraphics[scale=0.04]{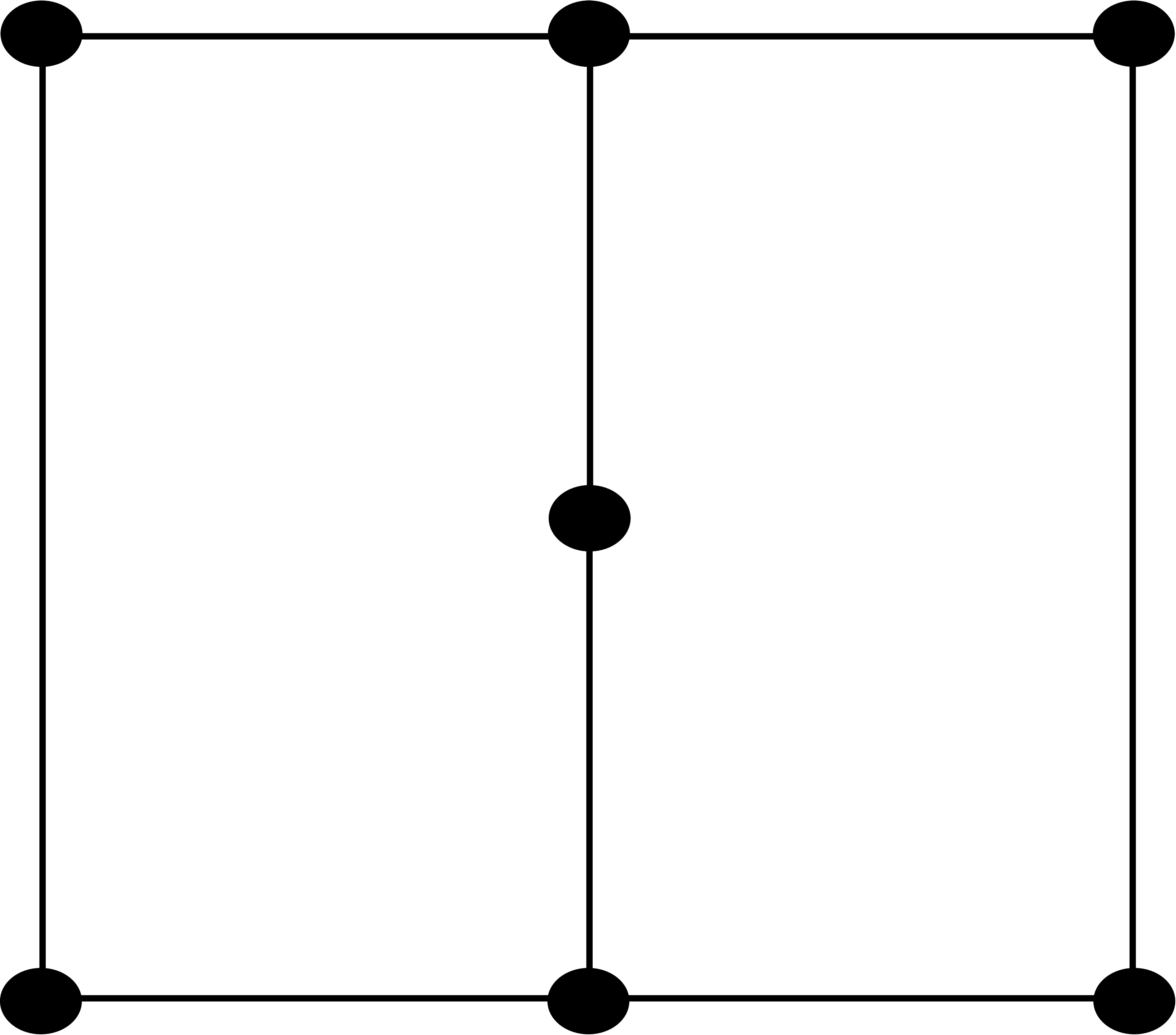}
  \caption{The reduced $B$-graph for $\Gamma$}
  \label{graph}}
\end{figure}
%\begin{figure}[H]
%  \centering
%   {\includegraphics[scale=0.08]{8-5}
 % \caption{$8_5$}
%  \label{85 knot}}
%\end{figure}
Hence the skein element $S_B^{(n)}$ for any knot from the family of the knots shown in \ref{85 gen knot} is given in Figure \ref{85}. Note that the bubble skein element appears in multiple places in this skein element. 
\begin{figure}[H]
  \centering
   {\includegraphics[scale=0.08]{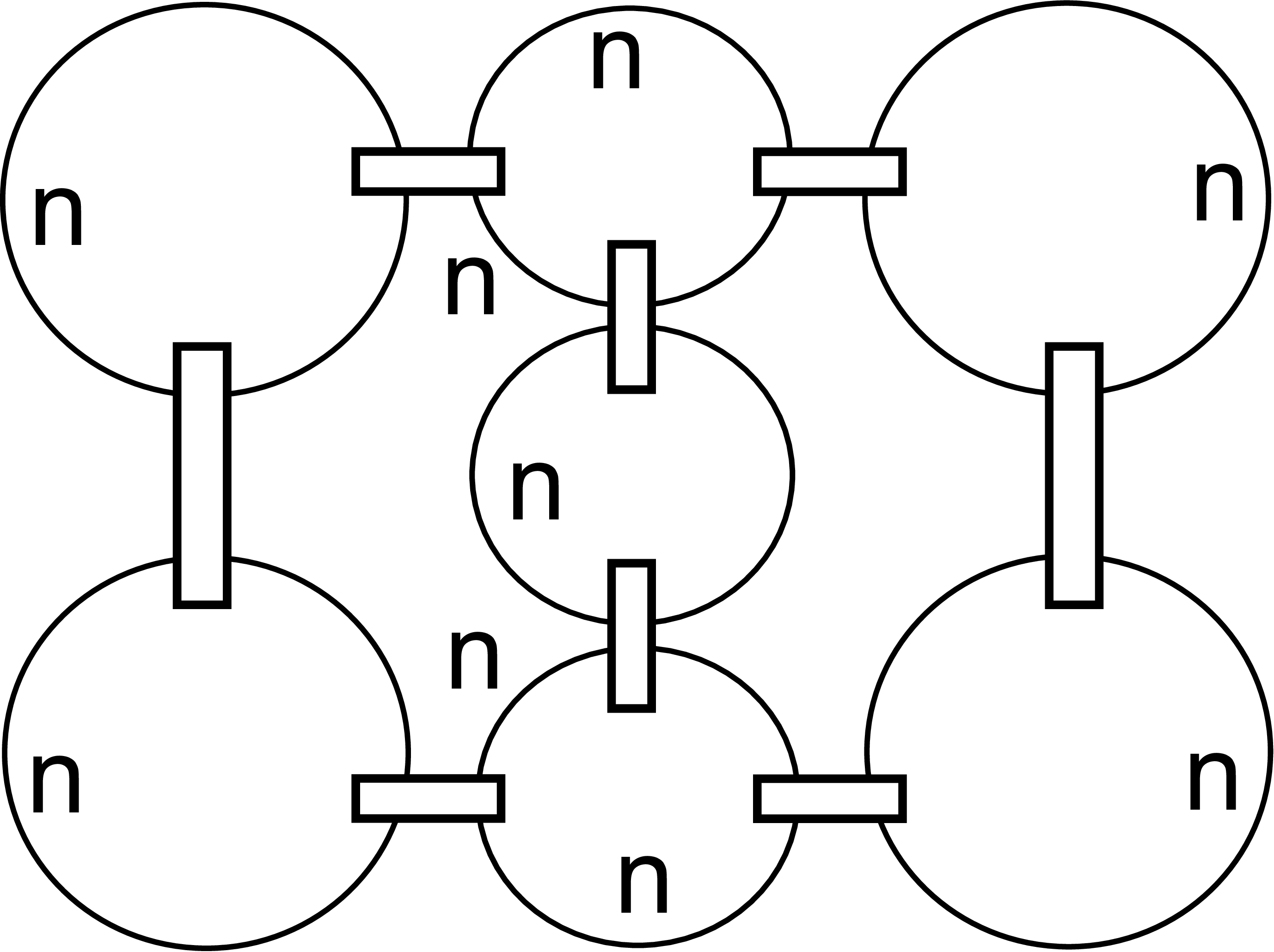}
  \caption{The skein element $S_B^{(n)}(\Gamma)$}
  \label{85}}
\end{figure}

Let $\Gamma$ be any knot from the family in Figure \ref{85 gen knot}. Theorem \ref{cody thm} implies
{\small
\begin{eqnarray*}
  \tilde{J}_{n,\Gamma}(A)\doteq_{4(n+1)} \begin{minipage}[h]{0.24\linewidth}
        \vspace{0pt}
        \scalebox{0.09}{\includegraphics{tailmain}}
   \end{minipage}.
  \end{eqnarray*}}
  \begin{remark}
  The algorithm of Masbaum and Vogel in \cite{Masbaum1} can be used to compute the evaluation of any quantum spin network in $\mathcal{S}(S^2)$. In particular, it can be used to give a formula for the skein element $S_B^{(n)}(\Gamma)$. However, it is difficult to compute the tail of $\Gamma$ using the formula obtained from this algorithm. For this reason, we will use the techniques we developed here to compute the evaluation of the skein element $S_B^{(n)}(\Gamma)$.
  \end{remark}
  Now we compute the tail of the knot $\Gamma$. 
\begin{lemma}
{\footnotesize
\begin{equation*}
\label{mainfor}
  \tilde{J}_{n,\Gamma}(A)
   \doteq_{4(n+1)} \displaystyle\sum\limits_{i=0}^n\ \sum\limits_{j=0}^n\left\lceil 
\begin{array}{cc}
n & n \\ 
n & n%
\end{array}%
\right\rceil _{i}\left\lceil 
\begin{array}{cc}
n & n \\ 
n & n%
\end{array}%
\right\rceil _{j}\left\lceil 
\begin{array}{cc}
i & n \\ 
n & n-j%
\end{array}%
\right\rceil _{0}\left\lceil 
\begin{array}{cc}
j & n \\ 
n & n-i%
\end{array}%
\right\rceil _{0}\left\lceil 
\begin{array}{cc}
j & i \\ 
n & n%
\end{array}%
\right\rceil _{0}\frac{\Delta_{2n}}{\Delta_{n+i}}\frac{\Delta_{2n}}{\Delta_{n+j}}\Delta_{i+j}
\end{equation*}
}
\end{lemma}  
\begin{proof}
  We use the bubble expansion formula on the top left bubble in the skein element $S_B^{(n)}(\Gamma)$ we obtain: 
{\small
\begin{eqnarray*}
  \tilde{J}_{n,\Gamma}(A)
   &\doteq_{4(n+1)} &\displaystyle\sum\limits_{i=0}^n
   \left\lceil 
\begin{array}{cc}
n & n \\ 
n & n%
\end{array}%
\right\rceil _{i}\hspace{1mm}
       \begin{minipage}[h]{0.34\linewidth}
        \vspace{0pt}
        \scalebox{0.09}{\includegraphics{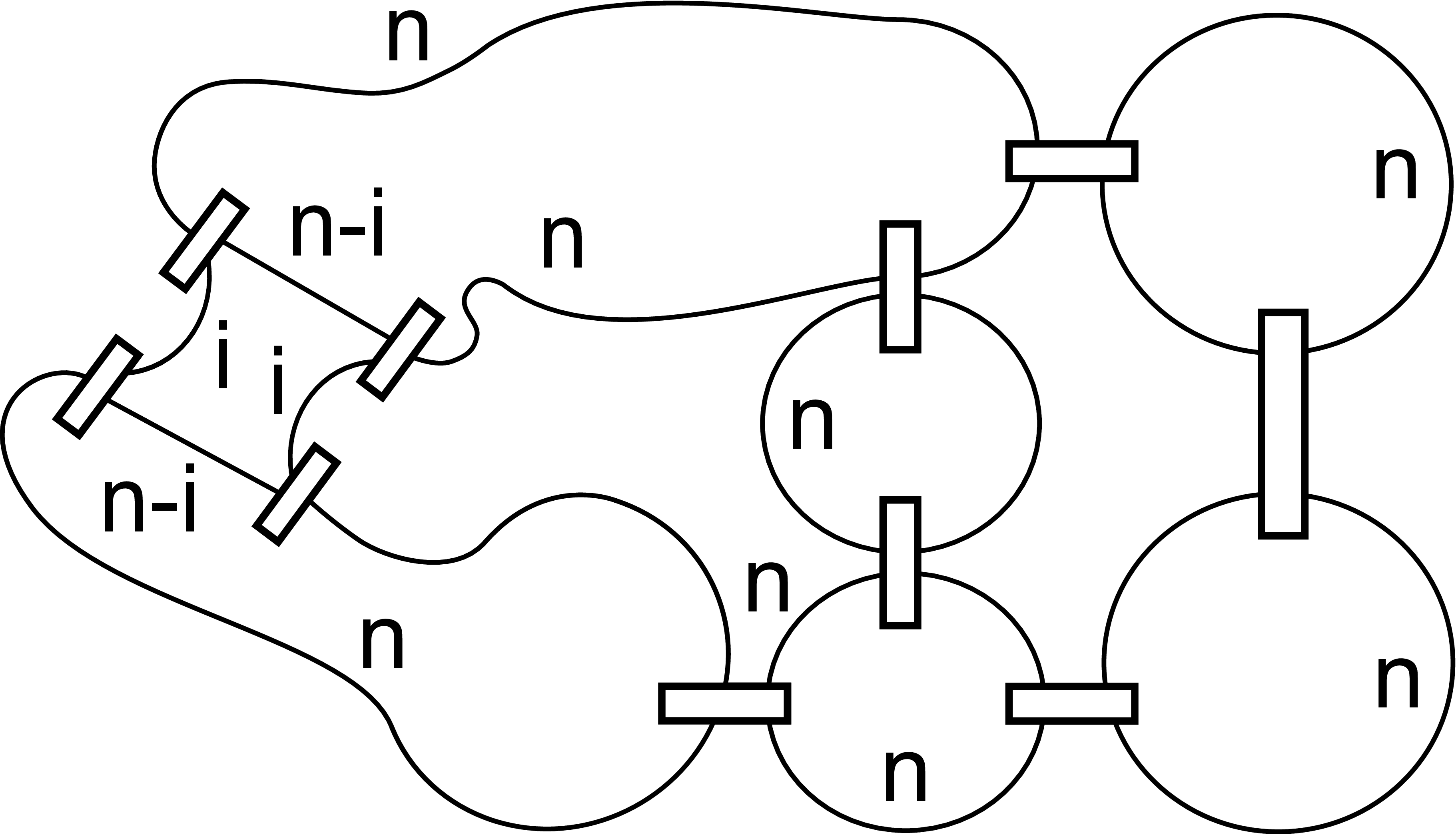}}
   \end{minipage}.\end{eqnarray*}}
Using properties of the Jones-Wenzl projector 
   {\small
\begin{eqnarray*}
 \tilde{J}_{n,\Gamma}(A)
   &\doteq_{4(n+1)}\displaystyle\sum\limits_{i=0}^n
   \left\lceil 
\begin{array}{cc}
n & n \\ 
n & n%
\end{array}%
\right\rceil _{i}\hspace{1mm}
  \begin{minipage}[h]{0.29\linewidth}
        \vspace{0pt}
        \scalebox{0.09}{\includegraphics{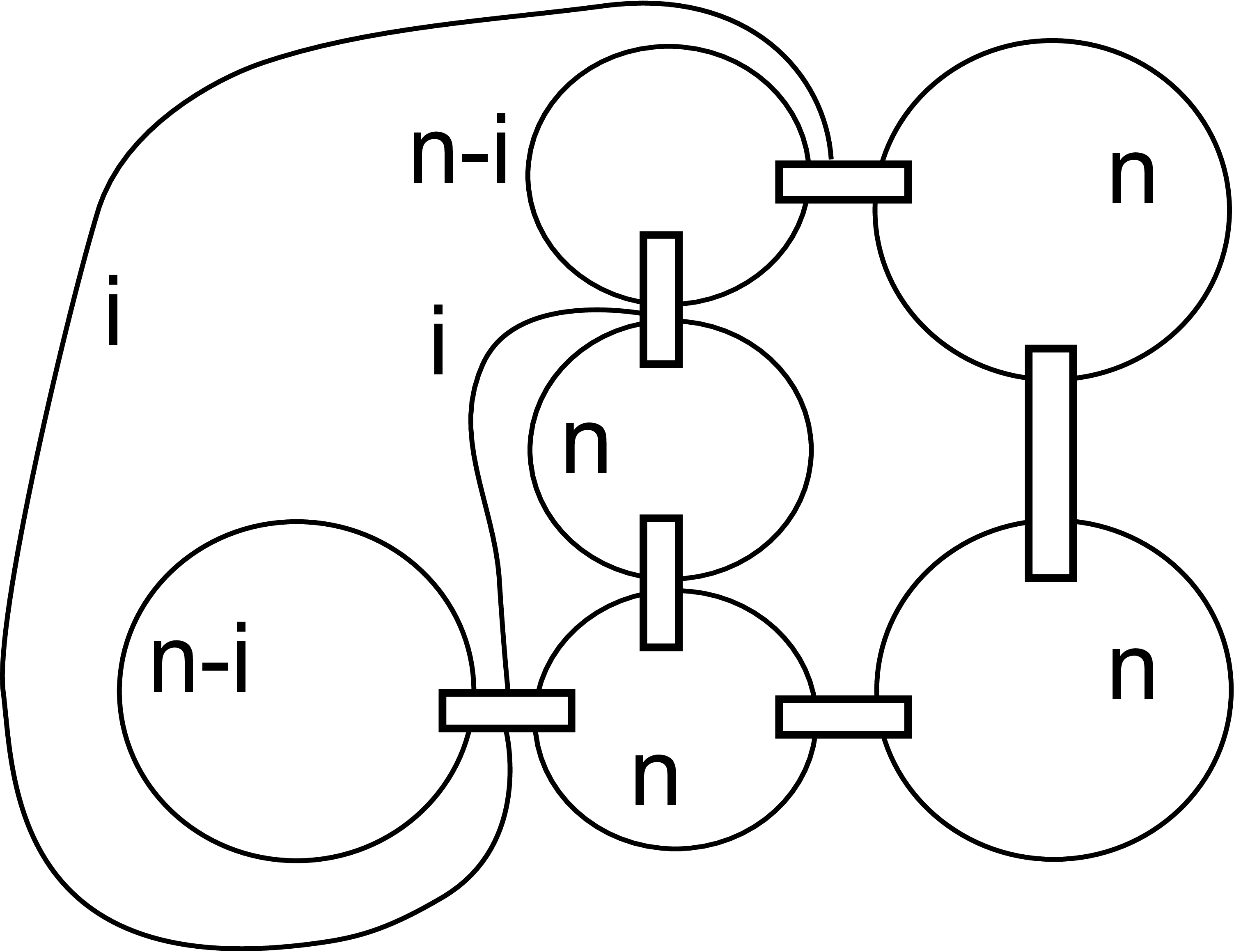}}
   \end{minipage}.\\
\end{eqnarray*}}
Using the identity (\ref{properties2}) 
{\small
\begin{eqnarray*}
  \tilde{J}_{n,\Gamma}(A)
   &\doteq_{4(n+1)}\displaystyle\sum\limits_{i=0}^n\left\lceil 
\begin{array}{cc}
n & n \\ 
n & n%
\end{array}%
\right\rceil _{i}\frac{\Delta_{2n}}{\Delta_{n+i}}\hspace{2mm}
   \begin{minipage}[h]{0.24\linewidth}
        \vspace{0pt}
        \scalebox{0.1}{\includegraphics{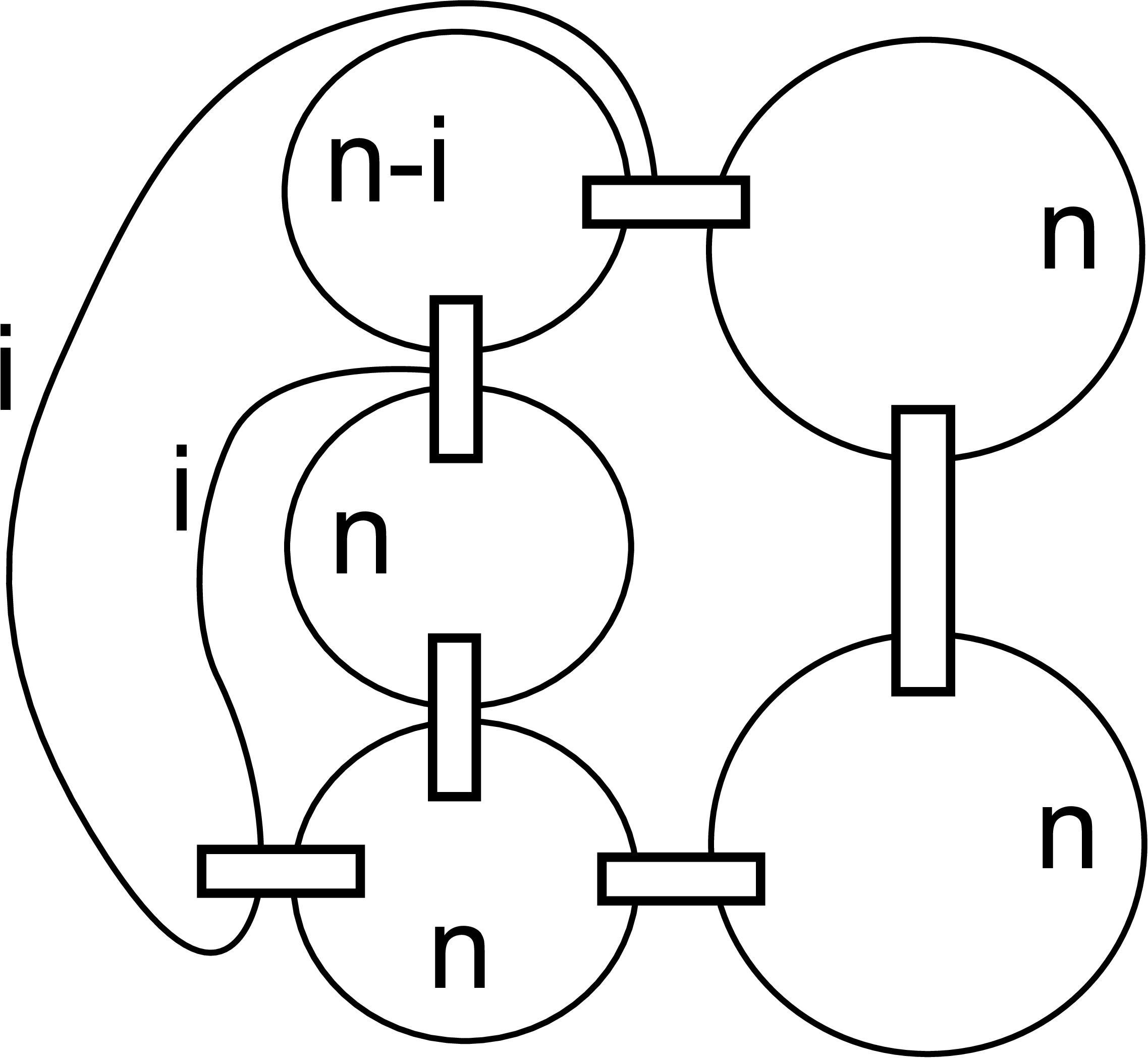}}
   \end{minipage}.
  \end{eqnarray*}}
   Now apply the bubble expansion formula to the lower mostright bubble 
   {\footnotesize
\begin{eqnarray*}
  \tilde{J}_{n,\Gamma}(A)
   &\doteq_{4(n+1)} &
   \displaystyle\sum\limits_{i=0}^n\ \sum\limits_{j=0}^n\left\lceil 
\begin{array}{cc}
n & n \\ 
n & n%
\end{array}%
\right\rceil _{i}\left\lceil 
\begin{array}{cc}
n & n \\ 
n & n%
\end{array}%
\right\rceil _{j}\frac{\Delta_{2n}}{\Delta_{n+i}}\hspace{2mm}
  \begin{minipage}[h]{0.24\linewidth}
        \vspace{0pt}
        \scalebox{0.1}{\includegraphics{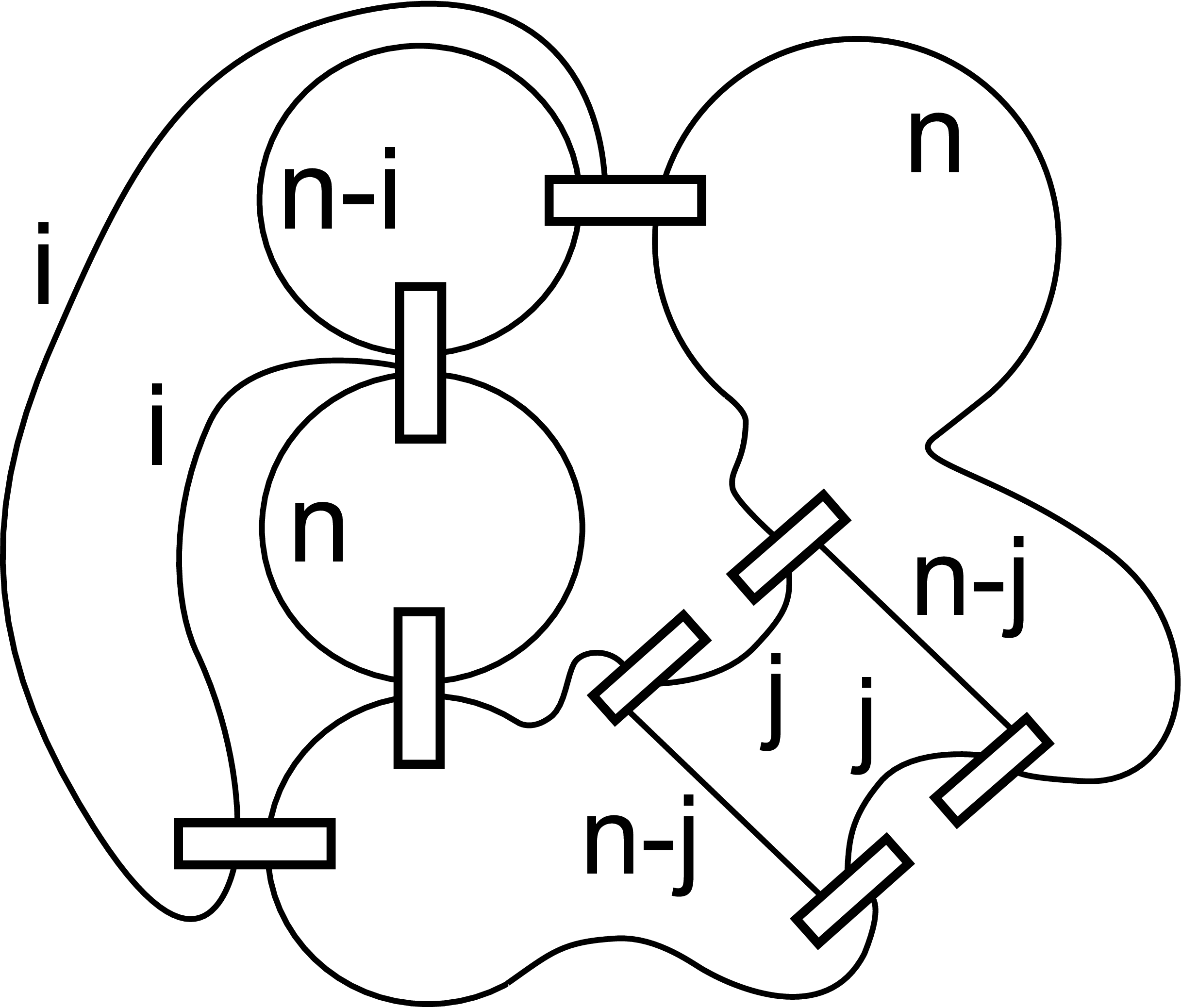}}
   \end{minipage}\\&=&\displaystyle\sum\limits_{i=0}^n\ \sum\limits_{j=0}^n\left\lceil 
\begin{array}{cc}
n & n \\ 
n & n%
\end{array}%
\right\rceil _{i}\left\lceil 
\begin{array}{cc}
n & n \\ 
n & n%
\end{array}%
\right\rceil _{j}\frac{\Delta_{2n}}{\Delta_{n+i}}\frac{\Delta_{2n}}{\Delta_{n+j}}\hspace{2mm}
  \begin{minipage}[h]{0.17\linewidth}
        \vspace{0pt}
        \scalebox{0.1}{\includegraphics{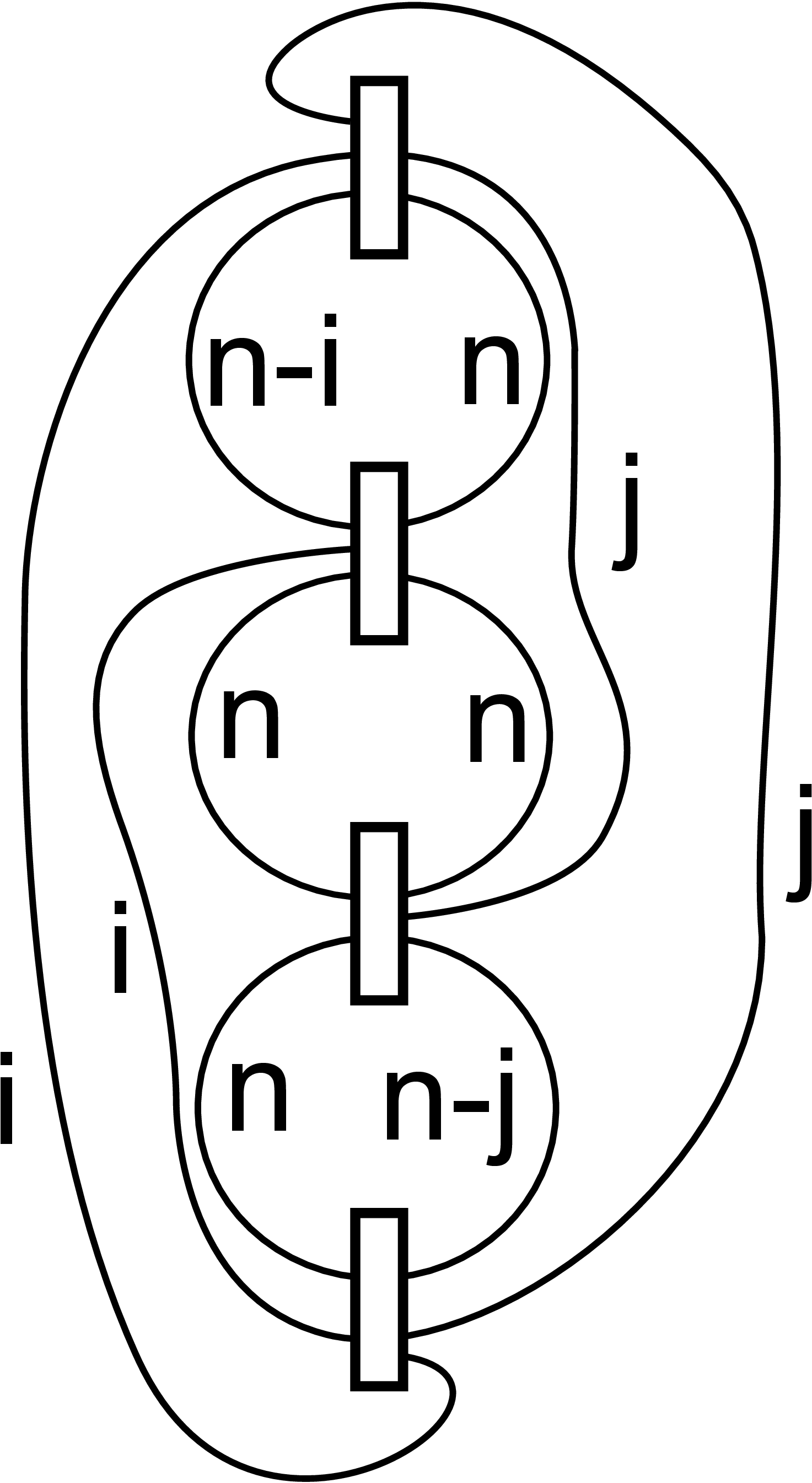}}
        \end{minipage}.
        \end{eqnarray*}}
        
 Similarly, we apply the bubble expansion formula on the top bubble that appears in the previous equation
 
 {\footnotesize
\begin{eqnarray*}
  \tilde{J}_{n,\Gamma}(A)
   &\doteq_{4(n+1)}&     
        \displaystyle\sum\limits_{i=0}^n\ \sum\limits_{j=0}^n\sum\limits_{k=0}^{\min(j,n-i)}\left\lceil 
\begin{array}{cc}
n & n \\ 
n & n%
\end{array}%
\right\rceil _{i}\left\lceil 
\begin{array}{cc}
n & n \\ 
n & n%
\end{array}%
\right\rceil _{j}\left\lceil 
\begin{array}{cc}
j & n \\ 
n & n-i%
\end{array}%
\right\rceil _{k}\\&\times&\frac{\Delta_{2n}}{\Delta_{n+i}}\frac{\Delta_{2n}}{\Delta_{n+j}}\hspace{2mm}
  \begin{minipage}[h]{0.24\linewidth}
        \vspace{0pt}
        \scalebox{0.11}{\includegraphics{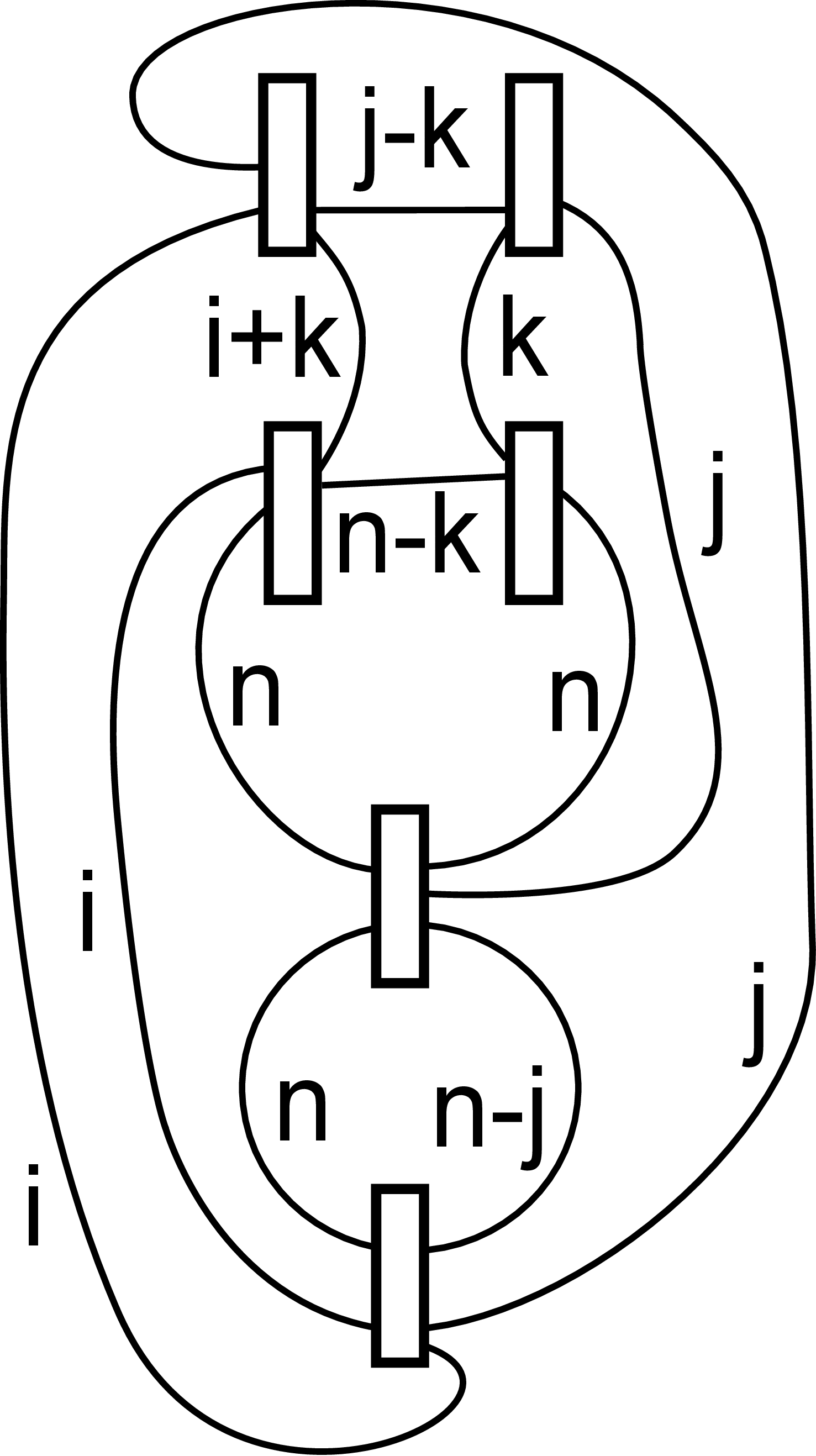}}
        \end{minipage}\\&=&\displaystyle\sum\limits_{i=0}^n\ \sum\limits_{j=0}^n\sum\limits_{k=0}^{\min(j,n-i)}\left\lceil 
\begin{array}{cc}
n & n \\ 
n & n%
\end{array}%
\right\rceil _{i}\left\lceil 
\begin{array}{cc}
n & n \\ 
n & n%
\end{array}%
\right\rceil _{j}\left\lceil 
\begin{array}{cc}
j & n \\ 
n & n-i%
\end{array}%
\right\rceil _{k}\\ &\times & \frac{\Delta_{2n}}{\Delta_{n+i}}\frac{\Delta_{2n}}{\Delta_{n+j}}\hspace{2mm}
  \begin{minipage}[h]{0.24\linewidth}
        \vspace{+2pt}
        \scalebox{0.11}{\includegraphics{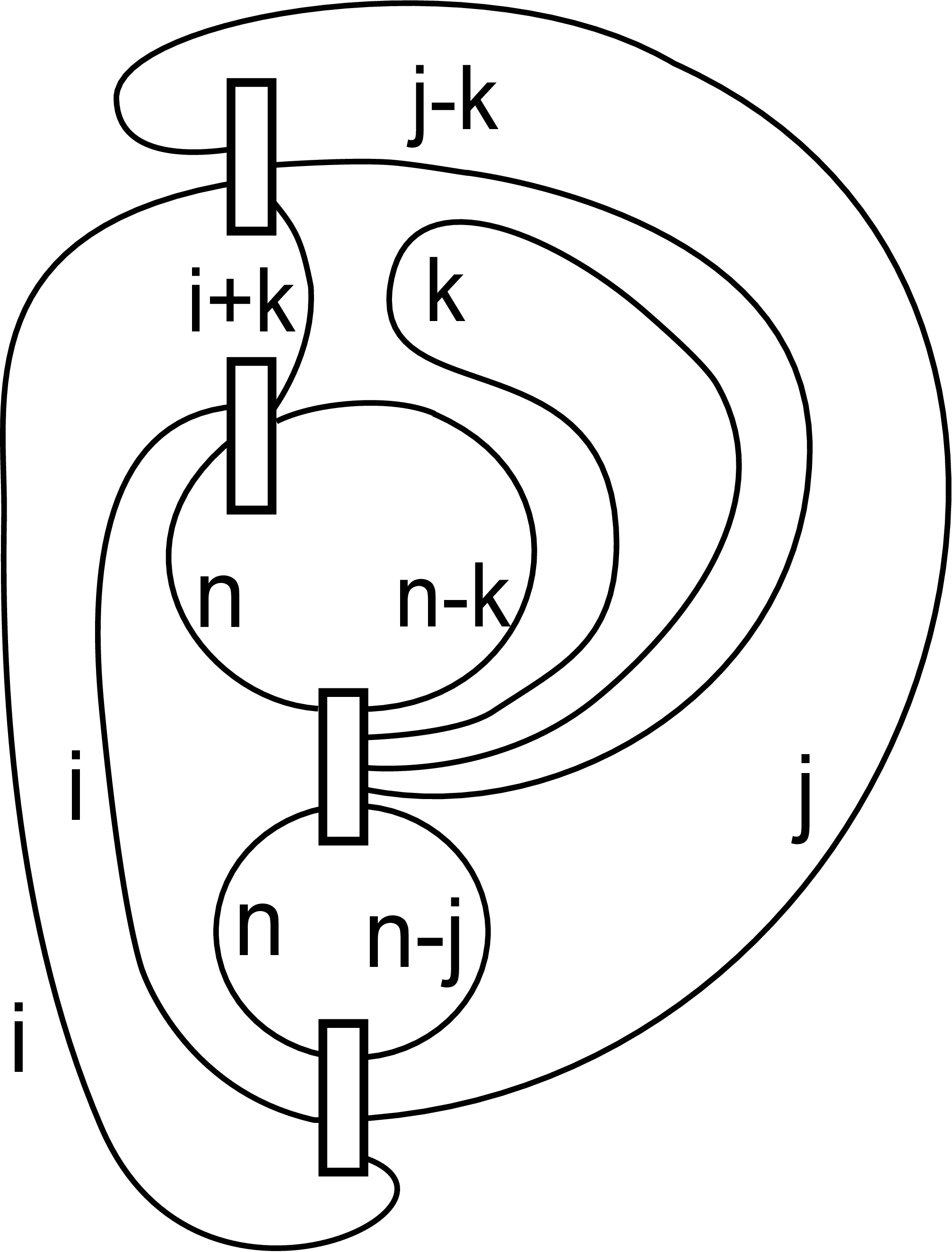}}
        \end{minipage}.
   \end{eqnarray*}
  }
Note that the previous sum is zero unless $k=0$. Hence
{\footnotesize
\begin{eqnarray*}
  \tilde{J}_{n,\Gamma}(A)
   &\doteq_{4(n+1)} &\displaystyle\sum\limits_{i=0}^n\ \sum\limits_{j=0}^n\left\lceil 
\begin{array}{cc}
n & n \\ 
n & n%
\end{array}%
\right\rceil _{i}\left\lceil 
\begin{array}{cc}
n & n \\ 
n & n%
\end{array}%
\right\rceil _{j}\left\lceil 
\begin{array}{cc}
j & n \\ 
n & n-i%
\end{array}%
\right\rceil _{0}\frac{\Delta_{2n}}{\Delta_{n+i}}\frac{\Delta_{2n}}{\Delta_{n+j}}\hspace{2mm}
  \begin{minipage}[h]{0.19\linewidth}
        \vspace{+2pt}
        \scalebox{0.1}{\includegraphics{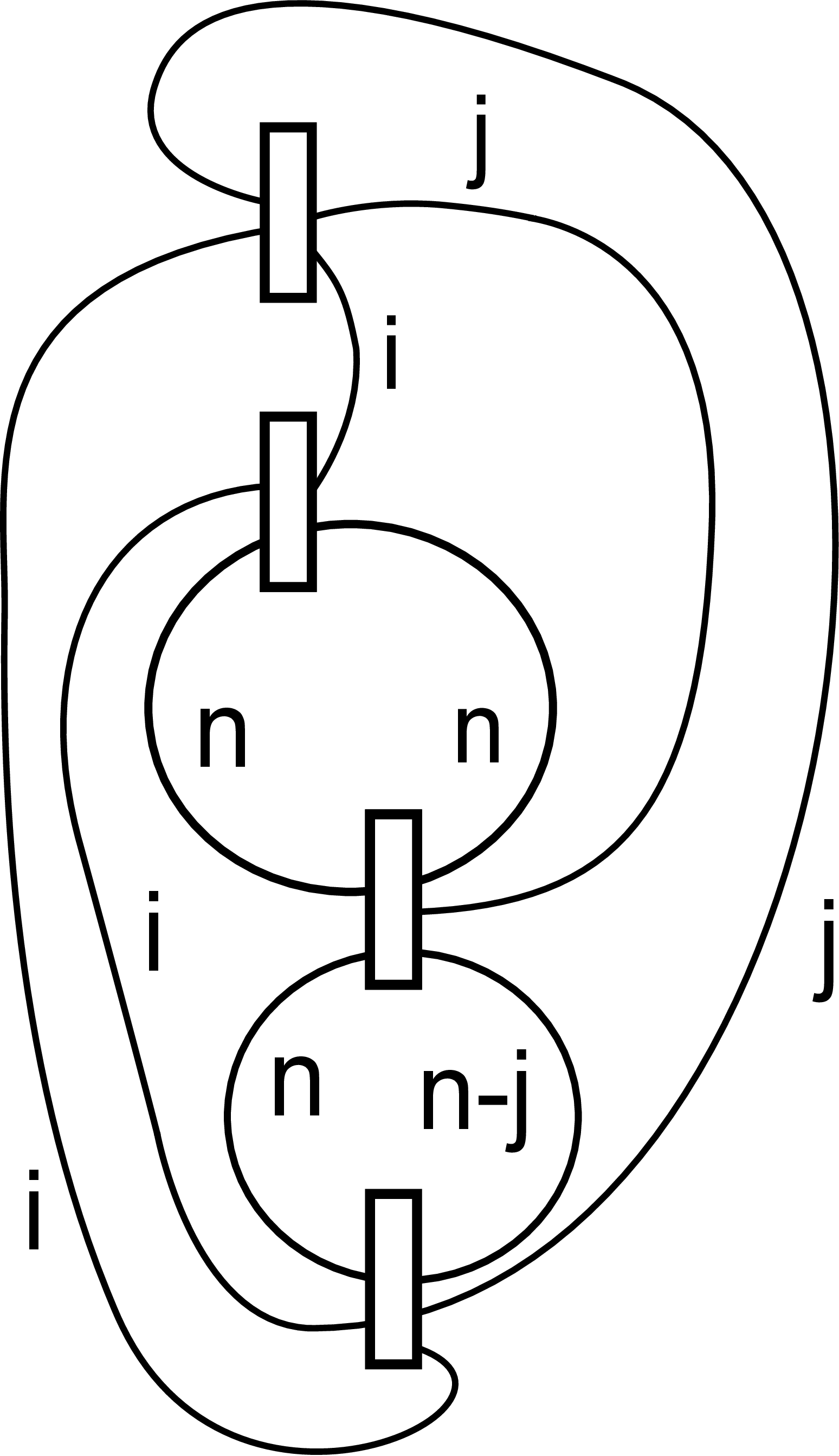}}
        \end{minipage}.
       \end{eqnarray*}}
       Similarly 
 {\footnotesize
\begin{eqnarray*}
\tilde{J}_{n,\Gamma}(A)
   &\doteq_{4(n+1)} &\displaystyle\sum\limits_{i=0}^n\ \sum\limits_{j=0}^n\left\lceil 
\begin{array}{cc}
n & n \\ 
n & n%
\end{array}%
\right\rceil _{i}\left\lceil 
\begin{array}{cc}
n & n \\ 
n & n%
\end{array}%
\right\rceil _{j}\left\lceil 
\begin{array}{cc}
i & n \\ 
n & n-j%
\end{array}%
\right\rceil _{0}\left\lceil 
\begin{array}{cc}
j & n \\ 
n & n-i%
\end{array}%
\right\rceil _{0}\frac{\Delta_{2n}}{\Delta_{n+i}}\frac{\Delta_{2n}}{\Delta_{n+j}}\hspace{2mm}
   \begin{minipage}[h]{0.20\linewidth}
        \vspace{0pt}
        \scalebox{0.09}{\includegraphics{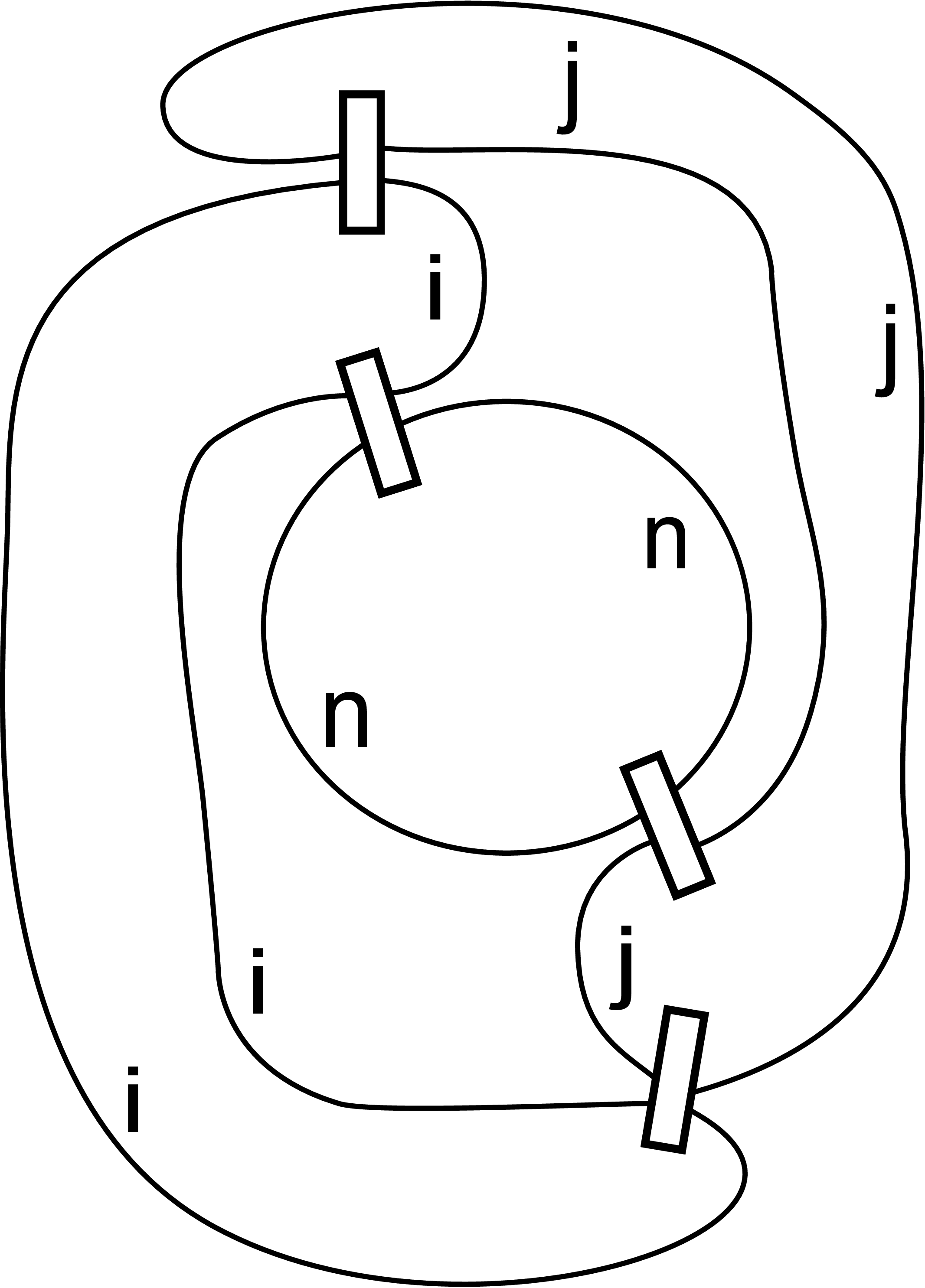}}
   \end{minipage}.
   \end{eqnarray*}}
Applying the bubble expansion one last time on the middle bubble in previous summation, we obtain 
   {\footnotesize
\begin{eqnarray*}
   \tilde{J}_{n,\Gamma}(A)
   &\doteq_{4(n+1)}&\displaystyle\sum\limits_{i=0}^n\ \sum\limits_{j=0}^n\sum\limits_{k=0}^{\min(i,j)}\left\lceil 
\begin{array}{cc}
n & n \\ 
n & n%
\end{array}%
\right\rceil _{i}\left\lceil 
\begin{array}{cc}
n & n \\ 
n & n%
\end{array}%
\right\rceil _{j}\left\lceil 
\begin{array}{cc}
i & n \\ 
n & n-j%
\end{array}%
\right\rceil _{0}\left\lceil 
\begin{array}{cc}
j & n \\ 
n & n-i%
\end{array}%
\right\rceil _{0}\left\lceil 
\begin{array}{cc}
j & i \\ 
n & n%
\end{array}%
\right\rceil _{k}
\\&\times& \frac{\Delta_{2n}}{\Delta_{n+i}}\frac{\Delta_{2n}}{\Delta_{n+j}}\hspace{2mm}
   \begin{minipage}[h]{0.24\linewidth}
        \vspace{0pt}
        \scalebox{0.09}{\includegraphics{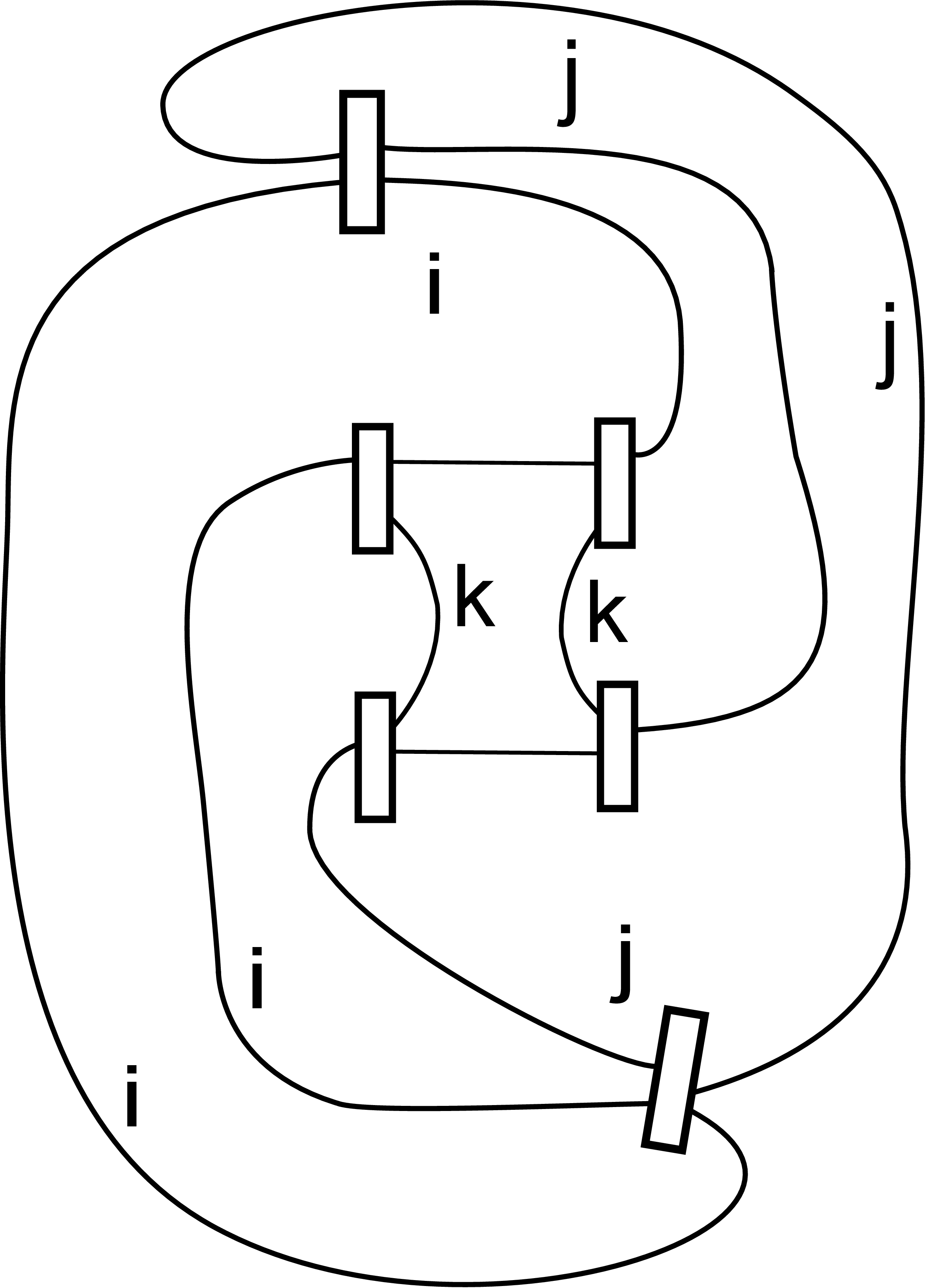}}
   \end{minipage}.\end{eqnarray*}
}
The previous summation is zero unless $k=0$. Hence
{\small
\begin{eqnarray*}
    \tilde{J}_{n,\Gamma}(A)
   &\doteq_{4(n+1)}&\displaystyle\sum\limits_{i=0}^n\ \sum\limits_{j=0}^n\left\lceil 
\begin{array}{cc}
n & n \\ 
n & n%
\end{array}%
\right\rceil _{i}\left\lceil 
\begin{array}{cc}
n & n \\ 
n & n%
\end{array}%
\right\rceil _{j}\left\lceil 
\begin{array}{cc}
i & n \\ 
n & n-j%
\end{array}%
\right\rceil _{0}\left\lceil 
\begin{array}{cc}
j & n \\ 
n & n-i%
\end{array}%
\right\rceil _{0}\left\lceil 
\begin{array}{cc}
j & i \\ 
n & n%
\end{array}%
\right\rceil _{0}\\ &\times & \frac{\Delta_{2n}}{\Delta_{n+i}}\frac{\Delta_{2n}}{\Delta_{n+j}}\hspace{2mm}
   \begin{minipage}[h]{0.24\linewidth}
        \vspace{0pt}
        \scalebox{0.09}{\includegraphics{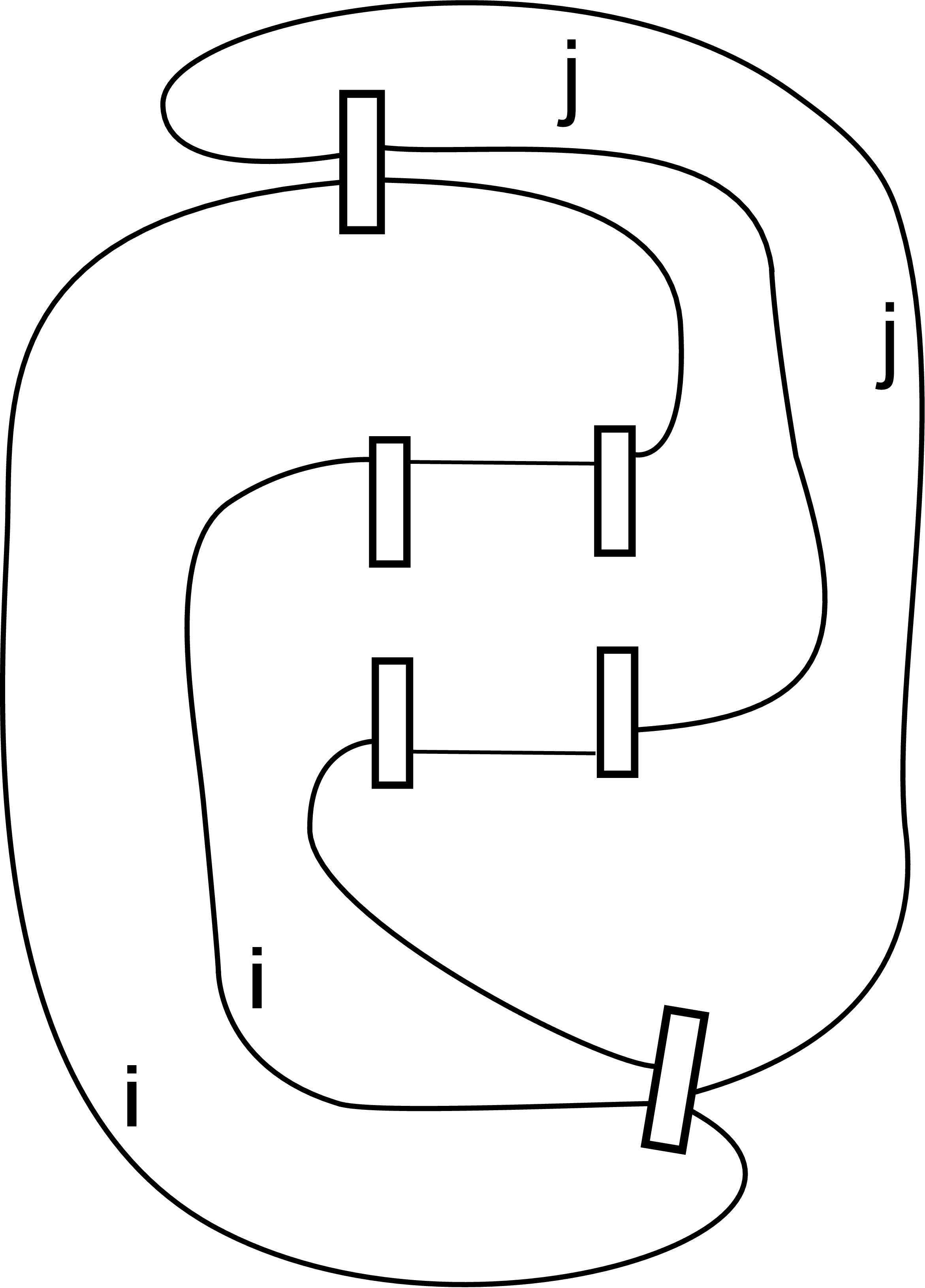}}
   \end{minipage}.
\end{eqnarray*}
}
which implies that
{\footnotesize
\begin{equation}
\label{mainfor}
  \tilde{J}_{n,\Gamma}(A)
   \doteq_{4(n+1)} \displaystyle\sum\limits_{i=0}^n\ \sum\limits_{j=0}^n\left\lceil 
\begin{array}{cc}
n & n \\ 
n & n%
\end{array}%
\right\rceil _{i}\left\lceil 
\begin{array}{cc}
n & n \\ 
n & n%
\end{array}%
\right\rceil _{j}\left\lceil 
\begin{array}{cc}
i & n \\ 
n & n-j%
\end{array}%
\right\rceil _{0}\left\lceil 
\begin{array}{cc}
j & n \\ 
n & n-i%
\end{array}%
\right\rceil _{0}\left\lceil 
\begin{array}{cc}
j & i \\ 
n & n%
\end{array}%
\right\rceil _{0}\frac{\Delta_{2n}}{\Delta_{n+i}}\frac{\Delta_{2n}}{\Delta_{n+j}}\Delta_{i+j}.
\end{equation}
}

\end{proof}
In the next proposition we work with the variable $q$. Recall that $A^4=q$.
\begin{proposition}
\begin{eqnarray*}
T_{\Gamma}(q)=(q;q)_{\infty}^2\sum\limits_{k=0}^{\infty}\frac{q^{k+k^{2}}}{%
(q;q)_{k}}(\sum\limits_{i=0}^{k}q^{(-2i(k-i))}\left[ 
\begin{array}{c}
k \\ 
i%
\end{array}%
\right] _{q}^{2}).
\end{eqnarray*}
\end{proposition}
\begin{proof}
Using Corollary \ref{coeffff} and the fact that
\begin{equation}
\label{fact}
\prod\limits_{i=0}^{j}[n-i]=q^{(2 + 3 j + j^2 - 2 n - 2 j n)/4} (1 - q)^{-1 - j}\frac{(q;q)_n}{(q;q)_{n-j-1}}.
\end{equation}

One obtains:
\begin{equation}
\label{mn}
\left\lceil 
\begin{array}{cc}
n & n \\ 
n & n%
\end{array}%
\right\rceil _{i}=(-1)^{i+n}q^{
 (2 i + 4 i^2 - 2 n)/4} 
 \frac{(q; q)^6_n(q;q)_{3n-i+1}}{(q;q)_{2 n}^2 (q ;q)_{2 n + 1} (q;q)_i^2 (q;q)_{n - i}^3},
\end{equation}
and
\begin{equation}
\label{mn1}
\left\lceil 
\begin{array}{cc}
j & n \\ 
n & n-i%
\end{array}%
\right\rceil _{0}=(-1)^{n-i} q^{(i - n)/2}  
 \frac{(q; q)_{i+j}(q;q)_{n}(q;q)_{n+i} (q ;q)_{2 n +j+ 1}}{ (q;q)_i(q;q)_{2n}(q;q)_{j+n}(q;q)_{n+j+i+1}}.
\end{equation}
Finally Lemma \ref{thetag} implies
\begin{equation}
\label{mn2}
\left\lceil 
\begin{array}{cc}
i & j \\ 
n & n%
\end{array}%
\right\rceil _{0}\Delta_{i+j}=\Lambda(n,i,j),
\end{equation}
and one could use (\ref{fact}) and the formula for the theta graph given in \cite{kaufflinks} or \cite{Masbaum1} to write

\begin{equation}
\label{mn3}
\left\lceil 
\begin{array}{cc}
i & j \\ 
n & n%
\end{array}%
\right\rceil _{0}\Delta_{i+j}=(-1)^{i+j+n} q^{-(i+j+n)/2}  
 \frac{(q;q)_{n}(q;q)_{j}(q;q)_{i} (q ;q)_{n +j+i+1}}{ (1-q)(q;q)_{i+n}(q;q)_{j+n}(q;q)_{j+i}}.
\end{equation}

Putting (\ref{mn1}), (\ref{mn2}), and (\ref{mn3}) in (\ref{mainfor}) we obtain

\begin{equation}
\label{mainforfinal}
  \tilde{J}_{n,\Gamma}(q)
   \doteq_{n} \displaystyle\sum\limits_{i=0}^n\ \sum\limits_{j=0}^nP(n,i,j),
\end{equation}

where
{\footnotesize
\begin{equation*}
P(n,i,j)=
\frac{(-1)^{i+j+ n} q^{\frac{i}{2}+i^2+\frac{j}{2}+j^2-\frac{5 n}{2}}(q;q)_{i+j} (q;q)_n^{15}
(q;q)_{1+i+2n}(q;q)_{1+j+2 n} (q;q)_{1-i+3 n} (q;q)_{1-j+3 n}}{(1-q)(q;q)_i^2
(q;q)_j^2(q;q)_{2 n}^6 (q;q)_{n-i}^3 (q;q)_{i+n} (q;q)_{n-j}^3
(q;q)_{j+n} (q;q)_{1+i+j+n} (q;q)_{1+2 n}^2}\frac{\Delta_{2n}}{\Delta_{n+i}}\frac{\Delta_{2n}}{\Delta_{n+j}}.
\end{equation*}
}

Now
\begin{eqnarray*}
 \frac{(q;q)_n}{(q;q)_{2n}}&=&\frac{\displaystyle\prod_{i=0}^{n-1}(1-q^{i+1})}{\displaystyle\prod_{i=0}^{2n-1}(1-q^{i+1})}\\&=&\frac{1}{\displaystyle\prod_{i=n}^{2n-1}(1-q^{i+1})}\\&=&\displaystyle\prod_{i=0}^{n-1}\frac{1}{(1-q^{i+n+1})}\doteq_n1.
  \end{eqnarray*}
  Similarly,
 \begin{eqnarray*}
 \frac{(q;q)_n}{(q;q)_{2n+1}}\doteq_n1.
   \end{eqnarray*} 
Moreover,
\begin{equation*}
\label{1}
\frac{(q;q)_{3n-i+1}}{(q;q)_{2n+1}}
=1-q^{2n+2}+O(2n+3)=_{n}1.
\end{equation*}
and
\begin{eqnarray*}
 \frac{(q;q)_{2n+i+1}}{(q;q)_{n+i}}&=&\frac{\displaystyle\prod_{k=0}^{3n+i}(1-q^{k+1})}{\displaystyle\prod_{i=0}^{n+i-1}(1-q^{k+1})}\\&=&\displaystyle\prod_{i=n+i}^{3n+i}(1-q^{k+1})\doteq_n1 .
\end{eqnarray*}
Hence,
\begin{equation*}
\displaystyle\sum\limits_{i=0}^n\ \sum\limits_{j=0}^nP(n,i,j)\doteq_{n}\displaystyle\sum\limits_{i=0}^n\ \sum\limits_{j=0}^n\frac{ q^{i+i^2+j+j^2}(q;q)_{i+j} (q;q)_n^{8}
}{(1-q)(q;q)_i^2
(q;q)_j^2 (q;q)_{n-i}^3  (q;q)_{n-j}^3
}. 
\end{equation*}
Further simplification yields:
\begin{eqnarray*}
\displaystyle\sum\limits_{i=0}^n\ \sum\limits_{j=0}^n\frac{ q^{i+i^2+j+j^2}(q;q)_{i+j} (q;q)_n^{8}
}{(1-q)(q;q)_i^2
(q;q)_j^2 (q;q)_{n-i}^3(q;q)_{n-j}^3
}&\doteq_n&(q;q)_{n}^2\displaystyle\sum\limits_{i=0}^n\ \sum\limits_{j=0}^n\frac{ q^{i+i^2+j+j^2}(q;q)_{i+j} 
}{(1-q)(q;q)_i^2
(q;q)_j^2}\\&=&\frac{(q;q)_{n}^2}{(1-q)}\sum\limits_{i=0}^{n}\sum\limits_{k=i}^{n}\frac{q^{k+k^{2}-2i(k-i)}}{%
(q;q)_{k}}\left[ 
\begin{array}{c}
k \\ 
i%
\end{array}%
\right] _{q}^{2}.
\end{eqnarray*}
The definition of the quantum binomial coefficients allows us to write
\begin{eqnarray*}
\frac{(q;q)_{n}^2}{(1-q)}\sum\limits_{i=0}^{n}\sum\limits_{k=i}^{n}\frac{q^{k+k^{2}-2i(k-i)}}{%
(q;q)_{k}}\left[ 
\begin{array}{c}
k \\ 
i%
\end{array}%
\right] _{q}^{2}=\frac{(q;q)_{n}^2}{1-q}\sum\limits_{k=0}^{n}\frac{q^{k+k^{2}}}{%
(q;q)_{k}}(\sum\limits_{i=0}^{k}q^{(-2i(k-i))}\left[ 
\begin{array}{c}
k \\ 
i%
\end{array}%
\right] _{q}^{2}).
\end{eqnarray*}
Hence
\begin{eqnarray*}
T_{\Gamma}(q)\doteq_n\frac{ \tilde{J}_{n,\Gamma}(q)}{\Delta_n(q)}\doteq_n(q;q)_{\infty}^2\sum\limits_{k=0}^{\infty}\frac{q^{k+k^{2}}}{%
(q;q)_{k}}(\sum\limits_{i=0}^{k}q^{(-2i(k-i))}\left[ 
\begin{array}{c}
k \\ 
i%
\end{array}%
\right] _{q}^{2}).
\end{eqnarray*}
\end{proof}
Using Mathematica we computed the first 120 terms of $T_{\Gamma}(q)$:
\\

$ T_{8_5}
   \doteq_{120}
1-2 q+q^2-2 q^4+3 q^5-3 q^8+q^9+4 q^{10}-q^{11}-2 q^{12}-2 q^{13}-3 q^{14}+3 q^{15}+7 q^{16}+2 q^{17}-4 q^{18}-4 q^{19}-4 q^{20}-5 q^{21}+3
q^{22}+9 q^{23}+9 q^{24}-4 q^{26}-9 q^{27}-8 q^{28}-
5 q^{29}-q^{30}+9 q^{31}+13 q^{32}+16 q^{33}+5 q^{34}-10 q^{35}-13 q^{36}-15 q^{37}-12 q^{38}-7 q^{39}+15 q^{41}+25 q^{42}+23 q^{43}+15 q^{44}-3
q^{45}-16 q^{46}-28 q^{47}-31 q^{48}-21 q^{49}-
12 q^{50}+4 q^{51}+16 q^{52}+37 q^{53}+41 q^{54}+39 q^{55}+26 q^{56}-6 q^{57}-34 q^{58}-48 q^{59}-51 q^{60}-49 q^{61}-32 q^{62}-8 q^{63}+20 q^{64}+39
q^{65}+67 q^{66}+76 q^{67}+67 q^{68}+43 q^{69}+
9 q^{70}-36 q^{71}-74 q^{72}-99 q^{73}-101 q^{74}-79 q^{75}-52 q^{76}-7 q^{77}+33 q^{78}+77 q^{79}+108 q^{80}+135 q^{81}+127 q^{82}+104 q^{83}+51
q^{84}-10 q^{85}-82 q^{86}-145 q^{87}-
174 q^{88}-182 q^{89}-160 q^{90}-115 q^{91}-37 q^{92}+37 q^{93}+119 q^{94}+177 q^{95}+218 q^{96}+238 q^{97}+229 q^{98}+171 q^{99}+88 q^{100}-17 q^{101}-126
q^{102}-236 q^{103}-313 q^{104}-
344 q^{105}-325 q^{106}-256 q^{107}-157 q^{108}-28 q^{109}+98 q^{110}+241 q^{111}+343 q^{112}+420 q^{113}+440 q^{114}+424 q^{115}+336 q^{116}+212
q^{117}+41 q^{118}-150 q^{119}-324 q^{120}$.

\end{document}